\documentclass[12pt]{article}

\usepackage[a4paper]{geometry}
\usepackage{amssymb}
\usepackage{amsmath}
\usepackage{amsthm}

\theoremstyle{plain}
\newtheorem{theorem}{Theorem}[section]
\newtheorem{proposition}[theorem]{Proposition}
\newtheorem{lemma}[theorem]{Lemma}
\newtheorem{corollary}[theorem]{Corollary}

\theoremstyle{definition}
\newtheorem{definition}[theorem]{Definition}

\newtheorem{assumption}[theorem]{Assumption}

\theoremstyle{remark}
\newtheorem{remark}[theorem]{Remark}
\newtheorem*{remark*}{Remark}
\newtheorem{remarks}[theorem]{Remarks}
\newtheorem*{remarks*}{Remarks}

\numberwithin{equation}{section}

\title{Resolvent expansion for the Schr\"odinger operator on a graph with infinite rays}
\author{Kenichi {\scshape Ito}\footnote{Graduate School of Mathematical Sciences, University of Tokyo,
3-8-1 Komaba, Meguro-ku, Tokyo 153-8914, Japan.
E-mail: \texttt{ito@ms.u-tokyo.ac.jp}. 
}
\ \& 
Arne {\scshape Jensen}\footnote{Department of Mathematical Sciences,
Aalborg University, Skjernvej 4A, DK-9220 Aalborg \O{}, Denmark.
E-mail: \texttt{matarne@math.aau.dk}. 
}}
\date{}

\begin{document}
\allowdisplaybreaks
\maketitle

\begin{abstract}
We consider the Schr\"odinger operator on 
a combinatorial graph consisting of a finite graph 
and a finite number of discrete half-lines,
all jointed together,
and compute an asymptotic expansion of its resolvent 
around the threshold $0$. 
Precise expressions are obtained for the 
first few coefficients of the expansion in terms of the generalized eigenfunctions. 
This result justifies the classification of threshold types solely
by growth properties of the generalized eigenfunctions. 
By choosing an appropriate free operator a priori possessing 
no zero eigenvalue or zero resonance 
we can simplify the expansion procedure as 
much as that on the single discrete half-line. 
\end{abstract}


\medskip

\tableofcontents

\section{Introduction}

The purpose of this paper is to compute an asymptotic expansion 
of the resolvent around the threshold $0$ for the 
discrete Schr\"odinger operator 
\begin{align}
H=-\Delta_G +V
\label{17072916}
\end{align}
on an infinite, undirected and simple graph.
Here we denote the set of vertices
by $G$, and the set of edges by $E_G$,
hence we are considering the graph $(G,E_G)$.
We sometimes call it simply the graph $G$.
For any function $u\colon G\to \mathbb C$ the Laplacian $-\Delta_G$ is defined as
\begin{align}
(-\Delta_G u)[x]
=\sum_{y\sim x} (u[x]-u[y]),
\label{17072919}
\end{align}
where for any two vertices $x,y \in G$ we say $x\sim y$ if $\{x,y\}\in E_G$.
We assume that the graph $G$ consists of a finite graph $K$
and a finite number of discrete half-lines $L_\alpha$, $\alpha=1,\dots,N$, jointed together.
Special cases are the discrete full line $\mathbb Z$ 
and the discrete half-line $\mathbb N$, considered in \cite{IJ1} and \cite{IJ2}, respectively.
The perturbation $V$ can be a general non-local operator,
which is assumed to decay at infinity in an appropriate sense.

The main results in this paper give a complete description 
in terms of growth properties of the generalized eigenfunctions
for the first few coefficients of the resolvent expansion:
\begin{align*}
(H+\kappa^2)^{-1}
=\kappa^{-2}G_{-2}+\kappa^{-1}G_{-1}+G_0+\kappa G_1+\cdots.
\end{align*}
More precisely, we prove that $G_{-2}$ is the \textit{bound projection} or 
the projection onto the bound eigenspace, 
and that $G_{-1}$ is the \textit{resonance projection} or 
the projection onto the resonance eigenspace. 
Explicit expressions for $G_0$ and $G_1$ are also computed.
It is well known that the coefficients $G_{-2}$ and $G_{-1}$ directly affect the 
local decay rate of the Schr\"odinger propagator 
$e^{-itH}$ as $t\to\pm\infty$, see \cite{JK}. 
Hence our results reveal 
a relationship between the growth rates of the generalized eigenfunctions in space
and the local decay rate of $e^{-itH}$ in time,
justifying a classification of threshold types
solely by the structure of the generalized eigenspace.

There exists a large literature on threshold resolvent expansions for Schr\"{o}dinger operators. However, a complete analysis taking into account all possible generalized threshold  eigenfunctions has been obtained only recently. The first one is in \cite{IJ1} on the discrete full line $\mathbb Z$ and more recently on the discrete half-line $\mathbb N$ in \cite{IJ2}.
In these papers the authors implement the expansion scheme of \cite{JN,JNerr} in its full generality.

This paper is a  generalization of \cite{IJ1,IJ2} 
to a graph with infinite rays. 
The strategy is essentially the same as before.
However, in this paper, based on ideas from \cite{CJN} and \cite{IJ2}, 
we set up a free operator a priori possessing no zero eigenvalue or zero resonance,
and this effectively simplifies the expansion procedure to the one in \cite{IJ2}. 
Actually, with simpler arguments, we obtain 
a more precise description for the coefficients than in \cite{IJ1}.

There is a large literature on spectral theory for graphs, both combinatorial graphs and metric graphs. Most of the literature on spectral theory for combinatorial graphs focuses on finite graphs, see for example~\cite{CdV,BH}. Combinatorial graphs with infinite rays have been studied in \cite{Go,LN}. Here one-dimensional methods for discrete Schr\"{o}dinger operators can be applied. For analogous metric graphs, i.e. differential operators on the edges with appropriate boundary conditions at the vertices, the direct and inverse spectral problems have been studied by many authors. Early work includes \cite{Ge,GP}.  See \cite{ES} for a physically motivated introduction. There are extensive references in the recent monograph~\cite{BK} on metric graphs, also called quantum graphs.

The first author is directly inspired by the papers \cite{MT,FI} 
to work on the structures considered here. See the references therein for  recent development in
this topic.

This paper is organized as follows: 
In Section~\ref{170805} we introduce and study a free operator. 
Since the free operator chosen here does not have a zero eigenvalue or a zero resonance, 
its resolvent expansion is free of singular parts. 
In Section~\ref{1707291915} we formulate threshold types for a perturbed operator 
 using only the structure of the generalized eigenspace,
and the main results of the paper 
are presented. 
Section~\ref{12.12.19.2.5} is devoted to a detailed analysis of generalized eigenfunctions. 
In particular, characterizations of eigenfunctions are given as well as their precise asymptotics.
We will prove the main theorems in Section~\ref{1608216},
separating cases according to threshold types.
In Appendix~\ref{17080222} we exhibit examples of particular interest. 
We compute $G_0$ and $G_1$ in Appendix~\ref{170818}. 
The results in Appendix~\ref{170818} are to be considered as  part of the main results of the paper,
however, their proofs are extremely long, hence we separate them from Section~\ref{1707291915}.
Appendix~\ref{17081223} presents a modified version of \cite[Lemma~4.6]{IJ1},
which is needed in Section~\ref{1608216} but does not fit the context there.

\section{The free operator}\label{170805}

In this section we define a graph $(G,E_G)$ with a finite number of 
infinite rays, and fix a free operator $H_0$ on it. 
We actually have a freedom of choice for a free operator,
as long as its has a simple structure, see Remarks~\ref{17090412} and 
\ref{17090413}.
The free operator of the paper is chosen to be as close as 
possible 
to the graph Laplacian $-\Delta_G$ given by \eqref{17072919}.
We also provide its resolvent expansion explicitly.
The result of this section can be seen as a preliminary to,
or a prototype of, more general main results of the paper presented in Section~\ref{1707291915}.

\subsection{Graph with rays}\label{17081312}
Let $(K,E_0)$ be a connected, finite, undirected and simple graph,
without loops or multiple edges,
and let $(L_\alpha,E_\alpha)$, $\alpha=1,\dots,N$, be $N$ copies of the discrete half-line,
i.e.
$$L_\alpha=\mathbb N=\{1,2,\ldots\},\quad E_\alpha=\bigl\{\{n,n+1\};\ n\in L_\alpha\bigr\}.$$
We construct the graph $(G,E_G)$ by jointing $(L_\alpha,E_\alpha)$ to $(K,E_0)$ at a vertex $x_\alpha\in K$
for $\alpha=1,\dots,N$:
\begin{align}
\begin{split} 
G&=K\cup L_1\cup\dots\cup L_N,\\
E_G&=E_0\cup E_1\cup\dots E_N\cup \bigl\{\{x_1,1^{(1)}\},\dots,\{x_N,1^{(N)}\}\bigr\}. 
\end{split}
\label{17081313}
\end{align} 
Here we distinguished $1$ of $L_\alpha$ by a superscript: $1^{(\alpha)}\in L_\alpha$.
Note that two different half-lines $(L_\alpha,E_\alpha)$ and $(L_\beta,E_\beta)$, $\alpha\neq\beta$, 
could be jointed to the same vertex $x_\alpha=x_\beta\in K$.

Let $h_0$ be the free \textit{Dirichlet} Schr\"odinger operators on $K$:
For any function $u\colon K\to\mathbb C$ 
we define 
\begin{align*}
(h_0 u)[x]
=\sum_{y\sim x} (u[x]-u[y])+\sum_{\alpha=1}^Ns_\alpha[x]u[x]
\quad\text{for }x\in K
,
\end{align*}
where $s_\alpha[x]=1$ if $x=x_\alpha$ and $s_\alpha[x]=0$ otherwise.
Note that the Dirichlet boundary condition is considered being set on the \textit{boundaries} $1^{(\alpha)}\in L_\alpha$ 
outside $K$.
Similarly, for $\alpha=1,\dots,N$
let $h_\alpha$ be the free Dirichlet Schr\"odinger operators on $L_\alpha$:
For any function $u\colon L_\alpha\to \mathbb C$
we define
\begin{align*}
(h_\alpha u)[n]
=\begin{cases}
2u[1]-u[2]& \text{for }n=1,\\
2u[n]-u[n+1]-u[n-1] & \text{for }n\geq2.
\end{cases}
\end{align*}

Then we define the free operator $H_0$ on $G$ as a direct sum
\begin{align}
H_0=h_0\oplus h_1\oplus\dots\oplus h_N,
\label{17072920}
\end{align}
according to a direct sum decomposition
\begin{align}
F(G)=F(K)\oplus F(L_1)\oplus\cdots\oplus F(L_N),
\label{17080223}
\end{align}
where $F(X)=\{u\colon X\to\mathbb C\}$ denotes the set of all the functions on a space 
$X$.
In the definition \eqref{17072920} 
interactions between $K$ and $L_\alpha$ are not considered,
and the free operator $H_0$ does not coincide with $-\Delta_G$ defined by \eqref{17072919}.
In fact, we can write 
\begin{align}
-\Delta_G =H_0+J,\quad
J=-\sum_{\alpha=1}^N\Bigl( |s_\alpha\rangle\langle f_\alpha|
+|f_\alpha\rangle\langle s_\alpha|\Bigr),
\label{170729}
\end{align}
where $f_\alpha[x]=1$ if $x=1^{(\alpha)}$ and $f_\alpha[x]=0$ otherwise.
However, the operator $H_0$ is simpler and more useful than $-\Delta_G$,
since it does not have a zero eigenvalue or a zero resonance,
and the asymptotic expansion of its resolvent around $0$ 
does not have a singular part,
which we will verify soon below. 
The interaction $J$ can be treated as a special case of 
general perturbations considered in this paper, see Assumption~\ref{assumV}
and Appendix~\ref{17080222}.
Hence this paper covers the Laplacian $-\Delta_G$ on $G$ as a perturbation of 
the free operator $H_0$.

\begin{remark}\label{17090412} 
We will not use an explicit expression for the free operator $h_0$ on $K$,
and we may change it as long as it does not have a zero eigenvalue.
Another simple choice could be, for example, 
\begin{align}
h_0=2\mathop{\mathrm{id}}\nolimits_{F(K)}.
\label{1709041306}
\end{align}
If we adopt \eqref{1709041306} as a free operator,
 the difference $J=-\Delta_G-H_0$ will be different 
from \eqref{170729}.
However, we can still treat it as a special case of 
perturbations considered in Assumption~\ref{assumV},
see also Proposition~\ref{170904}.
\end{remark}

\subsection{Expansion of the free resolvent}

The restriction of $H_0$ to the Hilbert space 
$\mathcal H=\ell^2(G)$, denoted by $H_0$ again, 
is obviously bounded, self-adjoint and non-negative.
The free resolvent is
\begin{align*}
R_0(\kappa)=(H_0+\kappa^2)^{-1}\in\mathcal B(\mathcal H)
\quad 
\text{for }\mathop{\mathrm{Re}}\kappa>0.
\end{align*}
According to the decomposition
$$\mathcal H=\mathfrak h_0\oplus \mathfrak h_1\oplus \cdots\oplus \mathfrak h_N;
\quad 
\mathfrak h_0=\ell^2(K),\quad \mathfrak h_\alpha=\ell^2(L_\alpha)\text{ for }\alpha=1,\dots,N,$$
 the free resolvent $R_0(\kappa)$ will split into a direct sum
\begin{align}
R_0(\kappa)=r_0(\kappa)\oplus r_1(\kappa)\oplus\cdots\oplus r_N(\kappa),
\label{17080220}
\end{align}
where $r_\alpha(\kappa)=(h_\alpha+\kappa^2)^{-1}\in\mathcal B(\mathfrak h_\alpha)$ for $\alpha=0,1,\dots,N$.
Hence the expansion of $R_0(\kappa)$ reduces to those of $r_\alpha(\kappa)$.

We first let $\alpha=0$. 
Since $K$ is a finite set, 
we have $\mathfrak h_0\cong \mathbb C^k$ with $k=\# K$,
and the operator $h_0$ has 
a $k\times k$ matrix representation.
By  standard arguments from linear algebra we can deduce the following result.

\begin{proposition}\label{17080221}
There exist positive real numbers
$0<\lambda_1\le \lambda_2\le\dots\le \lambda_k$, $k=\# K$,
such that 
\begin{align}
\sigma_{\mathrm{pp}}(h_0)=\{\lambda_1,\dots,\lambda_k\},\quad
\sigma_{\mathrm{ac}}(h_0)=\sigma_{\mathrm{sc}}(h_0)=\emptyset.
\label{1707395}
\end{align}
In particular, $h_0$ is invertible,
and, for $|\kappa|< \lambda_1^{1/2}$, the resolvent $r_0(\kappa)$ has the Taylor expansion:
\begin{align*}
r_0(\kappa)=\sum_{j=0}^\infty \kappa^{2j}g_{0,2j}
\text{ in }\mathcal B(\mathfrak h_0);\quad 
g_{0,2j}=(-1)^jh_0^{-j-1}\in\mathcal B(\mathfrak h_0).
\end{align*}
\end{proposition}
\begin{proof}
Since $k=\# K=\dim\mathfrak h_0$ is finite and the operator $h_0$ is self-adjoint,
there exist real eigenvalues $\lambda_1\le \cdots\le \lambda_k$ 
such that \eqref{1707395} holds.
It is well-known that 
for any $u\in \mathfrak h_0$ we can write 
\begin{align*}
\langle u,h_0u\rangle
=\sum_{\{x,y\}\in E_0}\bigl|u[x]-u[y]\bigr|^2+\sum_{\alpha=1}^N \bigl|u[x_\alpha]\bigr|^2.
\end{align*}
Then we have $\lambda_1\ge 0$,
but, moreover, we can deduce $\lambda_1>0$.
In fact, if we assume $\lambda_1=0$, then the corresponding eigenfunction $u_1\in\mathfrak h_0$ 
has to satisfy
\begin{align*}
\sum_{\{x,y\}\in E_0}\bigl|u_1[x]-u_1[y]\bigr|^2+\sum_{\alpha=1}^N \bigl|u_1[x_\alpha]\bigr|^2
=0.
\end{align*}
This and the connectedness of $K$ imply that $u_1$ is identically $0$ on $K$,
which is a contradiction. Hence we have $\lambda_1>0$,
and in particular $h_0$ is invertible.
The asserted Taylor expansion follows by the Neumann series.
\end{proof}

\begin{remark}\label{17090413}
Even if we adopt \eqref{1709041306} as a free operator,
Proposition~\ref{17080221} is still valid. 
In this paper 
we will use only the properties of $h_0$ listed in Proposition~\ref{17080221},
and an explicit expression is not needed.
It is clear that Proposition~\ref{17080221} 
holds true for any (positive) operator $h_0$ without zero eigenvalue.
\end{remark}

Next we consider the case $\alpha=1,\dots,N$.
The operators $h_\alpha$ are identical with 
the free Laplacian on the discrete half-line
investigate in \cite{IJ2}.
Hence we can just quote the results from there. 
Let us set for $s\in \mathbb R$ 
\begin{align}
\begin{split}
\mathfrak l_\alpha^{s}
&=\ell^{1,s}(L_\alpha)
=
\bigl\{x\colon L_\alpha\to \mathbb{C};\ 
\sum_{n\in L_\alpha}(1+n^2)^{s/2}|x[n]|<\infty\bigr\},\\
(\mathfrak l_\alpha^s)^*
&=\ell^{\infty,-s}(L_\alpha)
=\bigl\{x\colon L_\alpha\to \mathbb{C};\ 
\sup_{n\in L_\alpha}(1+n^2)^{-s/2}|x[n]|<\infty\bigr\},
\end{split}
\label{17073017}
\end{align}
and 
\begin{align*}
\mathfrak b^s_\alpha=\mathcal B\bigl(\mathfrak l_\alpha^{s},(\mathfrak l_\alpha^s)^*\bigr).
\end{align*}
We also set 
$$n\wedge m=\min\{n,m\},\quad n\vee m=\max\{n,m\}.$$

\begin{proposition}\label{prop12}
Let $\alpha\in \{1,\dots,N\}$. Then the operator $h_\alpha$ has the spectrum:
\begin{align*}
\sigma_{\mathrm{ac}}(h_\alpha)=[0,4],\quad 
\sigma_{\mathrm{pp}}(h_\alpha)=\sigma_{\mathrm{sc}}(h_\alpha)=\emptyset.
\end{align*}
In addition, for any integer $\nu \ge 0$,
as $\kappa\to 0$ with $\mathop{\mathrm{Re}}\kappa>0$,
the resolvent $r_\alpha(\kappa)$ has an asymptotic expansion:
\begin{align*}
r_\alpha(\kappa)=\sum_{j=0}^\nu \kappa^jg_{\alpha,j}+{\mathcal O}(\kappa^{\nu +1})
\quad \text{in }\mathfrak b^{\nu +2}_\alpha
\end{align*}
with $g_{\alpha,j}\in\mathfrak b^{j+1}_\alpha$ for $j$ even,
and $g_{\alpha,j}\in\mathfrak b^{j}_\alpha$ for $j$ odd, satisfying
\begin{equation}
\begin{split}
h_\alpha g_{\alpha,0}&=g_{\alpha,0}h_\alpha=I,\\ 
h_\alpha g_{\alpha,1}&=g_{\alpha,1}h_\alpha=0,\\ 
h_\alpha g_{\alpha,j}&=g_{\alpha,j}h_\alpha=-g_{\alpha,j-2}\quad \text{for }j\ge 2.
\end{split}
\label{18082023}
\end{equation}
The coefficients $g_{\alpha,j}$ have explicit kernels,
and the first few are given by 
\begin{align}
g_{\alpha,0}[n,m]&=n\wedge m,
\label{G00}\\
g_{\alpha,1}[n,m]&=-n\cdot m,\label{G01}\\
\begin{split}
g_{\alpha,2}[n,m]
&=-\tfrac{1}{6}(n\wedge m)
+\tfrac{1}{6}(n\wedge m)^3+\tfrac{1}{2}n\cdot m\cdot(n\vee m),
\end{split}
\label{G02}\\
\begin{split}
g_{\alpha,3}[n,m]
&=\tfrac{5}{24}n\cdot m-\tfrac{1}{6}n^3\cdot m-\tfrac{1}{6}n\cdot m^3.
\end{split}
\label{G03}
\end{align}
\end{proposition}
\begin{proof}
See \cite[Section~I\hspace{-.1em}I]{IJ2}. 
\end{proof}

Now, due to \eqref{17080220} and Propositions~\ref{17080221} and \ref{prop12},
we can expand the free resolvent $R_0(\kappa)$.
In addition to \eqref{17073017}, let us set $\mathfrak l_0=\ell^1(K)$,
$(\mathfrak l_0)^*=\ell^\infty(K)$, 
and write for $s\in \mathbb R$
\begin{align*}
\mathcal L^s&=\ell^{1,s}(G)=\mathfrak l_0\oplus \mathfrak l_1^s\oplus\cdots\oplus\mathfrak l_N^s
,\\
(\mathcal L^s)^*
&=\ell^{\infty,-s}(G)
=(\mathfrak l_0)^*\oplus (\mathfrak l_1^s)^*\oplus\cdots\oplus(\mathfrak l_N^s)^*,
\end{align*}
and 
\begin{align*}
\mathcal B^s=\mathcal B\bigl(\mathcal L^s,(\mathcal L^s)^*\bigr)
.
\end{align*}

\begin{corollary}
The operator $H_0$ has the spectrum: 
\begin{align*}
\sigma_{\mathrm{ac}}(H_0)=[0,4],\quad 
\sigma_{\mathrm{sc}}(H_0)=\emptyset,\quad
\sigma_{\mathrm{pp}}(H_0)=\{\lambda_1,\dots,\lambda_k\}.
\end{align*}
In addition, for any integer $\nu \ge 0$,
as $\kappa\to 0$ with $\mathop{\mathrm{Re}}\kappa>0$,
the resolvent $R_0(\kappa)$ has an asymptotic expansion:
\begin{align}
R_0(\kappa)=\sum_{j=0}^\nu \kappa^jG_{0,j}+{\mathcal O}(\kappa^{\nu +1})
\quad \text{in }\mathcal B^{\nu +2}
\label{free-expan}
\end{align}
with $G_{0,j}\in\mathcal B^{j+1}$ for $j$ even,
and $G_{0,j}\in\mathcal B^{j}$ for $j$ odd, given by 
\begin{align}
G_{0,j}
=\left\{
\begin{array}{ll}
g_{0,j}\oplus g_{1,j}\oplus\cdots\oplus g_{N,j}&\text{for }j\text{ even},\\
0\oplus g_{1,j}\oplus\cdots\oplus g_{N,j}& \text{for }j\text{ odd}.
\end{array}
\right.
\label{170803}
\end{align}
The coefficients satisfy
\begin{equation}
\begin{split}
H_0 G_{0,0}&=G_{0,0}H_0=I,\\ 
H_0 G_{0,1}&=G_{0,1}H_0=0,\\ 
H_0 G_{0,j}&=G_{0,j}H_0=-G_{0,j-2}\quad \text{for }j\ge 2.
\end{split}
\label{18082023a}
\end{equation}
\end{corollary}
\begin{proof}
The assertions are obvious by \eqref{17072920},
\eqref{17080220} and Propositions~\ref{17080221} and \ref{prop12}.
\end{proof}

\section{The perturbed operator}\label{1707291915}

In this section we introduce our class of perturbations,
formulate threshold types in terms of the generalized eigenfunctions,
and then state the main theorems of the paper.

\subsection{Generalized eigenspace and threshold types}

We consider the following class of perturbations, cf.~\cite{JN,IJ1,IJ2}.

\begin{assumption}\label{assumV}
Assume that $V\in {\mathcal B}(\mathcal H)$ is self-adjoint,
and that there exist 
an injective operator $v\in \mathcal B({\mathcal K},\mathcal L^\beta)$
with $\beta\ge 1$ 
and a self-adjoint unitary operator
$U\in \mathcal B({\mathcal K})$,
both defined on some Hilbert space ${\mathcal K}$, such that 
\begin{equation*}
V=vUv^*\in \mathcal B\bigl((\mathcal L^{\beta})^*,\mathcal L^\beta\bigr).
\end{equation*}
\end{assumption}
We note that the conditions in the assumption imply that $V$ is compact on $\mathcal{H}$.
A self-adjoint operator $V\in\mathcal B(\mathcal H)$ satisfies Assumption~\ref{assumV},
if it extends to a bounded operator 
$\ell^{2,-\beta-1/2-\epsilon}(G)\to\ell^{2,\beta+1/2+\epsilon}(G)$
for some $\epsilon>0$, where $\ell^{2,s}(G)$ denotes the standard weighted $\ell^2$-space on $G$,
see Proposition~\ref{170904}.
For other examples of perturbations satisfying Assumption~\ref{assumV} 
we refer to \cite[Appendix~B]{IJ1}. See also Appendix~\ref{17081323}.

Under Assumption~\ref{assumV} we let 
$$
H=H_0+V,
$$
and for $-\kappa^2\notin \sigma(H)$
$$R(\kappa)=(H+\kappa^2)^{-1}.$$
As we can see in Appendix~\ref{17080222}, 
the interaction operator $J$ from \eqref{170729} satisfies Assumption~\ref{assumV},
hence the above $H$ includes $-\Delta_G$ as a special case.
We know a priori  that $H_0$ does not have a zero eigenvalue or a zero resonance,
and its resolvent has a simpler expansion than that of $-\Delta_G$.
This fact effectively simplifies the expansion procedure for the 
perturbed resolvent $R(\kappa)$, and enables us to obtain more precise 
expressions for the coefficients than those in \cite{IJ1}.

Let us consider the solutions to the zero eigen-equation 
$H\Psi=0$ in the largest space where it can be defined.  
Define the \textit{generalized zero eigenspace} $\widetilde{\mathcal E}$ as
\begin{align*}
\widetilde{\mathcal{E}}&=\bigl\{\Psi\in(\mathcal{L}^{\beta})^*;\  H\Psi=0\bigr\}.
\end{align*}
We will analyze this space in detail in Section~\ref{12.12.19.2.5}.
Here let us just quote some of the remarkable consequences shown there. 
Let 
$\mathbf n^{(\alpha)}\in(\mathcal L^1)^*$ and $\mathbf 1^{(\alpha)}\in (\mathcal L^0)^*$ 
be the functions defined as
\begin{align*}
\mathbf{n}^{(\alpha)}[x]
&=\left\{
\begin{array}{ll}
m& \text{for }x=m\in L_\alpha,\\
0& \text{for }x\in G\setminus L_\alpha,
\end{array}
\right.
\intertext{and} 
\mathbf{1}^{(\alpha)}[x]
&=\left\{
\begin{array}{ll}
1& \text{for }x\in L_\alpha,\\
0& \text{for }x\in G\setminus L_\alpha,
\end{array}
\right.
\end{align*}
respectively. We abbreviate the spaces spanned by these functions as 
\begin{align*}
\mathbb C\mathbf n=\mathbb{C}\mathbf{n}^{(1)}\oplus\cdots\oplus\mathbb{C}\mathbf{n}^{(N)}
,\quad
\mathbb C\mathbf 1=\mathbb{C}\mathbf{1}^{(1)}\oplus\cdots\oplus\mathbb{C}\mathbf{1}^{(N)}.
\end{align*}
Note that 
$\mathbf n^{(\alpha)}$ and $\mathbf 1^{(\alpha)}$
are defined not only on $L_\alpha$ but actually on the whole $G$ by the 
zero-extension.
Hence the above direct sums mean that summands are linearly independent,
which are understood slightly differently from those so far 
based on the decomposition \eqref{17080223}.
In Proposition~\ref{13.1.16.2.51}
we will see that under Assumption~\ref{assumV} with $\beta\ge 1$
the generalized eigenfunctions have specific asymptotics:
\begin{align}
\widetilde{\mathcal{E}}\subset
\mathbb{C}\mathbf{n}
\oplus
\mathbb{C}\mathbf{1}\oplus\mathcal{L}^{\beta-2},
\label{13.3.7.13.48}
\end{align}
cf.\ \eqref{12.12.19.6.56} and \eqref{170804}.
With this asymptotics it makes sense to consider the following subspaces:
\begin{align}
\mathcal{E}=\widetilde{\mathcal{E}}\cap(\mathbb{C}\mathbf{1}\oplus\mathcal{L}^{\beta-2}),
\quad 
\mathsf{E}=\widetilde{\mathcal{E}}\cap\mathcal{L}^{\beta-2}.\label{16071617}
\end{align}
To be precise, a function in $\widetilde{\mathcal E}\setminus \mathcal E$ should be called a 
\textit{non-resonance eigenfunction},
one in $\mathcal E\setminus\mathsf E$ a \textit{resonance eigenfunction},
and one in $\mathsf E$ a \textit{bound eigenfunction},
but we shall often call them \textit{generalized eigenfunctions} or simply \textit{eigenfunctions}.

Let us introduce the same classification of threshold  
as in~\cite[Definition 1.6]{IJ1}.

\begin{definition}\label{def-reg-excp}
The threshold $z=0$ is said to be 
\begin{enumerate}
\item 
a \textit{regular point}, 
if $\mathcal{E}= \mathsf{E}= \{0\}$;
\item
an \textit{exceptional point of the first kind}, 
if $\mathcal{E}\supsetneq \mathsf{E}= \{0\}$;
\item
an \textit{exceptional point of the second kind}, 
if 
$\mathcal{E}=\mathsf{E}\supsetneq \{0\}$;
\item
an \textit{exceptional point of the third kind}, 
if 
$\mathcal{E}\supsetneq \mathsf{E}\supsetneq \{0\}$.
\end{enumerate} 
\end{definition}

In Proposition~\ref{13.1.16.2.51}
we will have more detailed expressions for the generalized eigenfunctions.
It should be noted here that there is a dimensional relation:
\begin{align}
\dim(\widetilde{\mathcal E}/\mathcal E)+\dim(\mathcal E/\mathsf E)=N,\quad 
0\le \dim\mathsf E< \infty.
\label{170802}
\end{align}
The former identity reflects a certain topological stability of the non-decaying eigenspace
under small perturbations. In addition, we will see in 
Theorems~\ref{17082015} and \ref{17082017} in Appendix~\ref{170818} that 
for any $\Psi_1\in\widetilde{\mathcal E}$
and $\Psi_2\in\mathcal E$, if we let 
\begin{align*}
\Psi_1-\sum_{\alpha=1}^N c_\alpha^{(1)}\mathbf n^{(\alpha)}
\in \mathbb C\mathbf 1\oplus\mathcal L^{\beta-2}
,\quad
\Psi_2-\sum_{\alpha=1}^Nc_\alpha^{(2)}\mathbf 1^{(\alpha)}\in \mathcal L^{\beta-2},
\end{align*}
then 
\begin{align}
\sum_{\alpha=1}^N\overline c_\alpha^{(2)} c_\alpha^{(1)}=0.
\label{1708221320}
\end{align}
Hence it would be natural to consider \textit{orthogonality} in $\widetilde{\mathcal E}$
in terms of the asymptotics, in addition to the one with respect to the standard $\ell^2$-inner product.
Let us define the \textit{generalized orthogonal projections} onto the eigenspaces. We use $\langle\cdot,\cdot\rangle$ to denote the duality between 
$\mathcal{L}^s$ and $(\mathcal{L}^s)^{\ast}$. If $\beta\geq2$ then $\langle \Phi,\Psi\rangle$ is defined for $\Phi\in\mathsf{E}$ and $\Psi\in\mathcal{E}$. If we only assume $\beta\geq1$ then we must assume $\Phi\cdot\Psi\in\mathcal{L}^0$ to justify the notation $\langle \Phi,\Psi\rangle$.
Here $(\Phi\cdot\Psi)[n]=\Psi[n]\Phi[n]$, $n\in G$, is the elementwise product of two sequences.

\begin{definition}\label{170819}
We call a subset $\{\Psi_\gamma\}_{\gamma}\subset \mathcal E$ 
a \textit{resonance basis}, if 
the set $\{[\Psi_\gamma]\}_{\gamma}$ of representatives 
forms a basis in $\mathcal E/\mathsf E$. 
It is said to be \textit{orthonormal}, if 
\begin{enumerate}
\item for any $\gamma$ and $\Psi\in\mathsf E$ one has $\overline{\Psi}\cdot\Psi_\gamma\in \mathcal L^0$ and%
 $\langle\Psi,\Psi_\gamma\rangle=0$;
\item there exists an orthonormal system 
$\{c^{(\gamma)}\}_\gamma\subset \mathbb C^N$
such that for any $\gamma$
$$\Psi_\gamma
-\sum_{\alpha=1}^Nc_\alpha^{(\gamma)}\mathbf 1^{(\alpha)}\in \mathcal L^{\beta-2}.
$$
\end{enumerate}
The \textit{orthogonal resonance projection} $\mathcal P$ is defined as 
\begin{align*}
\mathcal P=\sum_{\gamma}|\Psi_\gamma\rangle\langle\Psi_\gamma|.
\end{align*}
\end{definition}

\begin{definition}
We call a basis $\{\Psi_\gamma\}_{\gamma}\subset \mathsf E$ 
a \textit{bound basis} to distinguish it from a resonance basis. 
It is said to be \textit{orthonormal}, if 
for any $\gamma$ and $\gamma'$ one has 
$\overline{\Psi}_{\gamma'}\cdot\Psi_\gamma\in \mathcal L^0$ and
$$\langle\Psi_{\gamma'},\Psi_\gamma\rangle=\delta_{\gamma\gamma'}.$$
The \textit{orthogonal bound projection} $\mathsf P$ is defined as 
\begin{align*}
\mathsf P=\sum_{\gamma}|\Psi_\gamma\rangle\langle\Psi_\gamma|.
\end{align*}
\end{definition}

We remark that the above orthogonal projections 
$\mathcal P$ and $\mathsf P$ are 
independent of choice of orthonormal bases.

\subsection{Main results: Expansion of perturbed resolvent}
We now present the resolvent expansion for each threshold type given 
in Definition~\ref{def-reg-excp}. We have to impose different assumptions 
on the parameter $\beta$ depending on threshold types. 
For simplicity we state the results only for integer values of $\beta$,
but an extension to general $\beta$ is straightforward.

\begin{theorem}\label{thm-reg}
Assume that the threshold $0$ is a \emph{regular} point,
and that Assumption~\ref{assumV} is fulfilled for some integer $\beta\ge 2$.
Then
\begin{equation}
R(\kappa)=\sum_{j=0}^{\beta-2}\kappa^jG_j+\mathcal{O}(\kappa^{\beta-1})
\quad
\text{in }\mathcal{B}^{\beta-2}
\label{17081821}
\end{equation}
with 
$G_j\in\mathcal B^{j+1}$ for $j$ even,
and $G_j\in\mathcal B^{j}$ for $j$ odd.
The coefficients $G_j$ can be computed explicitly. In particular,
\begin{align}
G_{-2}=\mathsf P=0,\quad G_{-1}=\mathcal P=0.
\label{17081822}
\end{align}
\end{theorem}

\begin{theorem}\label{thm-ex1}
Assume that the threshold $0$ is an exceptional point of the \emph{first kind},
and that Assumption~\ref{assumV} is fulfilled for some integer $\beta\ge 3$.
Then
\begin{equation}\label{expand-first}
R(\kappa)=\sum_{j=-1}^{\beta-4}\kappa^jG_j+\mathcal{O}(\kappa^{\beta-3})\quad
\text{in }\mathcal{B}^{\beta-1}
\end{equation}
with 
$G_j\in\mathcal B^{j+3}$ for $j$ even,
and $G_j\in\mathcal B^{j+2}$ for $j$ odd.
The coefficients $G_j$ can be computed explicitly. In particular,
\begin{align}
G_{-2}=\mathsf P=0,\quad 
G_{-1}=\mathcal P\neq 0.
\label{1608229}
\end{align}
\end{theorem}

\begin{theorem}\label{thm-ex2}
Assume that the threshold $0$ is an exceptional point of the \emph{second kind},
and that Assumption~\ref{assumV} is fulfilled for some integer $\beta\ge 4$.
Then
\begin{equation}\label{expand-second}
R(\kappa)=\sum_{j=-2}^{\beta-6}\kappa^jG_j+\mathcal{O}(\kappa^{\beta-5})
\quad\text{in }
\mathcal B^{\beta-2}
\end{equation}
with $G_j\in\mathcal B^{j+3}$ for $j$ even,
and $G_j\in\mathcal B^{j+2}$ for $j$ odd.
The coefficients $G_j$ can be computed explicitly. In particular,
\begin{align}\label{ex2-G-2}
G_{-2}&=\mathsf P\neq 0,
\quad
G_{-1}=\mathcal P=0.
\end{align}
\end{theorem}

\begin{theorem}\label{thm-ex3}
Assume that the threshold $0$ is an exceptional point of the \emph{third
kind},
and that Assumption~\ref{assumV} is fulfilled for some integer $\beta\ge 4$.
Then
\begin{equation}\label{expand-third}
R(\kappa)=\sum_{j=-2}^{\beta-6}\kappa^jG_j+\mathcal{O}(\kappa^{\beta-5})
\quad
\text{in }\mathcal B^{\beta-2}
\end{equation}
with $G_j\in\mathcal B^{j+3}$ for $j$ even,
and $G_j\in\mathcal B^{j+2}$ for $j$ odd.
The coefficients $G_j$ can be computed explicitly. In particular,
\begin{align}
G_{-2}=\mathsf P\neq 0,
\quad
G_{-1}=\mathcal P\neq 0
.\label{G-1}
\end{align}
\end{theorem}

Theorems~\ref{thm-reg}--\ref{thm-ex3}
justify the classification of threshold types only by the 
growth properties of eigenfunctions:

\begin{corollary}
The threshold type determines and is determined 
by the coefficients $G_{-2}$ and $G_{-1}$ from Theorems~\ref{thm-reg}--\ref{thm-ex3}.
\end{corollary}
\begin{proof}
The assertion is obvious by Theorems~\ref{thm-reg}--\ref{thm-ex3}.
\end{proof}

The higher order coefficients $G_0,G_1,\ldots$ are much more complicated.
We postpone the expressions for $G_0$ and $G_1$ to Appendix~\ref{170818},
since their proofs are also very long.
Here let us only note some well-known relations 
that once the existence of an expansion is guaranteed, follow easily from the general theory.

\begin{corollary}\label{17081919}
The coefficients $G_j$ from Theorems~\ref{thm-reg}--\ref{thm-ex3} satisfy
\begin{align*}
HG_j&=G_jH=0\quad\text{for }j=-2,-1,\\
HG_0&=G_0H=I-\mathsf P,\\
HG_j&=G_jH=-G_{j-2}\quad\text{for }j\ge 1.
\end{align*}
\end{corollary}
\begin{proof}
See \cite[Corollary~I\hspace{-.1em}I\hspace{-.1em}I.8]{IJ2}.
\end{proof}

We shall prove Theorems~\ref{thm-reg}--\ref{thm-ex3}
in Sections~\ref{12.12.19.2.5} and \ref{1608216}, 
following the arguments of \cite{IJ1,IJ2}.

\section{Analysis of eigenfunctions}\label{12.12.19.2.5}

In this section we investigate the eigenspaces 
$\widetilde{\mathcal E}$, $\mathcal E$ and $\mathsf E$ in detail.
The main results of this section are stated in Proposition~\ref{13.1.16.2.51} and 
Corollary~\ref{160616438}. 
In Proposition~\ref{13.1.16.2.51} we provide $\mathcal K$-space characterizations
of the eigenspaces, and also precise asymptotics of eigenfunctions.
Corollary~\ref{160616438} rephrases a part of Proposition~\ref{13.1.16.2.51}
in terms of the \textit{intermediate operators} in the terminology of \cite{IJ1},
which more directly connects the threshold types to coefficients of the resolvent expansion.
In  later sections these results will be employed as effective tools 
in both implementing and analyzing the resolvent expansions.

\subsection{$\mathcal K$-space characterization and asymptotics}
Let us inductively define $\Psi^{(\alpha)}\in (\mathcal L^1)^*$, $\alpha=1,\dots,N$,
as $\widetilde\Psi^{(1)}=\mathbf n^{(1)}$ and 
\begin{align}
\Psi^{(\alpha)}=\|v^*\widetilde \Psi^{(\alpha)}\|^\dagger\widetilde\Psi^{(\alpha)}
,\quad
\widetilde\Psi^{(\alpha)}=\mathbf n^{(\alpha)}
-\sum_{\gamma=1}^{\alpha-1}
\bigl\langle v^*\Psi^{(\gamma)},v^*\mathbf n^{(\alpha)} \bigr\rangle \Psi^{(\gamma)},
\label{170807}
\end{align}
where $a^{\dagger}$ denotes 
the \textit{pseudo-inverse} of $a\in\mathbb C$, see e.g.\ \cite[Appendix]{IJ2},
and set 
\begin{align}
P=\sum_{\alpha=1}^N|\Phi^{(\alpha)}\rangle\langle \Phi^{(\alpha)}|,\quad 
\Phi^{(\alpha)}=v^*\Psi^{(\alpha)}\in\mathcal K.
\label{17080719}
\end{align}
Obviously $P\in\mathcal B(\mathcal K)$ is the orthogonal projection onto
the subspace 
$$\mathbb Cv^*\mathbf n^{(1)}+\dots+\mathbb C v^*\mathbf n^{(N)}\subset \mathcal K,$$
cf.\ the Gram--Schmidt process.  
Note that the above summands may be linearly dependent.
Let $M_0\in\mathcal B(\mathcal K)$ be the operator defined as
\begin{align}
M_0=U+v^*G_{0,0}v\in\mathcal B(\mathcal K),
\label{160817218}
\end{align}
and $Q\in\mathcal B(\mathcal K)$ 
the orthogonal projection onto 
$\mathop{\mathrm{Ker}}M_0$. 
We denote the \textit{pseudo-inverse} of $M_0$ 
by $M_0^\dagger$, see e.g.\ \cite[Appendix]{IJ2}.
In addition, we define the operators
\begin{align}
w=Uv^*
,\quad 
z=\sum_{\alpha=1}^N
|\Psi^{(\alpha)}\rangle
\langle \Phi^{(\alpha)}| M_0 
-G_{0,0}v
\label{12.12.27.16.4}
\end{align}
initially as $w\in\mathcal B\bigl((\mathcal L^\beta)^*,\mathcal K\bigr)$,  
$z\in \mathcal B\bigl(\mathcal K, (\mathcal L^1)^*\bigr)$, respectively.  
We actually define the operators $w$ and $z$ 
by the restrictions given in Proposition~\ref{13.1.16.2.51} below,
and will denote the restrictions by the same notation $w$ and $z$, respectively.

We now state the main results of this section.

\begin{proposition}\label{13.1.16.2.51}
Suppose that $\beta\ge 1$ in Assumption~\ref{assumV}.
Then the operator $z$ gives a well-defined injection 
\begin{align}
z\colon Q\mathcal K\oplus M_0^\dagger \bigl[P\mathcal K\cap (Q\mathcal K)^\perp\bigr]\to\widetilde{\mathcal E},
\label{17080315}
\end{align}
and the eigenspaces are expressed as 
\begin{align}
\widetilde{\mathcal E}
&=
z\bigl(Q\mathcal K\oplus M_0^\dagger \bigl[P\mathcal K\cap (Q\mathcal K)^\perp\bigr]\bigr)
\oplus
\bigl(\mathbb C\mathbf n \cap \mathop{\mathrm{Ker}} V\bigr),
\label{12.12.19.6.56}\\
\mathcal E
&=
z(Q\mathcal K)
,
\label{12.12.19.6.57}\\
\mathsf E&
=z\bigl(Q\mathcal K\cap (P\mathcal K)^\perp\bigr)
.
\label{12.12.19.6.58}
\end{align}
In addition, the operator $w$ also gives a well-defined surjection
\begin{align}
w&\colon \widetilde{\mathcal E}\to Q\mathcal K\oplus M_0^\dagger \bigl[P\mathcal K\cap (Q\mathcal K)^\perp\bigr],
\label{17080316}
\end{align}
and the compositions $zw$ and $wz$ are expressed as 
\begin{align}
wz=I,\quad 
zw=\Pi,
\label{12.11.26.19.10}
\end{align}
where $I$ is the identity map on  
$Q\mathcal K\oplus M_0^\dagger [P\mathcal K\cap (Q\mathcal K)^\perp]$, 
and $\Pi\colon 
\widetilde{\mathcal E}
\to 
z\bigl(Q\mathcal K\oplus M_0^\dagger \bigl[P\mathcal K\cap (Q\mathcal K)^\perp\bigr]\bigr)$ is the 
projection associated with the decomposition \eqref{12.12.19.6.56}.
In particular, the dimensional relation \eqref{170802} holds,
and any projected eigenfunction $\Psi\in \Pi (\widetilde{\mathcal E})$ 
has the asymptotics
\begin{align}
\Psi-\sum_{\alpha=1}^N
\Bigl[\bigl\langle v^*\Psi^{(\alpha)},v^*(1+G_{0,0}V)\Psi\bigr\rangle \Psi^{(\alpha)} 
-\bigl\langle \mathbf n^{(\alpha)},V\Psi\bigr\rangle\mathbf 1^{(\alpha)}\Bigr]\in \mathcal L^{\beta-2}
.
\label{170804}
\end{align}
\end{proposition}

As a corollary, we have a more direct correspondence between the threshold types
and the invertibility of the \textit{intermediate operators} $M_0$ and $m_0$.
Let us define   
\begin{align}
m_0&
=-\sum_{\alpha=1}^N\bigl|Qv^*\mathbf n^{(\alpha)}\bigr\rangle\bigl\langle Qv^*\mathbf n^{(\alpha)}\bigr|.
\label{160817219}
\end{align}
We note that $m_0$ is defined as an operator in $\mathcal B(\mathcal K)$,
but it may also be considered as an operator in $\mathcal B(Q\mathcal K)$ naturally
since $m_0=0$ on $(Q\mathcal K)^\perp$.
The phrasing ``(not) invertible in $\mathcal B(Q\mathcal K)$'' below is meant in 
the latter sense.

\begin{corollary}\label{160616438}
Suppose that $\beta\ge 1$ in Assumption~\ref{assumV}.
\begin{enumerate}
\item\label{12.12.19.6.30}
The threshold $0$ is a regular point if and only if 
$M_0$ is invertible in $\mathcal B(\mathcal K)$.

\item\label{12.12.19.6.31}
The threshold $0$ is an exceptional point of the first kind 
if and only if 
$M_0$ is not invertible in $\mathcal B(\mathcal K)$
and $m_0$ is invertible in $\mathcal B(Q\mathcal K)$.

\item\label{12.12.19.6.32}
The threshold $0$ is an exceptional point of the second kind 
if and only if 
$M_0$ is not invertible in $\mathcal B(\mathcal K)$
and $m_0=0$.

\item\label{12.12.19.6.33}
The threshold $0$ is an exceptional point of the third kind 
if and only if 
$M_0$ and $m_0$ are not invertible in $\mathcal B(\mathcal K)$ and 
$\mathcal B(Q\mathcal K)$, respectively, and $m_0\neq 0$.
\end{enumerate}
\end{corollary}

\subsection{Proofs}
In the remainder of this section 
we prove Proposition~\ref{13.1.16.2.51} and Corollary~\ref{160616438}.

\begin{lemma}\label{12.11.24.18.24}
For any $u\in\mathcal L^s$, $s\ge 1$,
the function $G_{0,0}u\in (\mathcal L^1)^*$ is expressed 
on $L_\alpha$ as 
\begin{align}
(G_{0,0}u)[n]
&=
\langle\mathbf n^{(\alpha)},u\rangle
-\sum_{m\in L_\alpha,\, m\ge n}(m-n)u[m]\quad \text{for } n\in L_\alpha.
\label{12.11.24.16.17}
\end{align}
In particular, $G_{0,0}$ is bounded as $\mathcal L^1\to (\mathcal L^0)^*$.
In addition, $G_{0,0}u\in \mathcal L^{s-2}$ holds 
if and only if $\langle \mathbf n^{(\alpha)},u\rangle=0$ for all $\alpha=1,\dots,N$.
\end{lemma}
\begin{proof}
Due to the decomposition \eqref{170803} 
we can consider each half-line $L_\alpha$ separately. 
Then the assertion reduces to \cite[Lemma~V.4]{IJ2},
and we omit the proof.
\end{proof}

\begin{lemma}\label{13.1.18.6.0}
The compositions $H_0G_{0,0}$ and $G_{0,0}H_0$,
defined on $\mathcal L^1$ and 
$\mathbb C\mathbf n\oplus 
\mathbb C\mathbf 1\oplus
\mathcal L^1$, respectively, 
are expressed as 
\begin{align*}
H_0G_{0,0}=I_{\mathcal L^1},\quad
G_{0,0}H_0=\Pi_0,
\end{align*}
where $\Pi_0\colon 
\mathbb C\mathbf n\oplus 
\mathbb C\mathbf 1\oplus
\mathcal L^1
\to 
\mathbb C\mathbf 1\oplus\mathcal L^1$ is the projection.
\end{lemma}
\begin{proof}
Due to the decomposition \eqref{170803} 
we can consider separately the finite part $K$ and each of the half-line $L_\alpha$. 
Then the assertion reduces to \cite[Lemma~V.5]{IJ2}, and we omit the detail.
\end{proof}

\begin{proof}[Proof of Proposition~\ref{13.1.16.2.51}.]
\textit{Step 1.}\quad
We first show the well-definedness of $z$ as \eqref{17080315}. 
Note that, since $M_0$ is self-adjoint on $\mathcal K$, and $Q\mathcal K=\mathop{\mathrm{Ker}}M_0$,
it is clear that 
\begin{align*}
Q\mathcal K\cap M_0^\dagger \bigl[P\mathcal K\cap (Q\mathcal K)^\perp\bigr]
=\{0\}.
\end{align*}
For any $\Phi_1\in Q\mathcal K$ and $\Phi_2\in P\mathcal K\cap (Q\mathcal K)^\perp$
we can write 
\begin{align}
z\bigl(\Phi_1+M_0^\dagger \Phi_2\bigr)
=
-G_{0,0}v\Phi_1
+
\sum_{\alpha=1}^N
\bigl\langle \Phi^{(\alpha)}, \Phi_2\bigr\rangle\Psi^{(\alpha)}
-G_{0,0}vM_0^\dagger \Phi_2.
\label{1708043}
\end{align}
Then it follows by Lemma~\ref{12.11.24.18.24}
that $z(\Phi_1+M_0^\dagger \Phi_2)\in (\mathcal L^\beta)^*$, 
and by $H=H_0+vUv^*$, Lemma~\ref{13.1.18.6.0} and $v^{\ast}G_{0,0}v=M_0-U$ that
\begin{align*}
Hz\bigl(\Phi_1+M_0^\dagger \Phi_2\bigr)
&=
-v\Phi_1
-vM_0^\dagger \Phi_2
-vU(M_0-U)\Phi_1
+
vUP\Phi_2
\\&\phantom{{}={}}{}
-vU(M_0-U)M_0^\dagger \Phi_2
\\
&=
0.
\end{align*}
This verifies $z(\Phi_1+M_0^\dagger \Phi_2)\in \widetilde{\mathcal E}$,
and hence the operator $z$ is well-defined as \eqref{17080315}.
The injectivity of $z$ will be discussed in Step 4.
 
\smallskip
\noindent
\textit{Step 2.}\quad
Next we show the well-definedness of $w$ as \eqref{17080316}.
Let $\Psi\in\widetilde{\mathcal E}$ and  compute 
\begin{align*}
M_0w\Psi
=(U+v^*G_{0,0}v)Uv^*\Psi
=v^*(I+G_{0,0}V)\Psi.
\end{align*}
Since 
$$H_0(I+G_{0,0}V)\Psi=(H_0+V)\Psi=0,$$
we can deduce that $(I+G_{0,0}V)\Psi\in \mathbb C\mathbf n$, and hence that 
\begin{align*}
M_0w\Psi
\in P\mathcal K\cap (Q\mathcal K)^\perp.
\end{align*}
This implies that $w\Psi\in Q\mathcal K\oplus M_0^\dagger \bigl[P\mathcal K\cap (Q\mathcal K)^\perp\bigr]$.
Hence the operator $w$ is well-defined as \eqref{17080316}.
The surjectivity of $w$ will be discussed in Step 4.

\smallskip
\noindent
\textit{Step 3.}\quad
Here we prove \eqref{12.12.19.6.56}.
For any $\Phi_1\in Q\mathcal K$ and $\Phi_2\in P\mathcal K\cap (Q\mathcal K)^\perp$
we can compute by using \eqref{1708043}, $V=vUv^*$ and $v^{\ast}G_{0,0}v=M_0-U$:
\begin{align}
\begin{split}
Vz(\Phi_1+M_0^\dagger \Phi_2)
&=
v\Bigl[-U(M_0-U)\Phi_1
+
UP\Phi_2
-U(M_0-U)M_0^\dagger \Phi_2\Bigr]
\\
&=
v\bigl[\Phi_1+M_0^\dagger \Phi_2\bigr]
.
\end{split}
\label{170804331}
\end{align}
This and the injectivity of $v$ show that 
\begin{align*}
z\bigl(Q\mathcal K\oplus M_0^\dagger \bigl[P\mathcal K\cap (Q\mathcal K)^\perp\bigr]\bigr)
\cap
\bigl(\mathbb C\mathbf n \cap \mathop{\mathrm{Ker}} V\bigr)
=\{0\}.
\end{align*}
Now by the result of Step 1 we have 
\begin{align*}
\widetilde{\mathcal E}
\supset 
z\bigl(Q\mathcal K\oplus M_0^\dagger \bigl[P\mathcal K\cap (Q\mathcal K)^\perp\bigr]\bigr)\oplus
\bigl(\mathbb C\mathbf n \cap \mathop{\mathrm{Ker}} V\bigr),
\end{align*}
and let us show the inverse inclusion.
For any $\Psi\in \widetilde{\mathcal E}$ we can write  
$$
zw\Psi
=\sum_{\alpha=1}^N
\bigl\langle \Phi^{(\alpha)},M_0Uv^*\Psi\bigr\rangle \Psi^{(\alpha)}  
-G_{0,0}V\Psi,
$$
from which we can easily deduce that 
$$H_0(\Psi-zw\Psi)=V(\Psi-zw\Psi)=0.$$
Hence we have 
\begin{align}
\Psi-zw\Psi\in \mathbb C\mathbf n\cap \mathop{\mathrm{Ker}}V.
\label{17080317}
\end{align}
Since $w\Psi \in Q\mathcal K\oplus M_0^\dagger \bigl[P\mathcal K\cap (Q\mathcal K)^\perp\bigr]$
by the result of Step 2, we obtain 
\begin{align*}
\widetilde{\mathcal E}
\subset 
z\bigl(Q\mathcal K\oplus M_0^\dagger \bigl[P\mathcal K\cap (Q\mathcal K)^\perp\bigr]\bigr)\oplus
\bigl(\mathbb C\mathbf n \cap \mathop{\mathrm{Ker}} V\bigr),
\end{align*}
verifying \eqref{12.12.19.6.56}.

\smallskip
\noindent
\textit{Step 4.}\quad
Let us prove \eqref{12.11.26.19.10}.
Similarly to \eqref{170804331}, we have that 
for any $\Phi_1\in Q\mathcal K$ and $\Phi_2\in P\mathcal K\cap (Q\mathcal K)^\perp$ 
\begin{align*}
wz(\Phi_1+M_0^\dagger \Phi_2)
=
\Phi_1+M_0^\dagger \Phi_2
.
\end{align*}
Hence the first identity of \eqref{12.11.26.19.10} follows,
and this in particular implies that 
$z$ is injective as \eqref{17080315}, and that $w$ is surjective as \eqref{17080316}.
The second identity of \eqref{12.11.26.19.10} 
follows by 
\eqref{12.12.19.6.56} and \eqref{17080317}.

\smallskip
\noindent
\textit{Step 5.}\quad
Let us show \eqref{12.12.19.6.57} and \eqref{12.12.19.6.58}.
By \eqref{12.12.19.6.56} it is clear that 
\begin{align*}
\mathsf E\subset \mathcal E\subset 
z\bigl(Q\mathcal K\oplus M_0^\dagger \bigl[P\mathcal K\cap (Q\mathcal K)^\perp\bigr]\bigr).
\end{align*}
Then recall the expression \eqref{1708043}
for any 
$\Phi_1\in Q\mathcal K$ and $\Phi_2\in P\mathcal K\cap (Q\mathcal K)^\perp$.
By Lemma~\ref{12.11.24.18.24}
the right-hand side of \eqref{1708043} belongs to $\mathbb C\mathbf 1\oplus\mathcal L^{\beta-2}$
if and only if $\Phi_2=0$,
from which \eqref{12.12.19.6.57} follows.
Similarly, by Lemma~\ref{12.11.24.18.24} 
the right-hand side of \eqref{1708043} belongs to $\mathcal L^{\beta-2}$
if and only if $\Phi_2=0$ and $P\Phi_1=0$.
This verifies \eqref{12.12.19.6.58}.

\smallskip
\noindent
\textit{Step 6.}\quad
Finally we prove the last assertions of the proposition.
We first note that by Lemma~\ref{12.11.24.18.24} 
and compactness of the embedding $(\mathcal L^0)^*\to (\mathcal L^1)^*$
the operator $Uv^*G_{0,0}v\in \mathcal \mathcal B(\mathcal K)$ is in fact compact,
and hence 
\begin{align*}
\dim Q\mathcal K
=\dim\mathop{\mathrm{Ker}} M_0
=\dim\mathop{\mathrm{Ker}} (1+Uv^*G_{0,0}v)<\infty.
\end{align*}
This and \eqref{12.12.19.6.58} in particular 
imply the last inequality of \eqref{170802}.
As for the first identity of \eqref{170802}, 
note \eqref{12.12.19.6.56}--\eqref{12.12.19.6.58} and 
\begin{align*}
\dim\bigl(\mathbb C\mathbf n \cap \mathop{\mathrm{Ker}} V\bigr)
=N-\dim P\mathcal K,
\end{align*}
and then it suffices to show that 
\begin{align}
\dim \bigl(Q\mathcal K\big/\bigl[Q\mathcal K\cap (P\mathcal K)^\perp\bigr]\bigr)
+\dim \bigl(M_0^\dagger \bigl[P\mathcal K\cap (Q\mathcal K)^\perp\bigr]\bigr)
=\dim P\mathcal K.
\label{170804544}
\end{align}
By the rank--nullity theorem applied to the restrictions 
$Q\colon P\mathcal K\to QP\mathcal K$ and $P\colon Q\mathcal K\to PQ\mathcal K$
we have 
\begin{align}
\begin{split}
\dim QP\mathcal K+\dim \bigl(P\mathcal K\cap (Q\mathcal K)^\perp\bigr)
&=
\dim P\mathcal K,
\\
\dim PQ\mathcal K+\dim \bigl(Q\mathcal K\cap (P\mathcal K)^\perp\bigr)
&=
\dim Q\mathcal K.
\end{split}
\label{1708045}
\end{align}
By the expressions 
\begin{align*}
QP\mathcal K&=\mathbb CQ\Phi_1+\dots+\mathbb CQ\Phi_N,
\\
PQ\mathcal K
&=\bigl\{\langle Q\Phi_1,\Phi\rangle\Phi_1+\dots+\langle Q\Phi_N,\Phi\rangle\Phi_N;\ \Phi\in\mathcal K\bigr\}
\end{align*}
we can deduce $\dim QP\mathcal K=\dim PQ\mathcal K$, and then by \eqref{1708045} 
it follows that 
\begin{align*}
\dim \bigl(P\mathcal K\cap (Q\mathcal K)^\perp\bigr)
+\dim Q\mathcal K
-\dim \bigl(Q\mathcal K\cap (P\mathcal K)^\perp\bigr)
&=
\dim P\mathcal K.
\end{align*}
Then, since the restriction of $M_0^\dagger$ to $(Q\mathcal K)^\perp$
is invertible, we obtain \eqref{170804544}.

The asymptotics \eqref{170804} for $\Psi\in \Pi(\widetilde{\mathcal E})$ 
can be verified by computing $\Psi=zw\Psi$ and applying Lemma~\ref{12.11.24.18.24}.
We omit the detailed computations.
\end{proof}

\begin{proof}[Proof of Corollary~\ref{160616438}]
The arguments below are based on 
Definition~\ref{def-reg-excp} and 
the expressions \eqref{12.12.19.6.56}--\eqref{12.12.19.6.58}.

\smallskip
\noindent
\textit{\ref{12.12.19.6.30}.}\quad
The threshold $0$ is a regular point 
if and only if $Q\mathcal K=\{0\}$,
which is obviously equivalent to the invertibility of $M_0$ in $\mathcal B(\mathcal K)$.

\smallskip
\noindent
\textit{\ref{12.12.19.6.31}.}\quad
The threshold $0$ be an exceptional point of the first kind
if and only if $Q\mathcal K\neq \{0\}$ and $Q\mathcal K\cap (P\mathcal K)^\perp=\{0\}$.
The former condition $Q\mathcal K\neq \{0\}$ is equivalent to 
non-invertibility of $M_0$ in $\mathcal B(\mathcal K)$.
The latter can be rewritten as $Q\mathcal K= QP\mathcal K$,
which is the case if and only if $m_0$ is invertible in $\mathcal B(Q\mathcal K)$.

\smallskip
\noindent
\textit{\ref{12.12.19.6.32}.}\quad
The threshold $0$ be an exceptional point of the second kind
if and only if 
$Q\mathcal K\neq \{0\}$,
which again is equivalent to 
non-invertibility of $M_0$ in $\mathcal B(\mathcal K)$,
and 
$Q\mathcal K=Q\mathcal K\cap (P\mathcal K)^\perp$.
The latter condition is equivalent to $Q\mathcal K\subset (P\mathcal K)^\perp$
or $P\mathcal K\subset (Q\mathcal K)^\perp$,
which in turn is equivalent to $m_0=0$.

\smallskip
\noindent
\textit{\ref{12.12.19.6.33}.}\quad
This case is treated as a complement of 
the above \ref{12.12.19.6.30}--\ref{12.12.19.6.32}. 
We may also discuss it directly as above, but let us omit it.
\end{proof}

\section{Implementation of resolvent expansion}\label{1608216}

In this section we prove 
Theorems~\ref{thm-reg}--\ref{thm-ex3},
employing the expansion scheme of Jensen--Nenciu \cite{JN}.
Due to our choice of the free operator many parts of the proofs 
are identical with those in \cite{IJ2},
but we will repeat most of them for readability.
On the other hand we will often refer to \cite[Appendix]{IJ2}
for a form of the expansion scheme adapted to the problem at hand.

\subsection{The first step in the expansion}\label{1608217}

Here we present the first step of the expansion
common to all the threshold types.

Define the operator
$M(\kappa)\in\mathcal B(\mathcal K)$
for $\mathop{\mathrm{Re}}\kappa>0$ by
\begin{equation}
M(\kappa)=U+v^*R_0(\kappa)v.\label{Mdef}
\end{equation}
Fix $\kappa_0>0$ such that $z=-\kappa^2$ belongs to the resolvent set of $H$ for any $\mathop{\mathrm{Re}}\kappa\in(0,\kappa_0)$. This is possible due to the decay assumptions on $V$.

\begin{lemma}\label{prop23}
Let the operator $M(\kappa)$ be defined as above.
\begin{enumerate}
\item\label{16082118}
Let Assumption~\ref{assumV} hold for some integer $\beta\ge 2$.
Then 
\begin{equation}\label{M-expand}
M(\kappa)=\sum_{j=0}^{\beta-2} \kappa^jM_j+\mathcal{O}(\kappa^{\beta-1})
\quad 
\text{in }\mathcal B(\mathcal K)
\end{equation}
with $M_j\in\mathcal B(\mathcal K)$ given by
\begin{align}\label{m-expand}
M_0=U+v^*G_{0,0}v,\quad
M_j=v^*G_{0,j}v\text{ for }  j\geq1.
\end{align}

\item\label{16082119}
Let Assumption~\ref{assumV} hold with $\beta\ge 1$.
For any $0<\mathop{\mathrm{Re}}\kappa<\kappa_0$ the operator $M(\kappa)$
is invertible in $\mathcal B(\mathcal K)$,
and 
\begin{align*}
M(\kappa)^{-1}&=U-Uv^*R(\kappa)vU.
\end{align*}
Moreover, 
\begin{align}
R(\kappa)&=R_0(\kappa)
-R_0(\kappa)vM(\kappa)^{-1}v^*R_0(\kappa).\label{second-resolvent}
\end{align}
\end{enumerate}
\end{lemma}
\begin{proof}
See \cite[Lemma~V\hspace{-.1em}I.1]{IJ2}.
\end{proof}

Note that the above $M_0$ coincides with \eqref{160817218}.
By Lemma~\ref{prop23} and Proposition~\ref{prop12} 
in order to expand $R(\kappa)$ 
it suffices to expand the inverse $M(\kappa)^{-1}$. 
If the leading operator $M_0$ of $M(\kappa)$
is invertible in $\mathcal B(\mathcal K)$,
we can employ the Neumann series to compute the expansion of $M(\kappa)^{-1}$.
Otherwise, we are to apply the inversion formula \cite[Proposition~A.2]{IJ2}.

\subsection{Regular threshold}\label{1608241726}

We start with the proof of Theorem~\ref{thm-reg}. 
The inversion formula \cite[Proposition~A.2]{IJ2} is not needed here.

\begin{proof}[Proof of Theorem~\ref{thm-reg}]
By the assumption and Corollary~\ref{160616438}
the operator $M_0$ is invertible in $\mathcal{B}(\mathcal{K})$.
Hence we can use the Neumann series to invert \eqref{M-expand}.
Let us write it as
\begin{align}
M(\kappa)^{-1}=\sum_{j=0}^{\beta-2}\kappa^jA_j+\mathcal O(\kappa^{\beta-1}),\quad
A_j\in\mathcal B(\mathcal K).
\label{160819}
\end{align}
The coefficients $A_j$ are written explicitly in terms of the $M_j$.
For example, the first two are given by 
\begin{align}
A_0&=M_0^{-1},\quad 
A_1=-M_0^{-1}M_1M_0^{-1}.\label{12.12.1.7.18}
\end{align}
We insert the expansions \eqref{free-expan} with $N=\beta-2$ and \eqref{160819}
into \eqref{second-resolvent}, and then obtain the expansion
\begin{align}
R(\kappa)&=\sum_{j=0}^{\beta-2}\kappa^j
G_j+\mathcal O(\kappa^{\beta-1});\quad
G_j=G_{0,j}
-\sum_{\genfrac{}{}{0pt}{}{j_1\ge 0,j_2\ge 0,j_3\ge 0}{j_1+j_2+j_3=j}}
G_{0,j_1}vA_{j_2}v^*G_{0,j_3}.
\label{170820}
\end{align}
This verifies \eqref{17081821} and \eqref{17081822}.
\end{proof}

\subsection{Exceptional threshold of the first kind}\label{1608241727}

Next, we prove Theorem~\ref{thm-ex1}.

\begin{proof}[Proof of Theorem~\ref{thm-ex1}]
By the assumption and Corollary~\ref{160616438}
the leading operator $M_0$ from \eqref{M-expand}
is not invertible in $\mathcal B(\mathcal K)$, 
and we are going to apply the inversion formula \cite[Proposition~A.2]{IJ2}
to invert the expansion~\eqref{M-expand}.
Let us write the expansion \eqref{M-expand} as
\begin{equation}\label{ex1-M-expand}
M(\kappa)
=\sum_{j=0}^{\beta-2}\kappa^jM_j+\mathcal{O}(\kappa^{\beta-1})=M_0+\kappa \widetilde{M}_1(\kappa).
\end{equation}
Let $Q$ be the orthogonal projection onto $\mathop{\mathrm{Ker}} M_0$
as in 
Section~\ref{12.12.19.2.5}, and define 
\begin{equation}\label{def-J0}
m(\kappa)=
\sum_{j=0}^{\infty}(-1)^j
\kappa^jQ\widetilde{M}_1(\kappa)\bigl[(M_0^\dagger+Q)\widetilde{M}_1(\kappa)\bigr]^{j}Q.
\end{equation}
Then by \cite[Proposition~A.2]{IJ2} we have  
\begin{equation}
M(\kappa)^{-1}=(M(\kappa)+Q)^{-1}+\kappa^{-1}
(M(\kappa)+Q)^{-1}m(\kappa)^{\dagger}(M(\kappa)+Q)^{-1}.
\label{16082255}
\end{equation}
Note that by using \eqref{ex1-M-expand} we 
can rewrite \eqref{def-J0} in the form
\begin{equation}
m(\kappa)=\sum_{j=0}^{\beta-3}\kappa^jm_j+\mathcal{O}(\kappa^{\beta-2});
\quad 
m_j\in\mathcal B(Q\mathcal K).
\label{16082317}
\end{equation}
For  later use we write down the first five coefficients:
\begin{align}
\begin{split}
m_0&=QM_1Q,
\end{split}
\label{1608231743}
\\
\begin{split}
m_1&=Q\Bigl(M_2-M_1(M_0^\dagger +Q)M_1\Bigr)Q,
\end{split}
\label{1608231744}
\\
\begin{split}
m_2&=
Q\Bigl(M_3-M_1(M_0^\dagger +Q)M_2-M_2(M_0^\dagger +Q)M_1
\\&\phantom{{}=Q(}{}
+M_1(M_0^\dagger +Q)M_1(M_0^\dagger +Q)M_1\Bigr)Q
,
\end{split}
\label{1608231745}
\\
\begin{split}
m_3&=
Q\Bigl(
M_4
-M_1(M_0^\dagger +Q)M_3
-M_2(M_0^\dagger +Q)M_2
-M_3(M_0^\dagger +Q)M_1
\\&\phantom{{}=Q(}{}
+M_1(M_0^\dagger +Q)M_1(M_0^\dagger +Q)M_2
\\&\phantom{{}=Q(}{}
+M_1(M_0^\dagger +Q)M_2(M_0^\dagger +Q)M_1
\\&\phantom{{}=Q(}{}
+M_2(M_0^\dagger +Q)M_1(M_0^\dagger +Q)M_1
\\&\phantom{{}=Q(}{}
-M_1(M_0^\dagger +Q)M_1(M_0^\dagger +Q)M_1(M_0^\dagger +Q)M_1\Bigr)Q,
\end{split}
\label{1608231746}
\\
\begin{split}
m_4&=
Q\Bigl(
M_5
\\&\phantom{{}=Q(}{}
-M_1(M_0^\dagger +Q)M_4-M_2(M_0^\dagger +Q)M_3
\\&\phantom{{}=Q(}{}
-M_3(M_0^\dagger +Q)M_2-M_4(M_0^\dagger +Q)M_1
\\&\phantom{{}=Q(}{}
+M_1(M_0^\dagger +Q)M_1(M_0^\dagger +Q)M_3
+M_1(M_0^\dagger +Q)M_2(M_0^\dagger +Q)M_2
\\&\phantom{{}=Q(}{}
+M_2(M_0^\dagger +Q)M_2(M_0^\dagger +Q)M_1
+M_2(M_0^\dagger +Q)M_1(M_0^\dagger +Q)M_2
\\&\phantom{{}=Q(}{}
+M_1(M_0^\dagger +Q)M_3(M_0^\dagger +Q)M_1
+M_3(M_0^\dagger +Q)M_1(M_0^\dagger +Q)M_1
\\&\phantom{{}=Q(}{}
-M_1(M_0^\dagger +Q)M_1(M_0^\dagger +Q)M_1(M_0^\dagger +Q)M_2
\\&\phantom{{}=Q(}{}
-M_1(M_0^\dagger +Q)M_1(M_0^\dagger +Q)M_2(M_0^\dagger +Q)M_1
\\&\phantom{{}=Q(}{}
-M_1(M_0^\dagger +Q)M_2(M_0^\dagger +Q)M_1(M_0^\dagger +Q)M_1
\\&\phantom{{}=Q(}{}
-M_2(M_0^\dagger +Q)M_1(M_0^\dagger +Q)M_1(M_0^\dagger +Q)M_1
\\&\phantom{{}=Q(}{}
+M_1(M_0^\dagger +Q)M_1(M_0^\dagger +Q)M_1(M_0^\dagger +Q)M_1(M_0^\dagger +Q)M_1\Bigr)Q.
\end{split}
\label{1608231747}
\end{align}
Note that by \eqref{m-expand}, \eqref{170803} and \eqref{G01} 
the above coefficient $m_0$ certainly coincides with \eqref{160817219}.
Then by the assumption and Corollary~\ref{160616438} 
the coefficient $m_0$ is invertible in $\mathcal B(Q\mathcal K)$.
Thus the Neumann series provides 
the expansion of the inverse $m(\kappa)^\dagger$.
Let us write it as 
\begin{equation}
m(\kappa)^\dagger=\sum_{j=0}^{\beta-3}\kappa^jA_j+\mathcal{O}(\kappa^{\beta-2});\quad
A_j\in\mathcal B(Q\mathcal K)
\label{ex1-m-expand}
\end{equation}
with
\begin{align}
\begin{split}
A_0&=m_0^\dagger,\quad
A_1=-m_0^\dagger m_1m_0^\dagger,\quad
A_2=m_0^\dagger \bigl(-m_2+m_1m_0^\dagger m_1\bigr)m_0^\dagger .
\end{split}
\label{170821}
\end{align}
The Neumann series also provide an expansion of 
$(M(\kappa)+Q)^{-1}$, which we write  
\begin{align}
(M(\kappa)+Q)^{-1}
=\sum_{j=0}^{\beta-2}\kappa^jB_j+\mathcal{O}(\kappa^{\beta-1});\quad
B_j\in\mathcal B(\mathcal K).
\label{1608225}
\end{align}
The first four coefficients can be written as follows: 
\begin{align}
\begin{split}
B_0&=M_0^\dagger +Q,\\
B_1&=-(M_0^\dagger +Q) M_1(M_0^\dagger +Q),\\
B_2&=(M_0^\dagger +Q) \bigl(-M_2+M_1(M_0^\dagger +Q) M_1\bigr)(M_0^\dagger +Q),\\
B_3&=(M_0^\dagger +Q) \Bigl(-M_3
+M_1(M_0^\dagger +Q) M_2
+M_2(M_0^\dagger +Q) M_1
\\&\phantom{{}=-(M_0^\dagger +Q) \Bigl({}}{}
-M_1(M_0^\dagger +Q) M_1(M_0^\dagger +Q) M_1\Bigr)(M_0^\dagger +Q)
.
\end{split}
\label{1708212}
\end{align} 
Now we insert the expansions \eqref{ex1-m-expand} and \eqref{1608225} 
into the formula \eqref{16082255},
and then 
\begin{align}
M(\kappa)^{-1}
=\sum_{j=-1}^{\beta-4}\kappa^jC_j+\mathcal O(\kappa^{\beta-3});
\quad
C_j=B_j
+\sum_{\genfrac{}{}{0pt}{}{j_1\ge 0,j_2\ge 0,j_3\ge 0}{j_1+j_2+j_3=j+1}}
B_{j_1}A_{j_2}B_{j_3}
\label{160822553}
\end{align}
with $B_{-1}=0$. 
Next we insert the expansions \eqref{free-expan} with $N=\beta-3$ and 
\eqref{160822553} into the formula \eqref{second-resolvent}.
Then we obtain the expansion
\begin{align}
R(\kappa)&=\sum_{j=-1}^{\beta-4}\kappa^j
G_j+\mathcal O(\kappa^{\beta-3});\quad
G_j=G_{0,j}
-\sum_{\genfrac{}{}{0pt}{}{j_1\ge 0,j_2\ge -1,j_3\ge 0}{j_1+j_2+j_3=j}}
G_{0,j_1}vC_{j_2}v^*G_{0,j_3}
\label{170820133}
\end{align}
with $G_{0,-1}=0$. This verifies \eqref{expand-first} and $G_{-2}=0$.

Let us next compute $G_{-1}$.
By the above expressions we can write
\begin{align*}
G_{-1}&=-G_{0,0}vC_{-1} v^*G_{0,0}
=-G_{0,0}vm_0^\dagger v^*G_{0,0}
.
\end{align*}
The expression \eqref{160817219} implies that 
$m_0$ is of finite rank, self-adjoint and non-positive on $Q\mathcal K$.
Since $m_0$ is invertible in $\mathcal B(Q\mathcal K)$, 
it is in fact strictly negative on $Q\mathcal K$,
and we can find an orthogonal basis $\{\Phi_\gamma\}\subset Q\mathcal K$ such that  
\begin{align}
m_0^\dagger =- \sum_\gamma|\Phi_\gamma\rangle\langle \Phi_\gamma|.
\label{16082516}
\end{align}
Then we can write 
\begin{align*}
G_{-1}=\sum_\gamma|\Psi_\gamma\rangle\langle \Psi_\gamma|,
\end{align*}
where $\Psi_\gamma=z\Phi_\gamma\in \mathsf E$ certainly form a resonance basis due to Proposition~\ref{13.1.16.2.51}.
The orthonormality of $\{\Psi_\gamma\}\subset \mathcal E$ follows by 
\begin{align}
\delta_{\gamma\gamma'}
&
=
\|\Phi_\gamma\|^{-2} \bigl\langle \Phi_{\gamma'},m_0m_0^\dagger\Phi_\gamma\bigr\rangle
=\sum_{\alpha=1}^N\bigl\langle V\Psi_{\gamma'},\mathbf n^{(\alpha)}\bigr\rangle
\bigl\langle \mathbf n^{(\alpha)},V\Psi_\gamma\bigr\rangle
\label{17081214}
\end{align}
and Lemma~\ref{12.11.24.18.24}, since $\mathsf E=\{0\}$. 
Hence we obtain the latter identity of \eqref{1608229}.
\end{proof}

\subsection{Exceptional threshold of the second kind}\label{1608241728}

Here we prove Theorem~\ref{thm-ex2}.
We begin with the following lemma.
\begin{lemma}\label{lemma24}
Let $u_1,u_2\in\mathcal{L}^4$. Assume that
\begin{equation}\label{zero-cond}
\langle\mathbf{n}^{(\alpha)},u_1\rangle=0,\quad \alpha=1,\dots,N.
\end{equation}
Then one has that  $G_{0,0}u_1\in\mathcal{L}^2$ and that
\begin{equation}
\langle u_2,G_{0,2}u_1\rangle=-\langle G_{0,0}u_2,G_{0,0}u_1\rangle.
\end{equation}
\end{lemma}
\begin{proof}
Due to the decomposition \eqref{170803} we can discuss separately 
on $K$ and $L_\alpha$, $\alpha=1,\dots,N$.
The assertion on $K$ is obvious by \eqref{170803}. 
As for each $L_\alpha$, we can discuss along the same lines as \cite[Lemma~I\hspace{-.1em}X.1]{IJ2}. 
However, since our assumption here is slightly relaxed,  
we can not directly use \cite[Lemma~4.16]{IJ1} as in \cite[Lemma~I\hspace{-.1em}X.1]{IJ2},
and we need to accordingly modify the assumption of \cite[Lemma~4.16]{IJ1}.  
To avoid notational confusion, we will provide this modification
later in Lemma~\ref{12.12.27.15.3}. 
With this modification we can prove the assertion similarly to \cite[Lemma~I\hspace{-.1em}X.1]{IJ2}. 
We omit the rest of the proof.
\end{proof}

\begin{remark*}
The above relaxed version is necessary not in the proof of Theorem~\ref{thm-ex2} 
but in the proof of Theorem~\ref{thm-ex3} when we prove \eqref{1708182228}. 
This is also the case in \cite{IJ2}; 
This is a small error in \cite{IJ2},
but can be corrected by using Lemma~\ref{12.12.27.15.3}. 
\end{remark*}

\begin{proof}[Proof of Theorem~\ref{thm-ex2}]
We repeat the former part of the proof of Theorem~\ref{thm-ex1}.
By the assumption and Corollary~\ref{160616438}
the leading operator $M_0$ from \eqref{M-expand} is not invertible
in $\mathcal B(\mathcal K)$. 
Write the expansion \eqref{M-expand} in the form \eqref{ex1-M-expand},
let $Q$ be the orthogonal projection onto $\mathop{\mathrm{Ker}} M_0$, and 
define $m(\kappa)$ as \eqref{def-J0}.
Then by \cite[Proposition~A.2]{IJ2} we have the formula \eqref{16082255}. 
The operator $m(\kappa)$ has the expansion \eqref{16082317}
with the coefficients given by \eqref{1608231743}--\eqref{1608231746},
but in this case we actually have
\begin{align}\label{am0b}
m_0=0,
\quad
m_1=QM_2Q,
\quad 
m_2=0.
\end{align}
In fact, by the assumption and Proposition~\ref{13.1.16.2.51} we have
\begin{align}
Qv^*\mathbf n^{(\alpha)}=0\quad \text{for }\alpha=1,\dots,N, 
\label{17081122}
\end{align}
and this leads to, by \eqref{m-expand}, \eqref{170803} and \eqref{G01},
\begin{align}
QM_1=M_1Q=0.
\label{16082220}
\end{align}
Hence \eqref{am0b} follows by \eqref{1608231743}--\eqref{1608231746} and \eqref{16082220}.
Now we note that then the operator $m_1$ has to be invertible in 
$\mathcal B(Q\mathcal K)$.
Otherwise, we can apply \cite[Proposition~A.2]{IJ2} once more,
but this leads to a singularity of order 
$\kappa^{-j}$, $j\ge 3$, in the expansion of $R(\kappa)$,
which contradicts the general estimate 
\begin{align}
\|R(\kappa)\|_{\mathcal B(\mathcal H)}\le \bigl[\mathop{\mathrm{dist}}(-\kappa^2,\sigma(H))\bigr]^{-1}.
\label{170829}
\end{align}
Hence  the Neumann series 
provides an expansion of $m(\kappa)^\dagger$ of the form
\begin{equation}
m(\kappa)^\dagger=\sum_{j=-1}^{\beta-5}\kappa^jA_j+\mathcal{O}(\kappa^{\beta-4}),\quad
A_j\in\mathcal B(Q\mathcal K),
\label{ex1-m-expandb}
\end{equation}
with 
\begin{align*}
&A_{-1}=m_1^{\dagger},\quad
A_0=-m_1^\dagger m_2m_1^\dagger
,\quad
A_1=m_1^\dagger \bigl(-m_3+ m_2m_1^\dagger m_2\bigr)m_1^\dagger 
.
\end{align*}
These are actually simplified by \eqref{am0b} as
\begin{align}
&A_{-1}=m_1^{\dagger},\quad
A_0=0,\quad
A_1=-m_1^\dagger m_3m_1^\dagger
.
\label{1608252242}
\end{align}
The Neumann series provides  an expansion of $(M(\kappa)+Q)^{-1}$ 
in the same form as \eqref{1608225}
with the same coefficients $B_j$ given there.
Now we insert the expansions \eqref{ex1-m-expandb} and \eqref{1608225} 
into the formula \eqref{16082255},
and then 
\begin{align}
M(\kappa)^{-1}
=\sum_{j=-2}^{\beta-6}\kappa^jC_j+\mathcal O(\kappa^{\beta-5});
\quad
C_j=B_j
+\sum_{\genfrac{}{}{0pt}{}{j_1\ge 0,j_2\ge -1,j_3\ge 0}{j_1+j_2+j_3=j+1}}
B_{j_1}A_{j_2}B_{j_3}
\label{160822553b}
\end{align}
with $B_{-2}=B_{-1}=0$.
We then insert the expansions \eqref{free-expan} with $N=\beta-4$ and \eqref{160822553b}
into the formula \eqref{second-resolvent}.
Finally we obtain the expansion
\begin{align}
R(\kappa)&=\sum_{j=-2}^{\beta-6}\kappa^j
G_j+\mathcal O(\kappa^{\beta-5});\quad
G_j=G_{0,j}
-\sum_{\genfrac{}{}{0pt}{}{j_1\ge 0,j_2\ge -2,j_3\ge 0}{j_1+j_2+j_3=j}}
G_{0,j_1}vC_{j_2}v^*G_{0,j_3}
\label{170820157}
\end{align}
with $G_{0,-2}=G_{0,-1}=0$. Hence we obtain the expansion \eqref{expand-second}.

Let us compute the coefficients $G_{-2}$ and $G_{-1}$.
We can use the above expressions of the coefficients 
to write 
\begin{align}
G_{-2}&
=
-G_{0,0}vC_{-2}v^*G_{0,0}
=
-G_{0,0}v m_1^{\dagger}v^*G_{0,0}
=-z (Qv^*G_{0,2}vQ)^{\dagger}z^*.
\label{160826}
\end{align}
By the self-adjointness of $G_{-2}$, \eqref{160826}, Proposition~\ref{13.1.16.2.51} 
and the assumptions we can see that 
\begin{align}
\mathop{\mathrm{Ran}}G_{-2}=(\mathop{\mathrm{Ker}}G_{-2})^\perp
\subset \mathcal E=\mathsf E.
\label{17081823}
\end{align}
In addition, 
for any $\Psi\in \mathsf E$ 
we can write $\Psi=z\Phi=-G_{0,0}v\Phi$ for some 
$\Phi\in Q\mathcal K$ by Proposition~\ref{13.1.16.2.51},  
so that by \eqref{17081122} and Lemma~\ref{lemma24} 
\begin{align*}
\langle \Psi,G_{-2}\Psi\rangle
&=
-\bigl\langle G_{0,0}v\Phi,G_{0,0}v (Qv^*G_{0,2}vQ)^{\dagger}z^*\Psi\bigr\rangle
=
\|\Psi\|_{\mathcal H}^2
.
\end{align*}
This implies that $G_{-2}$ is unitary on $\mathsf E$,
and in particular that the inclusion in \eqref{17081823} is in fact an equality.
Similarly we can show that $G_{-2}G_{-2}=G_{-2}$,
and hence $G_{-2}$ is the orthogonal bound projection $\mathsf P$, as asserted in \eqref{ex2-G-2}.

As for $G_{-1}$, 
we can first write 
\begin{align}
\begin{split}
G_{-1}
&=
-G_{0,0}vC_{-1}v^*G_{0,0}
-G_{0,0}vC_{-2}v^*G_{0,1}
-G_{0,1}vC_{-2}v^*G_{0,0}.
\end{split}
\label{160822203}
\end{align}
If we make use of the zero operators in \eqref{am0b}, \eqref{16082220}
and \eqref{1608252242},
we can easily see that $G_{-1}=0$ from \eqref{160822203}. 
We obtain the latter assertion of \eqref{ex2-G-2}.
\end{proof}

\subsection{Exceptional threshold of the third kind}\label{1608241729}

Finally we prove Theorem~\ref{thm-ex3}.
Compared with the proof of Theorem~\ref{thm-ex2},
this case needs once more application of the inversion formula 
\cite[Proposition~A.2]{IJ2},
and computations get slightly more complicated.

\begin{proof}[Proof of Theorem~\ref{thm-ex3}]
Let us repeat arguments of the previous proof to some extent.
We write the expansion \eqref{M-expand} in the form \eqref{ex1-M-expand},
let $Q$ be the orthogonal projection onto $\mathop{\mathrm{Ker}} M_0$, and 
define $m(\kappa)$ as \eqref{def-J0}.
Then by \cite[Proposition~A.2]{IJ2} we have the formula \eqref{16082255}, again. 
The operator $m(\kappa)$ 
has the expansion \eqref{16082317}
with coefficients given by \eqref{1608231743}--\eqref{1608231747},
but without \eqref{am0b}--\eqref{16082220}
by the assumption and Corollary~\ref{160616438}.
Now we apply the inversion formula \cite[Proposition~A.2]{IJ2}
to the operator $m(\kappa)$.
Write the expansion  \eqref{16082317} in the form
\begin{equation}
m(\kappa)=m_0+\kappa \widetilde m_1(\kappa).
\label{ex1-M-expandbbb}
\end{equation}
The leading operator $m_0$ is non-zero and not invertible in $\mathcal B(Q\mathcal K)$ 
by the assumption and Corollary~\ref{160616438}. 
Let $S$ be the orthogonal projection onto $\mathop{\mathrm{Ker}}m_0\subset Q\mathcal K$, 
and set 
\begin{equation}\label{def-J0bbb}
q(\kappa)=
\sum_{j=0}^{\infty}(-1)^j
\kappa^jS\widetilde{m}_1(\kappa)\bigl[(m_0^\dagger +S)\widetilde{m}_1(\kappa)\bigr]^{j}S.
\end{equation}
Then we have by \cite[Proposition~A.2]{IJ2} that 
\begin{equation}
m(\kappa)^\dagger =(m(\kappa)+S)^\dagger +\kappa^{-1}
(m(\kappa)+S)^\dagger  q(\kappa)^{\dagger}(m(\kappa)+S)^\dagger .
\label{16082255b}
\end{equation}
Using \eqref{16082317} and 
\eqref{ex1-M-expandbbb}, let us write \eqref{def-J0bbb} in the form
\begin{equation*}
q(\kappa)=\sum_{j=0}^{\beta-4}\kappa^jq_j+\mathcal{O}(\kappa^{\beta-3});
\quad 
q_j\in\mathcal B(S\mathcal K).
\end{equation*}
The first to fourth coefficients are given as 
\begin{align}\label{am0bb}
\begin{split}
q_0&=Sm_1S,\\
q_1&=S\bigl[m_2-m_1(m_0^\dagger +S)m_1\bigr]S,\\
q_2&=S\bigl[m_3-m_1(m_0^\dagger +S)m_2-m_2(m_0^\dagger +S)m_1
\\&\phantom{{}=S\bigl({}}{}
+m_1(m_0^\dagger +S)m_1(m_0^\dagger +S)m_1\bigr]S,
\\
q_3&=S\bigl[m_4
-m_1(m_0^\dagger +S)m_3-m_2(m_0^\dagger +S)m_2-m_3(m_0^\dagger +S)m_1
\\&\phantom{{}=S(({}}{}
+m_1(m_0^\dagger +S)m_1(m_0^\dagger +S)m_2
\\&\phantom{{}=S(({}}{}
+m_1(m_0^\dagger +S)m_2(m_0^\dagger +S)m_1
\\&\phantom{{}=S(({}}{}
+m_2(m_0^\dagger +S)m_1(m_0^\dagger +S)m_1
\\&\phantom{{}=S(({}}{}
-m_1(m_0^\dagger +S)m_1(m_0^\dagger +S)m_1(m_0^\dagger +S)m_1\bigr]S
.
\end{split}
\end{align}
Here we note that the leading operator $q_0$ has to be invertible
in $\mathcal B(S\mathcal K)$.
Otherwise, applying \cite[Proposition~A.2]{IJ2} once again,
we can show that $R(\kappa)$ has a singularity
of order $\kappa^{-j}$, $j\ge 3$ in its expansion.
This contradicts the general estimate \eqref{170829}.
Hence we can use the Neumann series 
to write $q(\kappa)^\dagger$, and obtain
\begin{equation}
q(\kappa)^\dagger=\sum_{j=0}^{\beta-4}\kappa^jA_j+\mathcal{O}(\kappa^{\beta-3}),\quad
A_j\in\mathcal B(S\mathcal K),
\label{ex1-m-expandbb}
\end{equation}
where 
\begin{align}
\begin{split}
A_0&=q_0^{\dagger},\\
A_1&=-q_0^\dagger q_1q_0^\dagger,\\
A_2&=q_0^\dagger \bigl(-q_2+ q_1q_0^\dagger q_1\bigr)q_0^\dagger,\\
A_3&=q_0^\dagger \bigl(-q_3
+ q_2q_0^\dagger q_1
+ q_1q_0^\dagger q_2
- q_1q_0^\dagger q_1q_0^\dagger q_1\bigr)q_0^\dagger 
.
\end{split}
\label{17082314}
\end{align}
We also write $(m(\kappa)+S)^\dagger$ 
employing the Neumann series as 
\begin{align}
(m(\kappa)+S)^\dagger
=\sum_{j=0}^{\beta-3}\kappa^jC_j+\mathcal{O}(\kappa^{\beta-2}),\quad
C_j\in\mathcal B(Q\mathcal K)
\label{1608225b}
\end{align}
with
\begin{align}
\begin{split}
C_0&=m_0^\dagger +S,\\
C_1&=-(m_0^\dagger +S) m_1(m_0^\dagger +S),\\
C_2&=(m_0^\dagger +S) \bigl[-m_2+m_1(m_0^\dagger +S)m_1\bigr](m_0^\dagger +S),\\
C_3&=
(m_0^\dagger +S) \bigl[-m_3
+m_1(m_0^\dagger +S)m_2
+m_2(m_0^\dagger +S)m_1
\\&\phantom{{}={}(m_0^\dagger +S) \bigl[{}}{}
-m_1(m_0^\dagger +S)m_1(m_0^\dagger +S)m_1\bigr](m_0^\dagger +S)
.
\end{split}
\label{1708231451}
\end{align}
We first insert the expansions 
\eqref{ex1-m-expandbb} and \eqref{1608225b} into \eqref{16082255b}:
\begin{align}
m(\kappa)^\dagger
=\sum_{j=-1}^{\beta-5}\kappa^jD_j+\mathcal O(\kappa^{\beta-4});
\quad
D_j=C_j
+\sum_{\genfrac{}{}{0pt}{}{j_1\ge 0,j_2\ge 0,j_3\ge 0}{j_1+j_2+j_3=j+1}}
C_{j_1}A_{j_2}C_{j_3}
\label{160822553bb}
\end{align}
with $C_{-1}=0$.
Next, noting that we have an expansion of 
$(M(\kappa)+Q)^{-1}$ of the same form as \eqref{1608225},
we insert the expansions \eqref{160822553bb} and \eqref{1608225} 
into \eqref{16082255}:
\begin{align}
M(\kappa)^{-1}
=\sum_{j=-2}^{\beta-6}\kappa^jE_j+\mathcal O(\kappa^{\beta-5});
\quad
E_j=B_j
+\sum_{\genfrac{}{}{0pt}{}{j_1\ge 0,j_2\ge -1,j_3\ge 0}{j_1+j_2+j_3=j+1}}
B_{j_1}D_{j_2}B_{j_3}
\label{160822553bbb}
\end{align}
with $B_{-2}=B_{-1}=0$. We finally 
insert the expansions \eqref{free-expan} with $N=\beta-4$ and \eqref{160822553bbb}
into \eqref{second-resolvent},
and then obtain the expansion
\begin{align}
R(\kappa)&=\sum_{j=-2}^{\beta-6}\kappa^j
G_j+\mathcal O(\kappa^{\beta-5});\quad
G_j=G_{0,j}
-\sum_{\genfrac{}{}{0pt}{}{j_1\ge 0,j_2\ge -2,j_3\ge 0}{j_1+j_2+j_3=j}}
G_{0,j_1}vE_{j_2}v^*G_{0,j_3}
\label{17082214}
\end{align}
with $G_{0,-2}=G_{0,-1}=0$. Hence the expansion \eqref{expand-third} is obtained.

Next we compute the coefficients $G_{-2}$ and $G_{-1}$.
Let us start with $G_{-2}$.
Unfolding the above expressions, we can see with ease that 
\begin{align*}
G_{-2}
&=-G_{0,0}vE_{-2}v^*G_{0,0}\\
&=-G_{0,0}vq_0^\dagger v^*G_{0,0}\\
&=-G_{0,0}v\bigl(SM_2S-SM_1(M_0^\dagger +S)M_1S\bigr)^\dagger v^*G_{0,0}
.
\end{align*}
By $Sm_0=0$ and the expression \eqref{160817219} of $m_0$ it follows that
\begin{align}
Sv^*\mathbf n^{(\alpha)}=SQv^*\mathbf n^{(\alpha)}=0\quad
\text{for }\alpha=1,\dots,N.
\label{16082320}
\end{align}
Hence we have 
\begin{align*}
G_{-2}
=-G_{0,0}v(Sv^*G_{0,2}vS)^\dagger v^*G_{0,0}
.
\end{align*}
Then we can verify $G_{-2}=\mathsf P$ in exactly the same manner as in the 
proof of Theorem~\ref{thm-ex2}. We obtain the former identity of \eqref{G-1}.

As for $G_{-1}$, it requires a slightly longer computations,
and we proceed step by step.
We can first write, using $A_*,B_*,C_*,D_*,E_*$ only, 
\begin{align*}
G_{-1}
&=
-G_{0,0}vE_{-1}v^*G_{0,0}
-G_{0,0}vE_{-2}v^*G_{0,1}
-G_{0,1}vE_{-2}v^*G_{0,0}
\\&=
-G_{0,0}v\Bigl(B_0\bigl(C_0+C_0A_1C_0+C_0A_0C_1+C_1A_0C_0\bigr)B_0
\\&\phantom{{}=-G_{0,0}v\Bigl(}{}
+B_0C_0A_0C_0B_1+B_1C_0A_0C_0B_0\Bigr)v^*G_{0,0}
\\&\phantom{{}={}}{}
-G_{0,0}vB_0C_0A_0C_0B_0v^*G_{0,1}
-G_{0,1}vB_0C_0A_0C_0B_0v^*G_{0,0}.
\end{align*}
Next, we implement the identities $B_0C_*=C_*B_0=C_*$ and $C_0A_*=A_*C_0=A_*$,
insert expressions \eqref{17082314}, \eqref{1708212}, \eqref{1708231451} of $A_*,B_*,C_*$, 
and then use \eqref{16082320}:
\begin{align*}
G_{-1}
&=
-G_{0,0}v\Bigl(C_0+A_1+A_0C_1+C_1A_0
+A_0B_1+B_1A_0\Bigr)v^*G_{0,0}
\\&\phantom{{}={}}{}
-G_{0,0}vA_0v^*G_{0,1}
-G_{0,1}vA_0v^*G_{0,0}
\\&
=
-G_{0,0}v\Bigl(m_0^\dagger +S-q_0^\dagger q_1q_0^\dagger
-q_0^\dagger  m_1(m_0^\dagger +S)-(m_0^\dagger +S) m_1q_0^\dagger
\\&\phantom{{}=-G_{0,0}v\Bigl(}{}
-q_0^\dagger M_1(M_0^\dagger +Q)
-(M_0^\dagger +Q) M_1q_0^\dagger \Bigr)v^*G_{0,0}
\\&\phantom{{}={}}{}
-G_{0,0}vq_0^\dagger v^*G_{0,1}
-G_{0,1}vq_0^\dagger v^*G_{0,0}
\\&
=
-G_{0,0}v\Bigl(m_0^\dagger +S-q_0^\dagger q_1q_0^\dagger
-q_0^\dagger  m_1(m_0^\dagger +S)-(m_0^\dagger +S) m_1q_0^\dagger
 \Bigr)v^*G_{0,0}.
\end{align*}
We further unfold $q_1$ and $m_1$ by \eqref{am0bb} and \eqref{1608231744}, 
and then use \eqref{16082320}:
\begin{align*}
G_{-1}
&=
-G_{0,0}v\Bigl(m_0^\dagger +S+q_0^\dagger M_2(m_0^\dagger +S)M_2q_0^\dagger
\\&\phantom{{}=-G_{0,0}v\Bigl(}{}
-q_0^\dagger  M_2(m_0^\dagger +S)
-(m_0^\dagger +S) M_2q_0^\dagger
 \Bigr)v^*G_{0,0}
\\&
=
-G_{0,0}v(I-q_0^\dagger M_2)m_0^\dagger (I-M_2q_0^\dagger)v^*G_{0,0}
\\&\phantom{{}={}}{}
-G_{0,0}v(I-q_0^\dagger M_2)S(I-M_2q_0^\dagger)v^*G_{0,0}.
\end{align*}
Since $SM_2S=Sm_1S=q_0$ by \eqref{16082320}, the last term can actually be removed.
We also use $M_2=v^*G_{0,2}$, $q_0^\dagger=-Uv^*\mathsf PvU$ and Lemma~\ref{lemma24},
and then obtain 
\begin{align}
G_{-1}
=
-(I-\mathsf P)G_{0,0}vm_0^\dagger v^*G_{0,0}(I-\mathsf P).
\label{1708182228}
\end{align}
By the assumption and Proposition~\ref{13.1.16.2.51}
we can find an orthogonal basis $\{\Phi_\gamma\}\subset Q\mathcal K\cap (S\mathcal K)^\perp$
such that 
\begin{align}
m_0^\dagger =-\sum_\gamma |\Phi_\gamma\rangle\langle\Phi_\gamma|,
\label{17082422}
\end{align}
so that we can write 
\begin{align*}
G_{-1}
=\sum_\gamma |\Psi_\gamma\rangle\langle \Psi_\gamma|;\quad 
\Psi_\gamma=-(I-\mathsf P)G_{0,0}v\Phi_\gamma
\in \mathcal E.
\end{align*}
Let us finally verify that $\{\Psi_\gamma\}\subset\mathcal E$ 
is an orthonormal resonance basis. 
Since $\{\Phi_\gamma\}\subset Q\mathcal K\cap (S\mathcal K)^\perp$ is a basis,
the set $\{[\Psi_\gamma]\}=\{[z\Phi_\gamma]\}$ of representatives forms a basis in $\mathcal E/\mathsf E$
due to Proposition~\ref{13.1.16.2.51}.
It is clear that $\{\Psi_\gamma\}$ is orthogonal to $\mathsf E$.
In addition, similarly to \eqref{17081214},
we have
\begin{align}
\delta_{\gamma\gamma'}
&
=
\|\Phi_\gamma\|^{-2} \bigl\langle \Phi_{\gamma'},m_0m_0^\dagger\Phi_\gamma\bigr\rangle
=\sum_{\alpha=1}^N\bigl\langle V\Psi_{\gamma'},\mathbf n^{(\alpha)}\bigr\rangle
\bigl\langle \mathbf n^{(\alpha)},V\Psi_\gamma\bigr\rangle.
\label{17082423}
\end{align}
Then by Lemma~\ref{12.11.24.18.24} the resonance basis 
$\{\Psi_\gamma\}\subset\mathcal E$ is orthonormal.
Hence we obtain the latter assertion of \eqref{G-1}, and we are done.
\end{proof}

\appendix

\section{Examples and applications}\label{17080222}

\subsection{Criterion for Assumption~\ref{assumV}}

Here we provide a criterion for Assumption~\ref{assumV}
in terms of weighted $\ell^2$-spaces. 
One application of this criterion is 
an abstract verification for $J=-\Delta_G-H_0$ to satisfy Assumption~\ref{assumV} 
particularly when we change a free operator, e.g., to \eqref{1709041306}.
We also refer to \cite[Appendix~B]{IJ1} for another criterion 
for Assumption~\ref{assumV}.

Here we let $(G,E_G)$ be a graph with rays as in Section~\ref{17081312},
and let us use the standard 
weighted space notation such as $\ell^{2,s}(G)$.

\begin{proposition}\label{170904}
Assume that $V\in\mathcal B(\mathcal H)$ is self-adjoint,
and that it extends to an operator in 
$\mathcal B\bigl(\ell^{2,-\beta-1/2-\epsilon}(G),\ell^{2,\beta+1/2+\epsilon}(G)\bigr)$ 
for some $\beta\ge 1$ and $\epsilon>0$.
Then $V$ satisfies Assumption~\ref{assumV} for the same 
$\beta$.
\end{proposition}
\begin{proof}
Let $$\widetilde{V}=\langle x\rangle^{\beta+1/2+\epsilon/2}
V\langle x\rangle^{\beta+1/2+\epsilon/2},$$
where 
$\langle x\rangle=1$ for $x\in K$ and $\langle x\rangle=(1+n^2)^{1/2}$
for $x=n\in L_\alpha$. Then $\widetilde{V}$ is a compact self-adjoint operator on $\mathcal H=\ell^2(G)$, and hence we can find a
bounded and at most countable set $\{\lambda_j\}\subset\mathbb R$
of its eigenvalues, possibly accumulating only on $0$,
and an orthonormal system  $\{u_j\}\subset \mathcal H$ 
of the associated eigenfunctions such that 
$$\widetilde V=\sum_{j}\lambda_j|u_j\rangle\langle u_j|
.$$ 
Let $P$ denote the orthogonal projection onto 
$\mathcal K=(\mathop{\mathrm{Ker}} \widetilde{V})^{\perp}$.
We view $P$ as a map from $\mathcal H$ to $\mathcal K$, 
and $P^{\ast}$ as a map from $\mathcal K$ to $\mathcal H$. 
Set 
\begin{align*}
v&=\langle x\rangle^{-\beta-1/2-\epsilon/2}
\biggl(\sum_{j}|\lambda_j|^{1/2}|u_j\rangle\langle u_j|\biggr)P^{\ast},
\\
U&=P\biggl(\sum_{j}(\mathop{\mathrm{sgn}}\lambda_j)|u_j\rangle\langle u_j|\biggr)P^{\ast}.
\end{align*}
Then it follows that 
$v\in\mathcal B(\mathcal K,\ell^{2,\beta+1/2+\epsilon/2}(G))$ is injective,
that $U\in\mathcal B(\mathcal K)$ is self-adjoint and unitary,
and that 
$$V=vUv^*.$$
Since $\ell^{2,\beta+1/2+\epsilon/2}(G)$ can be continuously embedded into 
$\ell^{1,\beta}(G)$,
we obtain the assertion.
\end{proof}

\subsection{Infinite star graph}\label{17081323}

Here we consider an infinite star graph,
and see how the free Laplacian on it fits into the setting of this paper.

Under the notation of Section~\ref{17081312}
let us consider the case where the graph $(K,E_0)$ is a single point without edges:
$K=\{0\}$, $E_0=\emptyset$.
Then the graph $G$ defined by \eqref{17081313} with $K=\{0\}$ is called 
an \textit{infinite star graph}. 
Note that it coincides with the discrete half-lines $\mathbb N$ if $N=1$,
and with the discrete full line $\mathbb Z$ if $N=2$.
The Laplacian $-\Delta_G$ and the free operator $H_0$ 
on the infinite star graph $G$ are written as,
for any function $u\colon G\to\mathbb C$, 
\begin{align*}
(-\Delta_G u)[x]
&=
\left\{
\begin{array}{ll}
Nu[0]-u[1^{(1)}]-\dots-u[1^{(N)}]&\text{for }x=0\in K,\\
2u[n]-u[n+1]-u[n-1]&\text{for }x=n\in L_\alpha,
\end{array}
\right.
\intertext{and}
(H_0 u)[x]
&=
\left\{
\begin{array}{ll}
Nu[0]&\text{for }x=0\in K,\\
2u[1]-u[2]&\text{for }x=1\in L_\alpha,\\
2u[n]-u[n+1]-u[n-1]&\text{for }x=2,3,\ldots\in L_\alpha,
\end{array}
\right.
\end{align*}
respectively.
Then obviously by \eqref{170729} we have 
\begin{align*}
-\Delta_G =H_0+J,\quad
J=-\sum_{\alpha=1}^N\Bigl( |s\rangle\langle f_\alpha|
+|f_\alpha\rangle\langle s|\Bigr),
\end{align*}
where $s[x]=1$ if $x=0$ and $s[x]=0$ otherwise,
and $f_\alpha[x]=1$ if $x=1^{(\alpha)}$ and $f_\alpha[x]=0$ otherwise.

Let us see that the above jointing operator $J$ actually satisfies 
Assumption~\ref{assumV}.
In fact, we can choose $\mathcal K=\mathbb C^2$ and, in the corresponding matrix representations, 
$$
v=\left[\begin{array}{cc}|s\rangle&|f_1\rangle +\dots+ |f_N\rangle\end{array}\right]
,\quad
U=\left[\begin{array}{cc}0&-1\\-1&0\end{array}\right]
.
$$
Note that 
\begin{align}
v^*
=
\left[\begin{array}{c}\langle s|\\\langle f_1|+\dots+\langle f_N|\end{array}\right],
\label{17081314}
\end{align}
and hence we have the factorization
\begin{align}
J=vUv^*.
\label{17081315}
\end{align}
  
\begin{proposition}
The Laplacian $-\Delta_G$ on an infinite star graph 
has eigenspaces 
\begin{align*}
\widetilde{\mathcal E}
&=
\bigl\{u\in \mathbb C\mathbf n;\ \bigl\langle (f_1+\dots+f_N),u\bigr\rangle =0\bigr\}
+\mathcal E,\\ 
\mathcal E
&=\mathbb C\bigl(s+\mathbf 1^{(1)}+\dots+\mathbf 1^{(N)} \bigr),\\
\mathsf E&=\{0\}.
\end{align*}
In particular, the threshold $0$ is an exceptional point of the first kind.
\end{proposition}
\begin{remark*}
For $N=1$ the above Laplacian $-\Delta_G$ corresponds to the Neumann Laplacian 
on $\mathbb N$, cf.\ \cite[Section~I\hspace{-.1em}V]{IJ2}.
\end{remark*}
\begin{proof}
By the formulas \eqref{160817218}, \eqref{17081314}, \eqref{17081315}, \eqref{170803} and \eqref{G00}
we have 
\begin{align*}
M_0=\left[
\begin{array}{cc}
1/N&-1\\
-1& N
\end{array}
\right].
\end{align*}
Hence with $P$ from \eqref{17080719} we obtain
\begin{align*}
P\mathcal K=\mathbb C \left[
\begin{array}{c}
0\\ 1
\end{array}
\right],\quad
Q\mathcal K
=\mathbb C \left[
\begin{array}{c}
N\\ 1
\end{array}
\right].
\end{align*}
Then by Proposition~\ref{13.1.16.2.51} the assertions follow.
\end{proof}

\subsection{Analytic family persistently with resonance}

Next we present an example of a family of perturbed Schr\"odinger operators
that persistently have a zero resonance.

For simplicity let us discuss on the discrete line $\mathbb Z$.
Using notation from Appendix~\ref{17081323} with $N=2$,
we consider a two-parameter family of the Schr\"odinger operators
\begin{align*}
H_{E,\tau}=H_0+V_{E,\tau};\quad V_{E,\tau}=(E-2)|s\rangle\langle s|+\tau J,\quad E,\tau\in\mathbb R.
\end{align*}
If $(E,\tau)=(2,0)$, then $H_{2,0}=H_0$, and we already know that 
it does not have a resonance. We may let $(E,\tau)\neq (2,0)$.

\begin{proposition}
The operator $H_{E,\tau}$ has a zero resonance if and only if $2\tau^2-E=0$.
In particular, $\{H_{2(\alpha+1)^2,\alpha+1}\}_{|\alpha|<1}$ is a real analytic family of operators
each of which has a zero resonance.
\end{proposition}
\begin{proof}
Let us first verify that $V_{E,\tau}$ satisfies Assumption~\ref{assumV}. 
We can write 
\begin{align*}
V_{E,\tau}=\widetilde v\widetilde U\widetilde v^*
\end{align*}
with 
$$
\mathcal K=\mathbb C^2,\quad
\widetilde v=\left[\begin{array}{cc}|s\rangle&|f_1\rangle + |f_2\rangle\end{array}\right]
,\quad
\widetilde U=\left[\begin{array}{cc}E-2&-\tau\\-\tau&0\end{array}\right]
,
$$
but this decomposition does not satisfy Assumption~\ref{assumV},
since $\widetilde U$ is not unitary.
Let us rewrite it as follows: 
We can find an orthogonal matrix $Z$ such that 
\begin{align*}
Z^{-1}\widetilde U Z=\left[\begin{array}{cc}\alpha_+&0\\0&\alpha_-\end{array}\right];\quad
\alpha_\pm=E-2\pm \sqrt{(E-2)^2+4\tau^2}.
\end{align*}
Then we have 
\begin{align*}
U:=A^{-1}Z^{-1}\widetilde U ZA^{-1}=\left[\begin{array}{cc}1&0\\0&-1\end{array}\right];\quad
A=\left[\begin{array}{cc}\alpha_+^{1/2}&0\\0&(-\alpha_-)^{1/2}\end{array}\right],
\end{align*}
so that we can write 
\begin{align*}
V_{E,\tau}=vU v^*;\quad v=\widetilde vZA.
\end{align*}
Hence the perturbation $V_{E,\tau}$ certainly satisfies Assumption~\ref{assumV}.

Now let us investigate $M_0$.
Instead of $M_0$ itself, we compute 
\begin{align*}
UM_0U
&=
U+Uv^*G_{0,0}vU
\\&
=A^{-1}Z^*\widetilde U
\bigl(I
+\widetilde v^*G_{0,0}\widetilde v \widetilde U\bigr)ZA^{-1}
\\&
=A^{-1}Z^*
\widetilde U
\left[\begin{array}{cc}\tfrac12 E&-\tfrac12\tau\\-2\tau&1\end{array}\right]ZA^{-1}.
\end{align*}
Since $\widetilde U$, $Z$, $A$, and $U$ are invertible,
the operator $M_0$ is non-invertible if and only if 
$$
\det \left[\begin{array}{cc}\tfrac12 E&-\tfrac12\tau\\-2\tau&1\end{array}\right]
=\tfrac12E-\tau^2=0.
$$
Hence by Proposition~\ref{13.1.16.2.51}
the resonance eigenspace $\mathcal E$ is non-trivial 
if and only if $2\tau^2-E=0$.
On the other hand, by solving the difference equation $H_{E,\tau}u=0$ we 
can directly show that $\mathsf E=\{0\}$.
Hence the assertion follows.
\end{proof}

\subsection{Resonant dimension for the free Laplacian}

Let $(G,E_G)$ be a graph with rays,
constructed by jointing the discrete half-line $(L_\alpha,E_\alpha)$
to a finite graph $(K,E_0)$ at the vertex $x_\alpha$
for $\alpha=1,\dots,N$, as in Section~\ref{17081312}.
With the same notation there we consider the free Laplacian
\begin{align*}
H=-\Delta_G=H_0+J,\quad
J=-\sum_{\alpha=1}^N\Bigl( |s_\alpha\rangle\langle f_\alpha|
+|f_\alpha\rangle\langle s_\alpha|\Bigr).
\end{align*}
Then we can count the dimension of eigenspaces as follows:

\begin{proposition}
For the above operator $H$ one has
\begin{align}
\dim\widetilde{\mathcal E}
&=N-\dim\mathcal E,
\label{17090113}
\\
\dim\mathcal E
&=\#\bigl\{x_\alpha;\ \langle s_\alpha,g_{0,0}s_\alpha\rangle
=[\#\{\beta;\ x_\beta=x_\alpha\}]^{-1}\bigr\},
\label{17090114}
\\
\dim\mathsf E&=0.
\label{17090115}
\end{align}
\end{proposition}
\begin{proof}
We first note that, similarly to Appendix~\ref{17081323}, we can easily 
see that the operator $J$ satisfies Assumption~\ref{assumV} for any $\beta\ge 1$.
Then, in particular, \eqref{17090113} follows immediately by Proposition~\ref{13.1.16.2.51}.

Next we let $u\in \mathsf E$.
By solving the second difference equation $\Delta_Gu=0$ on $L_\alpha$
with $u[x]\to 0$ as $x=n^{(\alpha)}\to \infty$
we can deduce that $u[x]=0$ for any $x\in L_\alpha$.
This implies that $u|_K$ is a zero eigenfunction for $-\Delta_K$.
Then $u[x]$ has to be $0$ for all $x\in K$, too. 
Hence we obtain $\mathsf E=\{0\}$, or \eqref{17090115}.

By Proposition~\ref{13.1.16.2.51} and \eqref{17090115}
we have $\dim\mathcal E=\dim Q\mathcal K=\dim\mathop{\mathrm{Ker}}M_0$.
The last dimension can be computed by writing down $M_0$ explicitly
as in Appendix~\ref{17081323}. Then we can see \eqref{17090114}.
We omit the detail.
\end{proof}

\subsection{Spider web}

The generality of Assumption~\ref{assumV} allows us to deal with  
interactions between rays that decay but do not vanish at infinity.
Here let us consider a graph like ``spider web'',
which could be considered as a discretized model of the spherical coordinates.

Let $G$ be the infinite star graph considered in Appendix~\ref{17081323}.
We choose a weight function $w\colon \mathbb N\to\mathbb R$ that satisfies 
for some $\beta\ge 1$
\begin{align*}
\sum_{n\in\mathbb N} (1+n^2)^\beta|w[n]|<\infty.
\end{align*}
Then we define a operator $H_w$ as, for any function $u\colon G\to\mathbb C$,
\begin{align*}
(H_w u)[n,\alpha]
&
=2u[n,\alpha]-u[n+1,\alpha]-u[n-1,\alpha]
\\&\phantom{{}={}}{}
+w[n]\bigl(2u[n,\alpha]-u[n,\alpha+1]-u[n,\alpha-1]\bigr)
\intertext{for $n\ge 1$, and }
(H_w u)[0]
&
=Nu[0]-u[1,1]-\dots-u[1,N].
\end{align*}
Here we let $n^{(\alpha)}=(n,\alpha)$, and interpret $u[0,\alpha]=u[0]$, 
$u[n,0]=u[n,N]$ and $u[n,N+1]=u[n,1]$. 

The operator $H_w$ may be considered as the Laplacian on a graph like ``spider web'',
and this actually fits into our setting. 
To see this we 
define the \textit{spherical discrete Fourier transform} $\mathcal F$ as 
\begin{align*}
(\mathcal F u)[n,k]&=N^{-1/2}\sum_{\alpha=1}^Ne^{-2\pi i\alpha k/N}u[n,\alpha]\quad \text{for }n\ge 1
\end{align*}
and $(\mathcal F u)[0]=u[0]$. 
Its inverse $\mathcal F^*$ is given by 
\begin{align*}
(\mathcal F^*v)[n,\alpha]&=N^{-1/2}\sum_{k=1}^Ne^{2\pi i\alpha k/N}v[n,k]\quad \text{for }n\ge 1
\end{align*}
and $(\mathcal F^* v)[0]=v[0]$.
Then we have 
\begin{align*}
(\mathcal FH_w u)[n,k]
&=
2(\mathcal F u)[n,k]-(\mathcal F u)[n+1,k]-(\mathcal F u)[n-1,k]
\\&\phantom{{}={}}{}
+w[n]\bigl(2-2\cos(2\pi k/N)\bigr)(\mathcal F u)[n,k]
\intertext{for $n\ge 2$,}
(\mathcal FH_w u)[1,k]
&=
2(\mathcal F u)[1,k]-(\mathcal F u)[2,k]
+w[1]\bigl(2-2\cos(2\pi k/N)\bigr)(\mathcal F u)[1,k]
\intertext{for $k\neq N$,}
(\mathcal FH_w u)[1,N]
&=
2(\mathcal F u)[1,N]-(\mathcal F u)[2,N]-N^{1/2}(\mathcal F u)[0]
\intertext{and }
(\mathcal FH_w u)[0]
&=
(\mathcal Fu)[0]-N^{1/2}(\mathcal Fu)[1,N].
\end{align*}
This implies that $\mathcal FH_w\mathcal F^*$ splits into a 
direct sum of a Robin Schr\"odinger 
and $(N-1)$ number of Dirichlet Schr\"odinger operators on $\mathbb N$. 
As for a discrete Robin Schr\"odinger operator see \cite[Section~I\hspace{-.1em}V]{IJ2}.
This obviously satisfies Assumption~\ref{assumV}.

\section{Zeroth and first order coefficients}\label{170818}

In this appendix we compute the coefficients $G_0$ and $G_1$.
It follows from the general theory that 
$G_0$ is a Green operator for $H$, see Corollary~\ref{17081919},
but here we compute an explicit expression for it. 
On the other hand, $G_1$ is intimately related to the 
\textit{orthogonal non-resonance projection}, but they do not exactly coincide.
This can also be seen from Corollary~\ref{17081919} particularly when $G_{-1}\neq 0$.

The main results of this appendix are stated in 
Theorems~\ref{170820205}, \ref{17082015}, \ref{170820200} and \ref{17082017}.
They are considered as  part of the main results of the paper. 
However, their expressions and proofs are quite lengthy, hence are separated from Section~\ref{1707291915}.
We remark that the coefficients $G_0$ and $G_1$ were missing in \cite{IJ2} 
for the exceptional cases of the first and third kinds,
but they are completed here in a more general setting on a graph with rays.

\subsection{Orthogonal non-resonance projection}
Let us first formulate the \textit{orthonormal non-resonance projection},
cf.\ Definition~\ref{170819}. Recall that the notation $\overline{\Psi}\cdot\Psi_\gamma$ is defined in the paragraph before Definition~\ref{170819}.
\begin{definition}
We call a subset $\{\Psi_\gamma\}_{\gamma}\subset \widetilde{\mathcal E}$ 
a \textit{non-resonance basis}, if 
the set $\{[\Psi_\gamma]\}_{\gamma}$ of representatives 
forms a basis in $\widetilde{\mathcal E}/\mathcal E$.
It is said to be \textit{orthonormal}, if 
\begin{enumerate}
\item for any $\gamma$ and $\Psi\in\mathsf E$ one has 
$\overline{\Psi}\cdot\Psi_\gamma\in \mathcal L^0$ and $\langle\Psi,\Psi_\gamma\rangle=0$;
\item there exists an orthonormal system 
$\{c^{(\gamma)}\}_\gamma\subset \mathbb C^N$
such that for any $\gamma$
$$\Psi_\gamma
-\sum_{\alpha=1}^Nc_\alpha^{(\gamma)}\mathbf n^{(\alpha)}\in \mathbb C\mathbf 1\oplus\mathcal L^{\beta-2}.
$$
\end{enumerate}
The \textit{orthogonal non-resonance projection} $\widetilde{\mathcal P}$ 
\textit{associated with} $\{\Psi_\gamma\}_{\gamma}$ is defined as 
\begin{align*}
\widetilde{\mathcal P}=\sum_{\gamma}|\Psi_\gamma\rangle\langle\Psi_\gamma|.
\end{align*}
\end{definition}

\begin{remarks*}
\ 

1.
We will later see that leading coefficients of orthonormal resonance and non-resonance bases 
form an orthonormal basis of $\mathbb C^N$,
cf.\ \eqref{1708221320}, Theorems~\ref{17082015} and \ref{17082017}. 

2.
In contrast to the orthonormal bound and resonance projections, 
the orthogonal non-resonance projection does depend on choice of an orthonormal non-resonance basis.
We have not formulated orthogonality between resonance and non-resonance bases,
which should not be confused with the above remark.
There is a freedom of choice of coefficients of $\mathbf 1^{(\alpha)}$ for an 
orthonormal non-resonance basis. 
An appropriate formulation is missing.
\end{remarks*}

\subsection{Regular threshold}
\begin{theorem}\label{170820205}
Under the assumptions of Theorem~\ref{thm-reg} one has 
\begin{align}
G_0&
=G_{0,0}-G_{0,0} vM_0^\dagger v^*G_{0,0},\label{160822427}
\\
G_1&=-\widetilde{\mathcal P},\label{160822428}
\end{align}
where $\widetilde{\mathcal P}$ is uniquely determined due to $\mathcal E=\{0\}$.
\end{theorem}
\begin{proof}
We discuss as continuing from the proof of Theorem~\ref{thm-reg}.
By \eqref{170820} and \eqref{12.12.1.7.18} we obviously have
\begin{align*}
G_0&=G_{0,0}-G_{0,0} vM_0^{-1}v^*G_{0,0}
,
\end{align*}
which verifies \eqref{160822427}.
The coefficient $G_1$ can be computed by \eqref{170820} and \eqref{12.12.1.7.18} as 
\begin{align*}
G_1
&=G_{0,1}
-G_{0,1}vM_0^{-1}v^*G_{0,0}
+G_{0,0}vM_0^{-1}M_1M_0^{-1}v^*G_{0,0}
-G_{0,0}vM_0^{-1}v^*G_{0,1}
\\&
=
\bigl(I-G_{0,0} vM_0^{-1}v^*\bigr)G_{0,1}\bigl(I-vM_0^{-1}v^*G_{0,0}\bigr)
.
\end{align*}
Then, recalling the expression of $G_{0,1}$ given by \eqref{170803} and \eqref{G01},
we can write 
\begin{align*}
G_1=-\sum_{\alpha=1}^N|\Psi_\alpha\rangle\langle\Psi_\alpha|;\quad
\Psi_\alpha=\bigl(I-G_{0,0} vM_0^{-1}v^*\bigr)\mathbf n^{(\alpha)}.
\end{align*}
By Lemma~\ref{12.11.24.18.24} we have 
$\Psi_\alpha-\mathbf n^{(\alpha)}\in\mathbb C\mathbf 1\oplus\mathcal L^{\beta-2}$,
so that it suffices to show that $\Psi_\alpha\in \widetilde{\mathcal E}$.
Let us decompose
\begin{align*}
\Psi_\alpha
=\mathbf n^{(\alpha)}
-\sum_{\gamma=1}^N
\bigl\langle \Phi^{(\gamma)},v^*\mathbf n^{(\alpha)}\bigr\rangle\Psi^{(\gamma)}
+zM_0^{-1}v^*\mathbf n^{(\alpha)}.
\end{align*}
Then by Proposition~\ref{13.1.16.2.51},
since $zM_0^{-1}v^*\mathbf n^{(\alpha)}\in \widetilde{\mathcal E}$, 
it suffices to show that 
\begin{align*}
\mathbf n^{(\alpha)}
-\sum_{\gamma=1}^N
\bigl\langle \Phi^{(\gamma)},v^*\mathbf n^{(\alpha)}\bigr\rangle\Psi^{(\gamma)}
\in
\mathbb C\mathbf n\cap\mathop{\mathrm{Ker}}V=\mathbb C\mathbf n\cap\mathop{\mathrm{Ker}}v^*,
\end{align*} 
but this is straightforward, cf.\ \eqref{17080719}. 
Hence the expression \eqref{160822428} is obtained.
\end{proof}

\subsection{Exceptional threshold of the first kind}

To state a theorem let us introduce some notation.
Let $\{\Psi_\gamma\}\subset \mathcal E$ be an orthogonal resonance basis.
Set $c^{(\gamma)}_\alpha=-\langle \mathbf n^{(\alpha)},V\Psi_\gamma\rangle$, 
and then $\{c^{(\gamma)}\}\subset\mathbb C^N$ 
is an orthonormal system by Proposition~\ref{13.1.16.2.51}.
We take an orthonormal system $\{c^{(\widetilde\gamma)}\}\subset\mathbb C^N$
such that $\{c^{(\gamma)}\}\cup\{c^{(\widetilde\gamma)}\}\subset\mathbb C^N$
is an orthonormal basis, and denote 
\begin{align}
\mathbf 1_\gamma=\sum_{\alpha=1}^Nc^{(\gamma)}_\alpha\mathbf 1^{(\alpha)},\quad 
\mathbf n_\gamma=\sum_{\alpha=1}^Nc^{(\gamma)}_\alpha\mathbf n^{(\alpha)},\quad
\mathbf n_{\widetilde\gamma}=\sum_{\alpha=1}^Nc^{(\widetilde\gamma)}_\alpha\mathbf n^{(\alpha)}.
\label{1708202319}
\end{align}

\begin{theorem}\label{17082015}
Under the assumptions of Theorem~\ref{thm-ex1} one has
\begin{align}
\begin{split}
G_0&
=G_{0,0}
-\mathcal PVG_{0,1}
-G_{0,1}V\mathcal P
\\&\phantom{{}={}}{}
-\bigl(G_{0,0}-\mathcal PVG_{0,1}\bigr)vM_0^\dagger v^*\bigl(G_{0,0}-G_{0,1}V\mathcal P\bigr)
+\mathcal PVG_{0,2}V\mathcal P
,
\end{split}
\label{16082513}
\\
\begin{split}
G_1
&=
-\widetilde{\mathcal P}
+\mathcal PVG_{0,3}V\mathcal P
+\mathcal PVG_{0,2}V\mathcal PVG_{0,2}V\mathcal P
\\&\phantom{{}={}}{}
-\bigl[I-\bigl(G_{0,0}-\mathcal PVG_{0,1}\bigr)vM_0^\dagger v^*\bigr]
\bigl(I+G_{0,1}V\mathcal PV\bigr)G_{0,2}V\mathcal P
\\&\phantom{{}={}}{}
-\mathcal PVG_{0,2}\bigl(I+V\mathcal PVG_{0,1}\bigr)
\bigl[I-vM_0^\dagger v^*\bigl(G_{0,0}-G_{0,1}V\mathcal P\bigr)\bigr]
,
\end{split}
\label{1708218}
\end{align}
where $\widetilde{\mathcal P}$ is an orthogonal non-resonance projection associated with 
the orthonormal non-resonance basis $\{\Psi_{\widetilde\gamma}\}_{\widetilde\gamma}$ consisting of 
\begin{align*}
\Psi_{\widetilde\gamma}
=\big[I-\bigl(G_{0,0}-\mathcal PV G_{0,1}\bigr) v M_0^\dagger v^*\bigr]\mathbf n_{\widetilde\gamma}
\end{align*}
with the above notation.
This in particular verifies the orthogonality \eqref{1708221320}.
\end{theorem}

Before the proof of Theorem~\ref{17082015}
let us remark that the above expressions 
\eqref{16082513} and \eqref{1708218} can be computed 
further if we make use of the following lemma.

\begin{lemma}\label{1708201554}
Under the assumptions of Theorem~\ref{thm-ex1} one has
\begin{align}
G_{0,0}V\mathcal P
=-\mathcal P
,\quad
G_{0,1}V\mathcal P
=\sum_\gamma|\mathbf n_\gamma\rangle\langle\Psi_\gamma|
,
\label{17082023}
\end{align}
and for $j=2,3,\ldots,\beta-4$
\begin{align}
G_{0,j}V\mathcal P
&=
\sum_\gamma 
\Bigl(
\bigl|G_{0,j-2}(\Psi_\gamma-\mathbf 1_\gamma)\bigr\rangle\langle\Psi_\gamma|
-|G_{0,j}H_0\mathbf 1_\gamma\rangle\langle\Psi_\gamma|\Bigr),
\label{1708202320}
\end{align}
where $\{\Psi_\alpha\}\subset\mathcal E$ is an orthonormal resonance basis.
\end{lemma}
\begin{proof}
The formulas in \eqref{17082023} follow by Proposition~\ref{13.1.16.2.51} and \eqref{1708202319}.
For the  proof of \eqref{1708202320}
it suffices to note that for any $u\in(\mathcal L^{\beta-2})^*$
\begin{align}
\langle u,H_0\Psi_\gamma\rangle
=
\bigl\langle H_0 u,(\Psi_\gamma-\mathbf 1_\gamma)\bigr\rangle
+\langle u,H_0 \mathbf 1_\gamma\rangle.
\label{1708202328}
\end{align}
We remark that, since $\Psi_\gamma$ and $u$ are not decaying at infinity,
we may apply a summation by parts only 
after we subtract the leading asymptotics of $\Psi_\gamma$ as above.
Let us omit the proof of \eqref{1708202328}.
Then \eqref{1708202320} follows immediately by 
\eqref{1708202328} and \eqref{18082023a}.
\end{proof}

\begin{remarks}\label{170826231}
\ 

1.
The formulas in \eqref{17082023}
may be considered as special cases 
of \eqref{1708202320}.

2.
Using Lemma~\ref{1708201554},
we can further compute, for example, as
\begin{align*}
G_0&
=G_{0,0}
-\sum_\gamma
\Bigl(
|\Psi_\gamma\rangle\langle\mathbf n_\gamma|+|\mathbf n_\gamma\rangle\langle\Psi_\gamma|
\Bigr)
\\&\phantom{{}={}}{}
-\Bigl(G_{0,0}-\sum_\gamma
|\Psi_\gamma\rangle\langle\mathbf n_\gamma|\Bigr)vM_0^\dagger v^*\Bigl(G_{0,0}-\sum_\gamma
|\mathbf n_\gamma\rangle\langle\Psi_\gamma|\Bigr)
\\&\phantom{{}={}}{}
+\sum_{\gamma,\gamma'} c_{\gamma\gamma'}|\Psi_\gamma\rangle\langle\Psi_{\gamma'}|
,
\end{align*}
where $\{\Psi_\alpha\}\subset\mathcal E$ is an orthonormal resonance basis and 
\begin{align*}
c_{\gamma\gamma'}
&=
\tfrac12\delta_{\gamma\gamma'}
-\bigl\langle(\Psi_\gamma-\mathbf 1_\gamma),(\Psi_{\gamma'}-\mathbf 1_{\gamma'})\bigr\rangle
-\bigl\langle v^*\mathbf n_\gamma,M_0^\dagger v^*\mathbf n_{\gamma'}\bigr\rangle
\\&\phantom{{}={}}{}
-\bigl\langle\mathbf 1_\gamma,(\Psi_{\gamma'}-\mathbf 1_{\gamma'})\bigr\rangle
-\bigl\langle(\Psi_\gamma-\mathbf 1_\gamma),\mathbf 1_{\gamma'}\bigr\rangle
.
\end{align*}
The corresponding expression for $G_1$ would be much more complicated.
We have not found any particular advantage of writing down $G_1$ in this form,
\end{remarks}

\begin{proof}[Proof of Theorem~\ref{17082015}]
The discussion here is a continuation from the proof of Theorem~\ref{thm-ex1}.
Using \eqref{170820133} and \eqref{160822553}
we first express $G_0$ in terms of $A_*$ and $B_*$,
and then insert expressions for them:
\begin{align*}
G_0
&
=
G_{0,0}
-G_{0,0}vC_{-1}v^*G_{0,1}
-G_{0,1}vC_{-1}v^*G_{0,0}
-G_{0,0}vC_0v^*G_{0,0}
\\&
=
G_{0,0}
-G_{0,0}vA_0v^*G_{0,1}
-G_{0,1}vA_0v^*G_{0,0}
\\&\phantom{{}={}}{}
-G_{0,0}v\bigl(
B_0
+B_0A_0B_1
+B_1A_0B_0
+B_0A_1B_0
\bigr)v^*G_{0,0}
\\&
= 
G_{0,0}
-G_{0,0}vm_0^\dagger v^*G_{0,1}
-G_{0,1}vm_0^\dagger v^*G_{0,0}
\\&\phantom{{}={}}{}
-G_{0,0}v\Bigl(
M_0^\dagger +Q
-m_0^\dagger M_1(M_0^\dagger +Q)
-(M_0^\dagger +Q) M_1m_0^\dagger
-m_0^\dagger m_1m_0^\dagger
\Bigr)v^*G_{0,0}.
\end{align*}
We expand the parentheses above, and unfold $m_1$ by \eqref{1608231744},
noting $m_0m_0^\dagger=m_0^\dagger m_0=Q$: 
\begin{align*}
G_0&
= 
G_{0,0}
-G_{0,0}vM_0^\dagger v^*G_{0,0}
+G_{0,0}vm_0^\dagger M_2m_0^\dagger v^*G_{0,0}
-G_{0,0}vm_0^\dagger M_1M_0^\dagger M_1m_0^\dagger v^*G_{0,0}
\\&\phantom{{}={}}{}
-G_{0,0}vm_0^\dagger v^*G_{0,1}
+G_{0,0}vm_0^\dagger M_1M_0^\dagger v^*G_{0,0}
\\&\phantom{{}={}}{}
-G_{0,1}vm_0^\dagger v^*G_{0,0}
+G_{0,0}vM_0^\dagger M_1m_0^\dagger v^*G_{0,0}
.
\end{align*}
Substitute the expressions $m_0^\dagger=-Uv^*\mathcal PvU$, which follows from \eqref{16082516}, 
and $M_j=v^*G_{0,j}v$ for $j=1,2$:
\begin{align*}
G_0
&
= 
G_{0,0}
-G_{0,0}vM_0^\dagger v^*G_{0,0}
+G_{0,0}V\mathcal PVG_{0,2}V\mathcal PVG_{0,0}
\\&\phantom{{}={}}{}
-G_{0,0}V\mathcal PVG_{0,1}vM_0^\dagger v^*G_{0,1}V\mathcal PVG_{0,0}
\\&\phantom{{}={}}{}
+G_{0,0}V\mathcal PVG_{0,1}
-G_{0,0}V\mathcal PVG_{0,1}vM_0^\dagger v^*G_{0,0}
\\&\phantom{{}={}}{}
+G_{0,1}V\mathcal PVG_{0,0}
-G_{0,0}vM_0^\dagger v^* G_{0,1}V\mathcal PVG_{0,0}
.
\end{align*}
Now we apply Lemma~\ref{1708201554}, and then we obtain \eqref{16082513}.

Next we compute $G_1$. 
By \eqref{170820133} and \eqref{160822553}
we first write it in terms only of $A_*$ and $B_*$,
noting also $B_0A_*=A_*B_0=A_*$,
\begin{align*}
G_1
&=G_{0,1}
-G_{0,1}vC_{-1}v^*G_{0,1}
-G_{0,0}vC_{-1}v^*G_{0,2}
-G_{0,2}vC_{-1}v^*G_{0,0}
\\&\phantom{{}={}}{}
-G_{0,0}vC_{0}v^*G_{0,1}
-G_{0,1}vC_{0}v^*G_{0,0}
-G_{0,0}vC_1v^*G_{0,0}
\\
&=G_{0,1}
-G_{0,1}vA_0v^*G_{0,1}
-G_{0,0}vA_0v^*G_{0,2}
-G_{0,2}vA_0v^*G_{0,0}
\\&\phantom{{}={}}{}
-G_{0,0}v(B_0+A_0B_1+B_1A_0+A_1)v^*G_{0,1}
\\&\phantom{{}={}}{}
-G_{0,1}v(B_0+A_0B_1+B_1A_0+A_1)v^*G_{0,0}
\\&\phantom{{}={}}{}
-G_{0,0}v(B_1+B_1A_0B_1+A_0B_2+B_2A_0+A_1B_1+B_1A_1+A_2)v^*G_{0,0}.
\end{align*}
Insert the expressions \eqref{170821} and \eqref{1708212} for $A_*$ and $B_*$,
and then use $QM_1Q=m_0$ and $m_0m_0^\dagger=m_0^\dagger m_0=Q$:
\begin{align*}
G_1
&=G_{0,1}
-G_{0,1}vm_0^\dagger v^*G_{0,1}
-G_{0,0}vm_0^\dagger v^*G_{0,2}
-G_{0,2}vm_0^\dagger v^*G_{0,0}
\\&\phantom{{}={}}{}
-G_{0,0}v
\Bigl(
M_0^\dagger +Q
-m_0^\dagger M_1(M_0^\dagger +Q)
-(M_0^\dagger +Q) M_1 m_0^\dagger
-m_0^\dagger m_1m_0^\dagger
\Bigr)v^*G_{0,1}
\\&\phantom{{}={}}{}
-G_{0,1}v
\Bigl(
M_0^\dagger +Q
-m_0^\dagger M_1(M_0^\dagger +Q)
-(M_0^\dagger +Q) M_1 m_0^\dagger
-m_0^\dagger m_1m_0^\dagger
\Bigr)v^*G_{0,0}
\\&\phantom{{}={}}{}
-G_{0,0}v
\Bigl(
-(M_0^\dagger +Q) M_1(M_0^\dagger +Q)
+(M_0^\dagger +Q) M_1 m_0^\dagger M_1(M_0^\dagger +Q)
\\&\phantom{{}={}-G_{0,0}v\Bigl(}{}
+m_0^\dagger \bigl[-M_2(M_0^\dagger +Q)+M_1(M_0^\dagger +Q) M_1(M_0^\dagger +Q)\bigr]
\\&\phantom{{}={}-G_{0,0}v\Bigl(}{}
+\bigl[-(M_0^\dagger +Q) M_2+(M_0^\dagger +Q) M_1(M_0^\dagger +Q) M_1\bigr]m_0^\dagger
\\&\phantom{{}={}-G_{0,0}v\Bigl(}{}
+m_0^\dagger m_1m_0^\dagger  M_1(M_0^\dagger +Q)
+(M_0^\dagger +Q) M_1m_0^\dagger m_1m_0^\dagger
\\&\phantom{{}={}-G_{0,0}v\Bigl(}{}
-m_0^\dagger m_2m_0^\dagger+m_0^\dagger m_1m_0^\dagger m_1m_0^\dagger 
\Bigr)v^*G_{0,0}
\\&
=
G_{0,1}
-G_{0,1}vm_0^\dagger v^*G_{0,1}
-G_{0,0}vm_0^\dagger v^*G_{0,2}
-G_{0,2}vm_0^\dagger v^*G_{0,0}
\\&\phantom{{}={}}{}
-G_{0,0}v
\Bigl(
M_0^\dagger -Q
-m_0^\dagger M_1M_0^\dagger
-M_0^\dagger M_1 m_0^\dagger
-m_0^\dagger m_1m_0^\dagger
\Bigr)v^*G_{0,1}
\\&\phantom{{}={}}{}
-G_{0,1}v
\Bigl(
M_0^\dagger -Q
-m_0^\dagger M_1M_0^\dagger 
-M_0^\dagger M_1 m_0^\dagger
-m_0^\dagger m_1m_0^\dagger
\Bigr)v^*G_{0,0}
\\&\phantom{{}={}}{}
-G_{0,0}v
\Bigl(
2m_0
-M_0^\dagger  M_1M_0^\dagger 
+M_0^\dagger  M_1 Q
+Q M_1M_0^\dagger 
+M_0^\dagger  M_1 m_0^\dagger M_1M_0^\dagger
\\&\phantom{{}={}-G_{0,0}v\Bigl(}{}
-m_0^\dagger M_2M_0^\dagger
-m_0^\dagger M_2 Q
+m_0^\dagger M_1M_0^\dagger  M_1M_0^\dagger 
+m_0^\dagger M_1M_0^\dagger  M_1Q
\\&\phantom{{}={}-G_{0,0}v\Bigl(}{}
-M_0^\dagger M_2m_0^\dagger
-Q M_2m_0^\dagger
+M_0^\dagger  M_1M_0^\dagger M_1m_0^\dagger
+Q M_1M_0^\dagger  M_1m_0^\dagger
\\&\phantom{{}={}-G_{0,0}v\Bigl(}{}
+m_0^\dagger m_1m_0^\dagger  M_1M_0^\dagger
+m_0^\dagger m_1m_0^\dagger  M_1 Q
\\&\phantom{{}={}-G_{0,0}v\Bigl(}{}
+M_0^\dagger  M_1m_0^\dagger m_1m_0^\dagger
+Q M_1m_0^\dagger m_1m_0^\dagger
\\&\phantom{{}={}-G_{0,0}v\Bigl(}{}
-m_0^\dagger m_2m_0^\dagger+m_0^\dagger m_1m_0^\dagger m_1m_0^\dagger 
\Bigr)v^*G_{0,0}.
\end{align*}
Next we insert the expressions \eqref{1608231744} and 
\eqref{1608231745} for $m_1$ and $m_2$, and then 
use $M_j=v^*G_{0,j}v$ for $j=1,2,3$.
After some computations we obtain
\begin{align*}
G_1
&=
\big(I-G_{0,0}vM_0^\dagger v^*+G_{0,0}vm_0^\dagger v^*G_{0,1}vM_0^\dagger v^*\big) 
\\&\phantom{{}={}}{}
\cdot \bigl(G_{0,1}-G_{0,1}vm_0^\dagger v^*G_{0,1}\bigr)
\bigl(I-vM_0^\dagger v^*G_{0,0}+vM_0^\dagger v^*G_{0,1}v m_0^\dagger v^*G_{0,0}\bigr)
\\&\phantom{{}={}}{}
+G_{0,0}vm_0^\dagger v^*G_{0,3}vm_0^\dagger v^*G_{0,0}
-G_{0,0}vm_0^\dagger v^*G_{0,2}vm_0^\dagger v^*G_{0,2}vm_0^\dagger v^*G_{0,0}
\\&\phantom{{}={}}{}
-\bigl(I-G_{0,0}vM_0^\dagger v^*+G_{0,0}vm_0^\dagger v^*G_{0,1}vM_0^\dagger v^*\bigr)
\bigl(I-G_{0,1}vm_0^\dagger v^*\bigr)G_{0,2}vm_0^\dagger v^*G_{0,0}
\\&\phantom{{}={}}{}
-G_{0,0}vm_0^\dagger v^*G_{0,2}\bigl(I-vm_0^\dagger v^*G_{0,1}\bigr)
\bigl(I-vM_0^\dagger v^*G_{0,0}+vM_0^\dagger v^*G_{0,1}vm_0^\dagger v^*G_{0,0}\bigr)
.
\end{align*}
Now we use $m_0^\dagger=-Uv^*\mathcal PvU$ with Lemma~\ref{1708201554}.
It follows that 
\begin{align*}
G_1
&=
\big(I-G_{0,0}v M_0^\dagger v^*+\mathcal PV G_{0,1} v M_0^\dagger v^*\bigr)
\\&\phantom{{}={}}{}
\cdot
\bigl(G_{0,1}+G_{0,1}V\mathcal PVG_{0,1}\bigr)
\bigl(I-vM_0^\dagger v^*G_{0,0}+vM_0^\dagger v^*G_{0,1}V\mathcal P\bigr)
\\&\phantom{{}={}}{}
+\mathcal PVG_{0,3}V\mathcal P
+\mathcal PVG_{0,2}V\mathcal PVG_{0,2}V\mathcal P
\\&\phantom{{}={}}{}
-\bigl(I-G_{0,0}vM_0^\dagger v^*+\mathcal PVG_{0,1}vM_0^\dagger v^*\bigr)
\bigl(I+G_{0,1}V\mathcal PV\bigr)G_{0,2}V\mathcal P
\\&\phantom{{}={}}{}
-\mathcal PVG_{0,2}\bigl(I+V\mathcal PVG_{0,1}\bigr)
\bigl(I-vM_0^\dagger v^*G_{0,0}+vM_0^\dagger v^*G_{0,1}V\mathcal P\bigr)
.
\end{align*}

Finally let us set 
\begin{align*}
T&=\big(I-G_{0,0}v M_0^\dagger v^*+\mathcal PV G_{0,1} v M_0^\dagger v^*\bigr)
\\&\phantom{{}={}}{}
\cdot
\bigl(G_{0,1}+G_{0,1}V\mathcal PVG_{0,1}\bigr)
\bigl(I-vM_0^\dagger v^*G_{0,0}+vM_0^\dagger v^*G_{0,1}V\mathcal P\bigr),
\end{align*}
for short.
Let $\{\Psi_\gamma\}\subset\mathcal E$ be an orthonormal resonance basis,
and define $\mathbf n_\gamma$ and $\mathbf n_{\widetilde\gamma}$ as in \eqref{1708202319}.
By Lemma~\ref{1708201554} we have 
\begin{align*}
G_{0,1}+G_{0,1}V\mathcal PVG_{0,1}
=-\sum_{\alpha=1}^N|\mathbf n^{(\alpha)}\rangle \langle \mathbf n^{(\alpha)}|
+\sum_{\gamma}|\mathbf n_\gamma\rangle \langle \mathbf n_\gamma|
=-\sum_{\widetilde\gamma}|\mathbf n_{\widetilde\gamma}\rangle \langle \mathbf n_{\widetilde\gamma}|,
\end{align*}
so that we can write 
\begin{align*}
T
=-\sum_{\widetilde\gamma}|\Psi_{\widetilde\gamma}\rangle \langle \Psi_{\widetilde\gamma}|
;\quad
\Psi_{\widetilde\gamma}
=\big(I-G_{0,0}v M_0^\dagger v^*+\mathcal PV G_{0,1} v M_0^\dagger v^*\bigr)\mathbf n_{\widetilde\gamma}.
\end{align*}
Hence it suffices to show that $\{\Psi_{\widetilde\gamma}\}$ is an orthonormal non-resonance basis.
Since 
\begin{align*}
Qv^*\bigl(G_{0,1}+G_{0,1}V\mathcal PVG_{0,1}\bigr)
=
Qv^*\bigl(I-G_{0,1}vm_0^\dagger v^*\bigr)G_{0,1}
=0,
\end{align*}
we have $Qv^*\mathbf n_{\widetilde\gamma}=0$,
and hence  
\begin{align*}
H\Psi_{\widetilde\gamma}
=\big(V-v M_0^\dagger v^*-VG_{0,0}v M_0^\dagger v^*\bigr)
\mathbf n_{\widetilde\gamma}
=vUQv^*\mathbf n_{\widetilde\gamma}
=0.
\end{align*}
This implies that $\Psi_{\widetilde\gamma}\in \widetilde{\mathcal E}$.
On the other hand, by Lemma~\ref{12.11.24.18.24} we have 
\begin{align*}
\Psi_{\widetilde\gamma}
-
\mathbf n_{\widetilde\gamma}
\in \mathbb C\mathbf 1\oplus \mathcal L^{\beta-2}.
\end{align*}
Since $\mathsf E=\{0\}$ under the assumption,
it follows that $\{\Psi_{\widetilde\gamma}\}\subset \widetilde{\mathcal E}$ 
is an orthonormal non-resonance basis.
We are done.
\end{proof}

\subsection{Exceptional threshold of the second kind}
\begin{theorem}\label{170820200}
Under the assumptions of Theorem~\ref{thm-ex2} one has 
\begin{align}
G_{0}&=
(I-\mathsf P)
\bigl(G_{0,0}
-G_{0,0}
vM_0^\dagger v^*G_{0,0}
\bigr)(1-\mathsf P)
,\label{ex2-G0}
\\
\begin{split}
G_1&=-\widetilde{\mathcal P},
\end{split}
\label{ex2-G1}
\end{align}
where $\widetilde{\mathcal P}$ is uniquely determined due to $\mathcal{E}/E=\{0\}$.
\end{theorem}
\begin{remark*}
There is an error in \cite[Theorem~I\hspace{-.1em}I\hspace{-.1em}I.6]{IJ2},
which is corrected in the above statement.
\end{remark*}
\begin{proof}
The discussion here continues from the proof of Theorem~\ref{thm-ex2}.
By \eqref{170820157} and \eqref{160822553b} 
we can write, implementing $B_0A_*=A_*B_0=A_*$,
\begin{align*}
G_0
&=
G_{0,0}
-G_{0,0}vC_0v^*G_{0,0}
-G_{0,0}vC_{-1}v^*G_{0,1}
-G_{0,1}vC_{-1}v^*G_{0,0}
\\&\phantom{{}={}}{}
-G_{0,0}vC_{-2}v^*G_{0,2}
-G_{0,1}vC_{-2}v^*G_{0,1}
-G_{0,2}vC_{-2}v^*G_{0,0}
\\&
=
G_{0,0}
\\&\phantom{{}={}}{}
-G_{0,0}v\Bigl(B_0
+A_1
+A_0B_1
+B_1A_0
+B_1A_{-1}B_1
+A_{-1}B_2
+B_2A_{-1}
\Bigr)v^*G_{0,0}
\\&\phantom{{}={}}{}
-G_{0,0}v\bigl(
A_0
+A_{-1}B_1
+B_1A_{-1}
\bigr)v^*
G_1
-G_{0,1}v\bigl(
A_0
+A_{-1}B_1
+B_1A_{-1}
\bigr)v^*G_{0,0}
\\&\phantom{{}={}}{}
-G_{0,0}vA_{-1}v^*G_{0,2}
-G_{0,1}vA_{-1}v^*G_{0,1}
-G_{0,2}vA_{-1}v^*G_{0,0}.
\end{align*}
Let us now use some vanishing relations coming from  
\eqref{am0b}, \eqref{16082220},
and \eqref{1608252242}:
\begin{align*}
G_0
&=
G_{0,0}
-G_{0,0}v\Bigl(B_0
+A_1
+B_1A_{-1}B_1
+A_{-1}B_2
+B_2A_{-1}
\Bigr)v^*G_{0,0}
\\&\phantom{{}={}}{}
-G_{0,0}vA_{-1}B_1v^*G_{0,1}
-G_{0,1}vB_1A_{-1}v^*G_{0,0}
\\&\phantom{{}={}}{}
-G_{0,0}vA_{-1}v^*G_{0,2}
-G_{0,2}vA_{-1}v^*G_{0,0},
\end{align*}
and then 
insert expressions for $A_*$ and $B_*$,
using the explicit kernels of operators
and implementing \eqref{am0b} and \eqref{16082220}.
We omit some computations, to obtain
\begin{align*}
G_0
&=
G_{0,0}
-G_{0,0}v\Bigl(M_0^\dagger +Q
-m_1^\dagger m_3m_1^\dagger
\\&\phantom{{}=G_{0,0}-G_{0,0}v\Bigl({}}{}
-m_1^\dagger  M_2(M_0^\dagger +Q)
-(M_0^\dagger +Q) M_2 m_1^\dagger 
\Bigr)v^*G_{0,0}
\\&\phantom{{}={}}{}
-G_{0,0}vm_1^\dagger v^*G_{0,2}
-G_{0,2}vm_1^\dagger v^*G_{0,0}
.
\end{align*}
Next we unfold $m_3$.
We use the expressions 
$m_3=QM_4Q-QM_2M_0^\dagger M_2Q-m_1m_1$
and $QM_2Q=m_1$
which holds under \eqref{16082220}, and then
\begin{align*}
G_0
&
=
G_{0,0}
-G_{0,0}v(I-m_1^\dagger  M_2)
M_0^\dagger 
(I- M_2m_1^\dagger)
v^*G_{0,0}
\\&\phantom{{}={}}{}
-G_{0,0}v\bigl(
Q
+m_1^\dagger m_1m_1m_1^\dagger
-m_1^\dagger  m_1
-m_1 m_1^\dagger 
\bigr)v^*G_{0,0}
\\&\phantom{{}={}}{}
-G_{0,0}vm_1^\dagger v^*G_{0,2}
-G_{0,2}vm_1^\dagger v^*G_{0,0}
+G_{0,0}vm_1^\dagger M_4m_1^\dagger v^*G_{0,0}
.
\end{align*}
Now we note that by \eqref{160826} we have 
\begin{align}
m_1^{\dagger}
=-Uv^*\mathsf PvU\label{16082619}
\end{align}
and this operator is bijective as $Q\mathcal K\to Q\mathcal K$.
Hence we have 
\begin{align*}
G_0
&
=
G_{0,0}
-(G_{0,0}+\mathsf PVG_{0,2})
vM_0^\dagger v^*
(G_{0,0}+ G_{0,2}V\mathsf P)
\\&\phantom{{}={}}{}
+\mathsf PVG_{0,2}
+G_{0,2}V\mathsf P
+\mathsf PV G_{0,4}V\mathsf P
\end{align*}
Furthermore, we make use of the identities $V\mathsf P=-H_0\mathsf P$, 
$\mathsf PV=-\mathsf PH_0$ and $H_0G_{0,j}=G_{0,j}H_0=G_{0,j-2}$ for $j=2,4$.
We can use the last identities directly unlike in the proof of Lemma~\ref{1708201554},
since bound eigenfunctions decay sufficiently rapidly. Then it follows that 
\begin{align*}
G_0
&
=
G_{0,0}
-(G_{0,0}-\mathsf PG_{0,0})
vM_0^\dagger v^*
(G_{0,0}- G_{0,0}\mathsf P)
-\mathsf PG_{0,0}
-G_{0,0}\mathsf P
+\mathsf P G_{0,0}\mathsf P
\\&
=(I-\mathsf P)
\bigl[G_{0,0}
-G_{0,0}
vM_0^\dagger v^*G_{0,0}
\bigr](1-\mathsf P)
.
\end{align*}
This verifies \eqref{ex2-G0}. 

The computation of $G_1$ in this case is very long,
and we do not present all the details.
We only describe some of important steps. 
First by \eqref{170820157}
we can write $G_1$, using only $A_*$ and $B_*$, as
\begin{align*}
G_1
&
=
G_{0,1}
\\&\phantom{{}={}}{}
-G_{0,0}vA_{-1}v^*G_{0,3}
-G_{0,1}vA_{-1}v^*G_{0,2}
-G_{0,2}vA_{-1}v^*G_{0,1}
-G_{0,3}vA_{-1}v^*G_{0,0}
\\&\phantom{{}={}}{}
-G_{0,0}v\bigl(
A_{-1}B_1
+B_1A_{-1}
+A_0
\bigr)v^*G_{0,2}
\\&\phantom{{}={}}{}
-G_{0,1}v\bigl(
A_{-1}B_1
+B_1A_{-1}
+A_0
\bigr)v^*G_{0,1}
\\&\phantom{{}={}}{}
-G_{0,2}v\bigl(
A_{-1}B_1
+B_1A_{-1}
+A_0
\bigr)v^*G_{0,0}
\\&\phantom{{}={}}{}
-G_{0,0}v\bigl(
B_0
+A_{-1}B_2
+B_1A_{-1}B_1
+B_2A_{-1}
\bigr)v^*G_{0,1}
\\&\phantom{{}={}}{}
-G_{0,1}v\bigl(
B_0
+A_{-1}B_2
+B_1A_{-1}B_1
+B_2A_{-1}
\bigr)v^*G_{0,0}
\\&\phantom{{}={}}{}
-G_{0,0}v\bigl(
B_1
+A_{-1}B_3
+B_1A_{-1}B_2
+B_2A_{-1}B_1
+B_3A_{-1}
\\&\phantom{{}={}-G_{0,0}v\bigl({}}{}
+A_0B_2
+B_1A_0B_1
+B_2A_0
+A_1B_1
+B_1A_1
+A_2
\bigr)v^*G_{0,0}.
\end{align*}
Then we insert the expressions of $A_*$ and $B_*$.
If we implement some of vanishing relations coming from 
\eqref{am0b}, \eqref{16082220}, and \eqref{1608252242},
we arrive at
\begin{align*}
G_1&
=
G_{0,1}
-G_{0,0}vm_1^\dagger v^*G_{0,3}
-G_{0,3}vm_1^\dagger v^*G_{0,0}
\\&\phantom{{}={}}{}
-G_{0,0}v\bigl(
M_0^\dagger 
-m_1^\dagger M_2M_0^\dagger 
\bigr)v^*G_{0,1}
-G_{0,1}v\bigl(
M_0^\dagger 
-M_0^\dagger  M_2m_1^\dagger 
\bigr)v^*G_{0,0}
\\&\phantom{{}={}}{}
-G_{0,0}v\Bigl[
-M_0^\dagger  M_1M_0^\dagger 
+m_1^\dagger (- M_3M_0^\dagger
+M_2M_0^\dagger  M_1M_0^\dagger 
)
\\&\phantom{{}={}-G_{0,0}v\bigl(}{}
+(-M_0^\dagger  M_3
+M_0^\dagger  M_1M_0^\dagger  M_2
)m_1^\dagger 
-m_1^\dagger m_4m_1^\dagger\Bigr]v^*G_{0,0}
.
\end{align*}
If we insert \eqref{16082619} and 
$$m_4=Q\bigl[M_5-M_2(M_0^\dagger+Q)M_3-M_3(M_0^\dagger+Q)M_2+M_2M_0^\dagger M_1M_0^\dagger M_2Q\bigr],$$
which holds especially in this case due to the vanishing relations
implemented above, we come to 
\begin{align*}
G_1&
=
G_{0,1}
+G_{0,0}V\mathsf PVG_{0,3}
+G_{0,3}V\mathsf PVG_{0,0}
\\&\phantom{{}={}}{}
-G_{0,0}\bigl(
vM_0^\dagger v^*
+V\mathsf PVG_{0,2}v M_0^\dagger v^*
\bigr)G_{0,1}
\\&\phantom{{}={}}{}
-G_{0,1}\bigl(
vM_0^\dagger v^*
+vM_0^\dagger  v^*G_{0,2}V\mathsf PV
\bigr)G_{0,0}
\\&\phantom{{}={}}{}
-G_{0,0}\Bigl[
-vM_0^\dagger  v^*G_{0,1}vM_0^\dagger v^*
+ V\mathsf PVG_{0,3}vM_0^\dagger v^*
\\&\phantom{{}={}-G_{0,0}\Bigl[}{}
-V\mathsf PVG_{0,2}v M_0^\dagger  v^*G_{0,1}vM_0^\dagger v^*
+vM_0^\dagger  v^*G_{0,3}V\mathsf PV
\\&\phantom{{}={}-G_{0,0}\Bigl[}{}
-vM_0^\dagger  v^*G_{0,1}vM_0^\dagger  v^*G_{0,2}V\mathsf PV 
-V\mathsf PVG_{0,5}V\mathsf PV
\\&\phantom{{}={}-G_{0,0}\Bigl[}{}
+V\mathsf PVG_{0,2}v M_0^\dagger v^*G_{0,3}V\mathsf PV
+V\mathsf PVG_{0,3}vM_0^\dagger v^*G_{0,2}V\mathsf PV
\\&\phantom{{}={}-G_{0,0}\Bigl[}{}
-V\mathsf PVG_{0,2}vM_0^\dagger v^*G_{0,1}vM_0^\dagger v^*G_{0,2}V\mathsf PV 
\Bigr]G_{0,0}
.
\end{align*}
Use $V\mathsf P=-H_0\mathsf P$, $\mathsf PV=-\mathsf PH_0$
and \eqref{18082023a},
and then we obtain
\begin{align*}
G_1&=
(I-\mathsf P)
\bigl(I-G_{0,0}vM_0^\dagger v^*\bigr)G_{0,1}
\bigl(I-vM_0^\dagger v^*G_{0,0}\bigr)
(I-\mathsf P)
.
\end{align*}
Let us set 
\begin{align*}
\Psi_\alpha=(I-\mathsf P)\bigl(I-G_{0,0}vM_0^\dagger v^*\bigr)\mathbf n^{(\alpha)}.
\end{align*}
Then we can show, as in the proof of Theorem~\ref{170820205}, 
that $\Psi_\alpha-\mathbf n^{(\alpha)}\in\mathbf 1\oplus\mathcal L^{\beta-2}$,
and that $\Psi_\alpha\in\widetilde{\mathcal E}$.
Hence $\{\Psi_\alpha\}\subset \widetilde{\mathcal E}$ is an orthonormal non-resonance basis,
and the expression \eqref{ex2-G1} is obtained.
Hence we are done.
\end{proof}

\subsection{Exceptional threshold of the third kind}
\begin{theorem}\label{17082017}
Under the assumptions of Theorem~\ref{thm-ex3} one has 
\begin{align}
\begin{split}
G_0
&
=
(I-\mathsf P)
\bigl[
G_{0,0}
-G_{0,0}vm_0^\dagger v^*G_{0,1}
-G_{0,1}vm_0^\dagger v^*G_{0,0}
\\&\phantom{{}=(I-\mathsf P)\bigl(}{}
-G_{0,0}\bigl(I-vm_0^\dagger v^*G_{0,1}\bigr)
vM_0^\dagger v^*\bigl(I-G_{0,1}vm_0^\dagger v^*\bigr)G_{0,0}
\\&\phantom{{}=(I-\mathsf P)\bigl(}{}
+G_{0,0}vm_0^\dagger v^*G_{0,2}vm_0^\dagger v^*G_{0,0}
\\&\phantom{{}=(I-\mathsf P)\bigl(}{}
+G_{0,0}vm_0^\dagger v^*G_{0,0}\mathsf PG_{0,0}v m_0^\dagger v^*G_{0,0}
\bigr]\bigl(I-\mathsf P\bigr)
,
\end{split}
\label{170826232}\\
\begin{split}
G_1&=-\widetilde{\mathcal P}
+\mathcal PVG_{0,3}V\mathcal P
-\mathcal PVG_{0,2}vm_0^\dagger v^*G_{0,2}V\mathcal P
\\&\phantom{{}={}}{}
-(I-\mathsf P)\bigl[I-G_{0,0}\bigl(I-vm_0^\dagger v^*G_{0,1} \bigr)vM_0^\dagger v^*\bigr]
\\&\phantom{{}={}-{}}{}
\cdot
\bigl(I-G_{0,1}vm_0^\dagger v^* \bigr)G_{0,2}V\mathcal P
\\&\phantom{{}={}}{}
-\mathcal PVG_{0,2}\bigl(I-vm_0^\dagger  v^*G_{0,1}\bigr)
\\&\phantom{{}={}-{}}{}
\cdot
\bigl[I-vM_0^\dagger v^*\bigl(I-G_{0,1}vm_0^\dagger v^*\bigr)G_{0,0}\bigr](I-\mathsf P)
,
\end{split}
\label{170826233}
\end{align}
where $\widetilde{\mathcal P}$ is an orthogonal non-resonance projection associated with 
the orthonormal non-resonance basis $\{\Psi_{\widetilde\gamma}\}_{\widetilde\gamma}$ consisting of 
\begin{align*}
\Psi_{\widetilde\gamma}
=
(I-\mathsf P)\bigl[I-G_{0,0}\bigl(I-vm_0^\dagger  v^*G_{0,1}\bigr)vM_0^\dagger v^*\bigr]
\mathbf n_{\widetilde\gamma}
\end{align*}
with the notation from \eqref{1708202319}.
This in particular verifies the orthogonality \eqref{1708221320}.
\end{theorem}
\begin{remarks*}
\ 

1.
If we use \eqref{17082422} and \eqref{17082423}, we can further rewrite
the above formulas in terms of the vectors of the form \eqref{1708202319},
but here we do not elaborate further, cf.\ \eqref{170826231}.

2. 
The formulas in 
Theorems~\ref{170820205}, \ref{17082015} and \ref{170820200} 
may of course be seen as special cases of that in Theorem~\ref{17082017}.
\end{remarks*}
\begin{proof}
The discussion here is a continuation from the proof of Theorem~\ref{thm-ex3}. 
The computations are very long, and we present only important steps, skipping details. 
By \eqref{17082214},\eqref{160822553bbb} and \eqref{160822553bb}
we first write $G_0$ in terms only of $A_*,B_*,C_*$,
and then use the identities $B_0C_*=C_*B_0=C_*$ and $C_0A_*=A_*C_0=A_*$:
\begin{align*}
G_0&=G_{0,0}
-G_{0,1}vE_{-2}v^*G_{0,1}
-G_{0,0}vE_{-2}v^*G_{0,2}
-G_{0,2}vE_{-2}v^*G_{0,0}
\\&\phantom{{}={}}{}
-G_{0,0}vE_{-1}v^*G_{0,1}
-G_{0,1}vE_{-1}v^*G_{0,0}
-G_{0,0}vE_0v^*G_{0,0}
\\&
=
G_{0,0}
-G_{0,1}vA_0v^*G_{0,1}
-G_{0,0}vA_0v^*G_{0,2}
-G_{0,2}vA_0v^*G_{0,0}
\\&\phantom{{}={}}{}
-G_{0,0}v\Bigl(
C_0+A_1+A_0C_1+C_1A_0
+A_0B_1+B_1A_0\Bigr)v^*G_{0,1}
\\&\phantom{{}={}}{}
-G_{0,1}v\Bigl(
C_0+A_1+A_0C_1+C_1A_0
+A_0B_1+B_1A_0\Bigr)v^*G_{0,0}
\\&\phantom{{}={}}{}
-G_{0,0}v
\Bigl(
B_0
+B_1A_0B_1
+A_0B_2
+B_2A_0
\\&\phantom{{}=-G_{0,0}v\Bigl(}{}
+\bigl(C_0+A_1+A_0C_1+C_1A_0\bigr)B_1
+B_1\bigl(C_0+A_1+A_0C_1+C_1A_0\bigr)
\\&\phantom{{}=-G_{0,0}v\Bigl(}{}
+C_1+C_1A_0C_1+A_0C_2+C_2A_0
+A_1C_1+C_1A_1+A_2
\Bigr)
v^*G_{0,0}.
\end{align*}
Then we use \eqref{17082314}, \eqref{1708212} and \eqref{1708231451}. 
After some computations, particularly employing 
$m_0=QM_1Q$, $q_0=Sm_1S$ and vanishing relations
\begin{align}
Sv^*G_{0,1}=0,\quad 
SM_1=SB_1=Sm_0=Sm_2S=0
\label{170824}
\end{align}
essentially due to \eqref{16082320}, 
we have 
\begin{align*}
G_0&
=
G_{0,0}
-G_{0,0}vq_0^{\dagger}v^*G_{0,2}
-G_{0,2}vq_0^{\dagger}v^*G_{0,0}
\\&\phantom{{}={}}{}
-G_{0,0}v\bigl(I-q_0^{\dagger}m_1\bigr)m_0^\dagger v^*G_{0,1}
-G_{0,1}vm_0^\dagger \bigl(I-m_1q_0^{\dagger}\bigr)v^*G_{0,0}
\\&\phantom{{}={}}{}
-G_{0,0}v
\Bigl(
M_0^\dagger -Q
-q_0^{\dagger} M_2M_0^\dagger 
-q_0^{\dagger} M_2Q
-M_0^\dagger  M_2q_0^{\dagger}
-Q M_2q_0^{\dagger}
\\&\phantom{{}=-G_{0,0}v\Bigl(}{}
-m_0^\dagger M_1M_0^\dagger 
+q_0^{\dagger} m_1m_0^\dagger M_1M_0^\dagger 
+q_0^{\dagger} m_1
\\&\phantom{{}=-G_{0,0}v\Bigl(}{}
-M_0^\dagger  M_1m_0^\dagger 
+M_0^\dagger  M_1m_0^\dagger m_1q_0^{\dagger}
+m_1q_0^{\dagger}
\\&\phantom{{}=-G_{0,0}v\Bigl(}{}
-m_0^\dagger  m_1m_0^\dagger 
+2Sm_1S
+m_0^\dagger m_1q_0^{\dagger} m_1m_0^\dagger
+m_0^\dagger m_1 S
+S  m_1m_0^\dagger 
\\&\phantom{{}=-G_{0,0}v\Bigl(}{}
-q_0^{\dagger}m_2m_0^\dagger 
+q_0^{\dagger}m_1m_0^\dagger m_1m_0^\dagger 
+q_0^{\dagger}m_1m_0^\dagger m_1S
\\&\phantom{{}=-G_{0,0}v\Bigl(}{}
-m_0^\dagger m_2q_0^{\dagger}
+m_0^\dagger m_1m_0^\dagger m_1q_0^{\dagger}
+Sm_1m_0^\dagger m_1q_0^{\dagger}
\\&\phantom{{}=-G_{0,0}v\Bigl(}{}
+q_0^\dagger q_1q_0^\dagger m_1m_0^\dagger
+q_0^\dagger q_1
+m_0^\dagger  m_1q_0^\dagger q_1q_0^\dagger
+ q_1q_0^\dagger
\\&\phantom{{}=-G_{0,0}v\Bigl(}{}
-q_0^\dagger q_2q_0^\dagger+q_0^\dagger q_1q_0^\dagger q_1q_0^\dagger
\Bigr)
v^*G_{0,0}.
\end{align*}
We next use \eqref{am0bb} and \eqref{170824}: 
\begin{align*}
G_0&
=
G_{0,0}
-G_{0,0}vq_0^{\dagger}v^*G_{0,2}
-G_{0,2}vq_0^{\dagger}v^*G_{0,0}
\\&\phantom{{}={}}{}
-G_{0,0}v\bigl(I-q_0^{\dagger}m_1\bigr)m_0^\dagger v^*G_{0,1}
-G_{0,1}vm_0^\dagger \bigl(I-m_1q_0^{\dagger}\bigr)v^*G_{0,0}
\\&\phantom{{}={}}{}
-G_{0,0}v
\Bigl(
M_0^\dagger -Q
-q_0^{\dagger} M_2M_0^\dagger 
-q_0^{\dagger} M_2Q
-M_0^\dagger  M_2q_0^{\dagger}
-Q M_2q_0^{\dagger}
\\&\phantom{{}=-G_{0,0}v\Bigl(}{}
-m_0^\dagger M_1M_0^\dagger 
+q_0^{\dagger} m_1m_0^\dagger M_1M_0^\dagger 
+q_0^{\dagger} m_1
\\&\phantom{{}=-G_{0,0}v\Bigl(}{}
-M_0^\dagger  M_1m_0^\dagger 
+M_0^\dagger  M_1m_0^\dagger m_1q_0^{\dagger}
+m_1q_0^{\dagger}
\\&\phantom{{}=-G_{0,0}v\Bigl(}{}
-m_0^\dagger  m_1m_0^\dagger 
+m_0^\dagger m_1q_0^{\dagger} m_1m_0^\dagger
\\&\phantom{{}=-G_{0,0}v\Bigl(}{}
-q_0^{\dagger}m_2m_0^\dagger 
+q_0^{\dagger}m_1m_0^\dagger m_1m_0^\dagger 
-m_0^\dagger m_2q_0^{\dagger}
+m_0^\dagger m_1m_0^\dagger m_1q_0^{\dagger}
\\&\phantom{{}=-G_{0,0}v\Bigl(}{}
-q_0^\dagger m_1m_0^\dagger m_1q_0^\dagger m_1m_0^\dagger
-m_0^\dagger  m_1q_0^\dagger m_1m_0^\dagger m_1q_0^\dagger
\\&\phantom{{}=-G_{0,0}v\Bigl(}{}
-q_0^\dagger m_3q_0^\dagger
+q_0^\dagger m_1m_0^\dagger m_2q_0^\dagger
+q_0^\dagger m_2m_0^\dagger m_1q_0^\dagger
-q_0^\dagger m_1m_0^\dagger m_1m_0^\dagger m_1q_0^\dagger
\\&\phantom{{}=-G_{0,0}v\Bigl(}{}
+q_0^\dagger m_1m_0^\dagger m_1q_0^\dagger m_1m_0^\dagger m_1q_0^\dagger
\Bigr)
v^*G_{0,0}.
\end{align*}
Insert \eqref{1608231744}--\eqref{1608231746} and use  \eqref{170824}:
\begin{align*}
G_0&
=
G_{0,0}
-G_{0,0}vq_0^{\dagger}v^*G_{0,2}
-G_{0,2}vq_0^{\dagger}v^*G_{0,0}
\\&\phantom{{}={}}{}
-G_{0,0}v\bigl(I-q_0^{\dagger}M_2\bigr)m_0^\dagger v^*G_{0,1}
-G_{0,1}vm_0^\dagger\bigl(I-M_2q_0^{\dagger}\bigr) v^*G_{0,0}
\\&\phantom{{}={}}{}
-G_{0,0}v
\Bigl(
M_0^\dagger 
-q_0^{\dagger} M_2M_0^\dagger 
-M_0^\dagger  M_2q_0^{\dagger}
\\&\phantom{{}=-G_{0,0}v\Bigl(}{}
-m_0^\dagger M_1M_0^\dagger 
+q_0^{\dagger} M_2m_0^\dagger M_1M_0^\dagger 
-M_0^\dagger  M_1m_0^\dagger 
+M_0^\dagger  M_1m_0^\dagger M_2q_0^{\dagger}
\\&\phantom{{}=-G_{0,0}v\Bigl(}{}
-m_0^\dagger M_2m_0^\dagger 
+m_0^\dagger M_1M_0^\dagger M_1m_0^\dagger 
+m_0^\dagger M_2q_0^{\dagger} M_2m_0^\dagger
\\&\phantom{{}=-G_{0,0}v\Bigl(}{}
-q_0^{\dagger}M_3m_0^\dagger 
+q_0^{\dagger}M_2M_0^\dagger M_1m_0^\dagger 
+q_0^{\dagger}M_2m_0^\dagger M_2m_0^\dagger 
\\&\phantom{{}=-G_{0,0}v\Bigl(}{}
-q_0^{\dagger}M_2m_0^\dagger M_1M_0^\dagger M_1m_0^\dagger 
\\&\phantom{{}=-G_{0,0}v\Bigl(}{}
-m_0^\dagger M_3q_0^{\dagger}
+m_0^\dagger M_1M_0^\dagger M_2q_0^{\dagger}
+m_0^\dagger M_2m_0^\dagger M_2q_0^{\dagger}
\\&\phantom{{}=-G_{0,0}v\Bigl(}{}
-m_0^\dagger M_1M_0^\dagger M_1m_0^\dagger M_2q_0^{\dagger}
\\&\phantom{{}=-G_{0,0}v\Bigl(}{}
-q_0^\dagger M_2m_0^\dagger M_2q_0^\dagger M_2m_0^\dagger
-m_0^\dagger  M_2q_0^\dagger M_2m_0^\dagger M_2q_0^\dagger
\\&\phantom{{}=-G_{0,0}v\Bigl(}{}
-q_0^\dagger M_4q_0^\dagger
+q_0^\dagger M_2M_0^\dagger M_2q_0^\dagger
\\&\phantom{{}=-G_{0,0}v\Bigl(}{}
+q_0^\dagger M_2m_0^\dagger M_3q_0^\dagger
-q_0^\dagger M_2m_0^\dagger M_1M_0^\dagger M_2q_0^\dagger
+q_0^\dagger M_3m_0^\dagger M_2q_0^\dagger
\\&\phantom{{}=-G_{0,0}v\Bigl(}{}
-q_0^\dagger M_2M_0^\dagger M_1m_0^\dagger M_2q_0^\dagger
-q_0^\dagger M_2m_0^\dagger M_2m_0^\dagger M_2q_0^\dagger
\\&\phantom{{}=-G_{0,0}v\Bigl(}{}
+q_0^\dagger M_2m_0^\dagger M_1M_0^\dagger M_1m_0^\dagger M_2q_0^\dagger
+q_0^\dagger M_2m_0^\dagger M_2q_0^\dagger M_2m_0^\dagger M_2q_0^\dagger
\Bigr)
v^*G_{0,0}
\\&
=
G_{0,0}
+G_{0,0}vq_0^\dagger M_4q_0^\dagger v^*G_{0,0}
-G_{0,0}vq_0^{\dagger}v^*G_{0,2}
-G_{0,2}vq_0^{\dagger}v^*G_{0,0}
\\&\phantom{{}={}}{}
-G_{0,0}v\bigl(I-q_0^{\dagger}M_2\bigr)M_0^\dagger\bigl(I-M_2q_0^{\dagger}\bigr)v^*G_{0,0}
\\&\phantom{{}={}}{}
-G_{0,0}v\bigl(I-q_0^{\dagger}M_2\bigr)m_0^\dagger M_1M_0^\dagger M_1m_0^\dagger \bigl(I-M_2q_0^{\dagger}\bigr)v^*G_{0,0}
\\&\phantom{{}={}}{}
+G_{0,0}v
\bigl(I-q_0^\dagger M_2\bigr)m_0^\dagger M_2m_0^\dagger \bigl(I-M_2q_0^{\dagger}\bigr)v^*G_{0,0}
\\&\phantom{{}={}}{}
-G_{0,0}v
\bigl(I-q_0^\dagger M_2\bigr)m_0^\dagger M_2q_0^{\dagger} M_2m_0^\dagger\bigl(I-M_2q_0^{\dagger}\bigr)v^*G_{0,0}
\\&\phantom{{}={}}{}
-G_{0,0}v\bigl(I-q_0^{\dagger}M_2\bigr)m_0^\dagger 
\Bigl(v^*G_{0,1}-M_3q_0^{\dagger}v^*G_{0,0}-M_1M_0^\dagger\bigl(I-M_2q_0^{\dagger}\bigr)v^*G_{0,0}\Bigr)
\\&\phantom{{}={}}{}
-\Bigl(G_{0,1}v-G_{0,0}vq_0^{\dagger}M_3-G_{0,0}v\bigl(I-q_0^{\dagger}M_2\bigr)M_0^\dagger M_1\Bigr)
m_0^\dagger \bigl(I-M_2q_0^{\dagger}\bigr)v^*G_{0,0}
.
\end{align*}
Now we use $M_j=v^*G_{0,j}v$ for $j=1,2,3$, $q_0^\dagger=-Uv^*\mathsf PvU$, 
$V\mathsf P=-H_0\mathsf P$ and $\mathsf PV=-\mathsf PH_0$:
\begin{align*}
G_1
&
=
(I-\mathsf P)\Bigl(G_{0,0}-G_{0,0}vM_0^\dagger v^*G_{0,0}
+G_{0,0}vm_0^\dagger v^*G_{0,2}vm_0^\dagger v^*G_{0,0}
\\&\phantom{{}={}(I-\mathsf P)\Bigl({}}{}
-G_{0,0}vm_0^\dagger v^*G_{0,1}vM_0^\dagger v^*G_{0,1}vm_0^\dagger v^*G_{0,0}
\\&\phantom{{}={}(I-\mathsf P)\Bigl({}}{}
+G_{0,0}vm_0^\dagger v^*G_{0,0}\mathsf PG_{0,0}v m_0^\dagger v^*G_{0,0}
\\&\phantom{{}={}(I-\mathsf P)\Bigl({}}{}
-G_{0,0}vm_0^\dagger v^*G_{0,1}
-G_{0,1}vm_0^\dagger v^*G_{0,0}
\\&\phantom{{}={}(I-\mathsf P)\Bigl({}}{}
+G_{0,0}vm_0^\dagger v^*G_{0,1}vM_0^\dagger v^*G_{0,0}
+G_{0,0}vM_0^\dagger v^*G_{0,1}vm_0^\dagger v^*G_{0,0}\Bigr)\bigl(I-\mathsf P\bigr)
.
\end{align*}
Hence we obtain \eqref{170826232}.

Let us next compute $G_1$. 
We proceed similarly to $G_0$.
By \eqref{17082214}, \eqref{160822553bbb} and \eqref{160822553bb}
we write $G_0$ in terms only of $A_*,B_*,C_*$,
noting $B_0C_*=C_*B_0=C_*$ and $C_0A_*=A_*C_0=A_*$:
\begin{align*}
G_1
&
=G_{0,1}
\\&\phantom{{}={}}{}
-G_{0,0}vE_{-2}v^*G_{0,3}
-G_{0,3}vE_{-2}v^*G_{0,0}
-G_{0,1}vE_{-2}v^*G_{0,2}
-G_{0,2}vE_{-2}v^*G_{0,1}
\\&\phantom{{}={}}{}
-G_{0,1}vE_{-1}v^*G_{0,1}
-G_{0,0}vE_{-1}v^*G_{0,2}
-G_{0,2}vE_{-1}v^*G_{0,0}
\\&\phantom{{}={}}{}
-G_{0,0}vE_0v^*G_{0,1}
-G_{0,1}vE_0v^*G_{0,0}
-G_{0,0}vE_1v^*G_{0,0}
\\&
=
G_{0,1}
-G_{0,0}vA_0v^*G_{0,3}
-G_{0,3}vA_0v^*G_{0,0}
-G_{0,1}vA_0v^*G_{0,2}
-G_{0,2}vA_0v^*G_{0,1}
\\&\phantom{{}={}}{}
-G_{0,1}v\Bigl(C_0+A_1+A_0C_1+C_1A_0+A_0B_1+B_1A_0\Bigr)v^*G_{0,1}
\\&\phantom{{}={}}{}
-G_{0,0}v\Bigl(C_0+A_1+A_0C_1+C_1A_0+A_0B_1+B_1A_0\Bigr)v^*G_{0,2}
\\&\phantom{{}={}}{}
-G_{0,2}v\Bigl(C_0+A_1+A_0C_1+C_1A_0+A_0B_1+B_1A_0\Bigr)v^*G_{0,0}
\\&\phantom{{}={}}{}
-G_{0,0}v\Bigl(
B_0
+B_1A_0B_1
+A_0B_2
+B_2A_0
\\&\phantom{{}=-G_{0,0}v\Bigl(}{}
+\bigl(C_0+A_1+A_0C_1+C_1A_0\bigr)B_1
+B_1\bigl(C_0+A_1+A_0C_1+C_1A_0\bigr)
\\&\phantom{{}=-G_{0,0}v\Bigl(}{}
+C_1+C_1A_0C_1+A_0C_2+C_2A_0
+A_1C_1+C_1A_1+A_2
\Bigr)v^*G_{0,1}
\\&\phantom{{}={}}{}
-G_{0,1}v\Bigl(
B_0
+B_1A_0B_1
+A_0B_2
+B_2A_0
\\&\phantom{{}=-G_{0,0}v\Bigl(}{}
+\bigl(C_0+A_1+A_0C_1+C_1A_0\bigr)B_1
+B_1\bigl(C_0+A_1+A_0C_1+C_1A_0\bigr)
\\&\phantom{{}=-G_{0,0}v\Bigl(}{}
+C_1+C_1A_0C_1+A_0C_2+C_2A_0+A_1C_1+C_1A_1+A_2
\Bigr)v^*G_{0,0}
\\&\phantom{{}={}}{}
-G_{0,0}v\Bigl(
B_1
+A_0B_3
+B_3A_0
+B_1A_0B_2
+B_2A_0B_1
\\&\phantom{{}=-G_{0,0}v\Bigl(}{}
+B_1\bigl(C_0+A_1+A_0C_1+C_1A_0\bigr)B_1
\\&\phantom{{}=-G_{0,0}v\Bigl(}{}
+\bigl(C_0+A_1+A_0C_1+C_1A_0\bigr)B_2
+B_2\bigl(C_0+A_1+A_0C_1+C_1A_0\bigr)
\\&\phantom{{}=-G_{0,0}v\Bigl(}{}
+\bigl(C_1+C_1A_0C_1+A_0C_2+C_2A_0
+A_1C_1+C_1A_1+A_2\bigr)B_1
\\&\phantom{{}=-G_{0,0}v\Bigl(}{}
+B_1\bigl(C_1+C_1A_0C_1+A_0C_2+C_2A_0
+A_1C_1+C_1A_1+A_2\bigr)
\\&\phantom{{}=-G_{0,0}v\Bigl(}{}
+C_2+A_0C_3+C_3A_0+C_1A_0C_2+C_2A_0C_1+C_1A_1C_1
\\&\phantom{{}=-G_{0,0}v\Bigl(}{}
+A_1C_2+C_2A_1
+A_2C_1+C_1A_2+A_3
\Bigr)
v^*G_{0,0}
\end{align*}
We proceed similarly to $G_0$, 
but formulas for $G_1$ are extremely long,
and let omit most of the computation process.
We use \eqref{17082314}, \eqref{1708212}, \eqref{1708231451} 
for the expressions of 
$A_*$, $B_*$ and $C_*$, 
and then substitute 
\eqref{1608231744}--\eqref{1608231747} and \eqref{am0bb} for $m_*$ and $q_*$, respectively.
We also use \eqref{170824} many times. As a result, we obtain
\begin{align*}
G_1
&
=
G_{0,1}
-G_{0,0}vq_0^{\dagger}v^*G_{0,3}
-G_{0,3}vq_0^{\dagger}v^*G_{0,0}
-G_{0,1}vm_0^\dagger v^*G_{0,1}
\\&\phantom{{}={}}{}
-G_{0,0}v\bigl(I-q_0^{\dagger}M_2\bigr)m_0^\dagger \bigl(I- M_2q_0^{\dagger}\bigr)v^*G_{0,2}
\\&\phantom{{}={}}{}
-G_{0,2}v\bigl(I-q_0^{\dagger}M_2\bigr)m_0^\dagger \bigl(I- M_2q_0^{\dagger}\bigr)v^*G_{0,0}
\\&\phantom{{}={}}{}
-G_{0,0}v\Bigl(
-q_0^{\dagger}M_3m_0^\dagger
+\bigl(I-q_0^\dagger M_2\bigr)\bigl(I-m_0^\dagger  M_1\bigr)M_0^\dagger\bigl(I-M_1m_0^\dagger  \bigr)
\\&\phantom{{}=-G_{0,0}v\Bigl(}{}
-\bigl(I-q_0^\dagger M_2\bigr)m_0^\dagger\bigl(I-M_2q_0^{\dagger}\bigr)M_2m_0^\dagger  
\Bigr)v^*G_{0,1}
\\&\phantom{{}={}}{}
-G_{0,1}v\Bigl(
-m_0^\dagger  M_3 q_0^{\dagger}
+\bigl(I-m_0^\dagger  M_1\bigr)M_0^\dagger\bigl(I- M_1m_0^\dagger \bigr) \bigl(I-M_2q_0^{\dagger}\bigr)
\\&\phantom{{}=-G_{0,0}v\Bigl(}{}
-m_0^\dagger M_2\bigl(I- q_0^{\dagger}M_2\bigr)m_0^\dagger  \bigl(I-M_2q_0^{\dagger}\bigr)
\Bigr)v^*G_{0,0}
\\&\phantom{{}={}}{}
-G_{0,0}v\Bigl(
-(M_0^\dagger+Q) \bigl(M_1- M_1 m_0^\dagger M_1\bigr)(M_0^\dagger+Q) 
\\&\phantom{{}=-G_{0,0}v\Bigl(}{}
+(M_0^\dagger+Q)\bigl(M_1-M_1 m_0^\dagger M_1\bigr)(M_0^\dagger+Q)
\\&\phantom{{}=-G_{0,0}v\Bigl({}+{}}{}
\cdot 
\bigl(M_1m_0^\dagger +M_2q_0^{\dagger}-M_1m_0^\dagger M_2 q_0^{\dagger}\bigr)
\\&\phantom{{}=-G_{0,0}v\Bigl(}{}
+\bigl( m_0^\dagger M_1+q_0^{\dagger}M_2 -q_0^{\dagger}M_2m_0^\dagger M_1\bigr)
\\&\phantom{{}=-G_{0,0}v\Bigl({}+{}}{}
\cdot 
(M_0^\dagger+Q) \bigl(M_1-M_1m_0^\dagger M_1\bigr)(M_0^\dagger+Q)
\\&\phantom{{}=-G_{0,0}v\Bigl(}{}
-\bigl( I-q_0^{\dagger}M_2\bigr)m_0^\dagger\bigl(I-M_2 q_0^{\dagger}\bigr)M_2\bigl(I-m_0^\dagger M_1\bigr)(M_0^\dagger+Q)
\\&\phantom{{}=-G_{0,0}v\Bigl(}{}
-(M_0^\dagger+Q)\bigl(I-M_1 m_0^\dagger\bigr)M_2\bigl(I-q_0^{\dagger} M_2\bigr)m_0^\dagger \bigl(I-M_2 q_0^{\dagger}\bigr)
\\&\phantom{{}=-G_{0,0}v\Bigl(}{}
-q_0^{\dagger}M_3\bigl(I-m_0^\dagger M_1\bigr)(M_0^\dagger+Q)
-(M_0^\dagger+Q)\bigl(I-M_1m_0^\dagger\bigr)M_3q_0^{\dagger}
\\&\phantom{{}=-G_{0,0}v\Bigl(}{}
-q_0^\dagger \bigl(m_4-m_2m_0^\dagger m_2\bigr)q_0^\dagger
\\&\phantom{{}=-G_{0,0}v\Bigl(}{}
-\bigl(I-q_0^\dagger m_1\bigr)m_0^\dagger \bigl(I-m_1q_0^{\dagger}\bigr)
\bigl(m_2-m_1m_0^\dagger m_1\bigr)
\\&\phantom{{}=-G_{0,0}v\Bigl({}-{}}{}
\cdot\bigl(I-q_0^{\dagger}m_1\bigr)m_0^\dagger  \bigl(I-m_1q_0^\dagger\bigr)
\\&\phantom{{}=-G_{0,0}v\Bigl(}{}
-q_0^{\dagger}\bigl(m_3-m_2m_0^\dagger m_1\bigr) \bigl(I-q_0^\dagger m_1\bigr)m_0^\dagger \bigl(I-m_1q_0^\dagger\bigr)
\\&\phantom{{}=-G_{0,0}v\Bigl(}{}
-\bigl(I-q_0^\dagger m_1\bigr)m_0^\dagger  \bigl(I-m_1 q_0^\dagger\bigr)\bigl(m_3-m_1m_0^\dagger m_2\bigr)q_0^{\dagger}
\Bigr)
v^*G_{0,0}
\\&
=
G_{0,1}
-G_{0,0}vq_0^{\dagger}v^*G_{0,3}
-G_{0,3}vq_0^{\dagger}v^*G_{0,0}
-G_{0,1}vm_0^\dagger v^*G_{0,1}
\\&\phantom{{}={}}{}
-G_{0,0}v\bigl(I-q_0^{\dagger}M_2\bigr)m_0^\dagger \bigl(I- M_2q_0^{\dagger}\bigr)v^*G_{0,2}
\\&\phantom{{}={}}{}
-G_{0,2}v\bigl(I-q_0^{\dagger}M_2\bigr)m_0^\dagger \bigl(I- M_2q_0^{\dagger}\bigr)v^*G_{0,0}
\\&\phantom{{}={}}{}
-G_{0,0}v\Bigl(
-q_0^{\dagger}M_3m_0^\dagger
+\bigl(I-q_0^\dagger M_2\bigr)\bigl(I-m_0^\dagger  M_1\bigr)M_0^\dagger\bigl(I-M_1m_0^\dagger  \bigr)
\\&\phantom{{}=-G_{0,0}v\Bigl(}{}
-\bigl(I-q_0^\dagger M_2\bigr)m_0^\dagger\bigl(I-M_2q_0^{\dagger}\bigr)M_2m_0^\dagger  
\Bigr)v^*G_{0,1}
\\&\phantom{{}={}}{}
-G_{0,1}v\Bigl(
-m_0^\dagger  M_3 q_0^{\dagger}
+\bigl(I-m_0^\dagger  M_1\bigr)M_0^\dagger\bigl(I- M_1m_0^\dagger \bigr) \bigl(I-M_2q_0^{\dagger}\bigr)
\\&\phantom{{}=-G_{0,0}v\Bigl(}{}
-m_0^\dagger M_2\bigl(I- q_0^{\dagger}M_2\bigr)m_0^\dagger  \bigl(I-M_2q_0^{\dagger}\bigr)
\Bigr)v^*G_{0,0}
\\&\phantom{{}={}}{}
-G_{0,0}v\Bigl(
-M_0^\dagger \bigl(M_1- M_1 m_0^\dagger M_1\bigr)M_0^\dagger 
\\&\phantom{{}=-G_{0,0}v\Bigl(}{}
+M_0^\dagger\bigl(M_1-M_1 m_0^\dagger M_1\bigr)M_0^\dagger \bigl(M_1m_0^\dagger +M_2q_0^{\dagger}-M_1m_0^\dagger M_2 q_0^{\dagger}\bigr)
\\&\phantom{{}=-G_{0,0}v\Bigl(}{}
+\bigl( m_0^\dagger M_1+q_0^{\dagger}M_2 -q_0^{\dagger}M_2m_0^\dagger M_1\bigr)M_0^\dagger \bigl(M_1-M_1m_0^\dagger M_1\bigr)M_0^\dagger
\\&\phantom{{}=-G_{0,0}v\Bigl(}{}
-\bigl( I-q_0^{\dagger}M_2\bigr)m_0^\dagger\bigl(I-M_2 q_0^{\dagger}\bigr)M_2\bigl(I-m_0^\dagger M_1\bigr)M_0^\dagger
\\&\phantom{{}=-G_{0,0}v\Bigl(}{}
-M_0^\dagger\bigl(I-M_1 m_0^\dagger\bigr)M_2\bigl(I-q_0^{\dagger} M_2\bigr)m_0^\dagger \bigl(I-M_2 q_0^{\dagger}\bigr)
\\&\phantom{{}=-G_{0,0}v\Bigl(}{}
-q_0^{\dagger}M_3\bigl(I-m_0^\dagger M_1\bigr)M_0^\dagger
-M_0^\dagger\bigl(I-M_1m_0^\dagger\bigr)M_3q_0^{\dagger}
\\&\phantom{{}=-G_{0,0}v\Bigl(}{}
-q_0^\dagger \Bigl[
M_5
-M_2M_0^\dagger M_3
-M_3M_0^\dagger M_2
+M_2M_0^\dagger \bigl(M_1-M_1m_0^\dagger M_1\bigr)M_0^\dagger M_2
\\&\phantom{{}=-G_{0,0}v\Bigl({}-q_0^\dagger \Bigl[}{}
-M_3m_0^\dagger M_3
+M_2M_0^\dagger M_1m_0^\dagger M_3
+M_3m_0^\dagger M_1M_0^\dagger M_2
\Bigr]q_0^\dagger
\\&\phantom{{}=-G_{0,0}v\Bigl(}{}
-\bigl(I-q_0^\dagger m_1\bigr)m_0^\dagger \bigl(I-m_1q_0^{\dagger}\bigr)
\\&\phantom{{}=-G_{0,0}v\Bigl({}-{}}{}
\cdot\Bigl[
M_3
+M_1M_0^\dagger \bigl(M_1-M_1m_0^\dagger M_1\bigr)M_0^\dagger M_1
-M_2m_0^\dagger M_2
\\&\phantom{{}=-G_{0,0}v\Bigl({}-{}\cdot\Bigl[{}}{}
-M_1M_0^\dagger \bigl(I-M_1m_0^\dagger \bigr)M_2
-M_2\big(I-m_0^\dagger M_1\bigr)M_0^\dagger M_1
\Bigr]
\\&\phantom{{}=-G_{0,0}v\Bigl({}-{}}{}
\cdot\bigl(I-q_0^{\dagger}m_1\bigr)m_0^\dagger  \bigl(I-m_1q_0^\dagger\bigr)
\\&\phantom{{}=-G_{0,0}v\Bigl(}{}
-q_0^{\dagger}\Bigl[
M_4
-M_2M_0^\dagger \bigl(I-M_1m_0^\dagger \bigr)M_2
-M_3\bigl(I-m_0^\dagger M_1\bigr)M_0^\dagger M_1
\\&\phantom{{}=-G_{0,0}v\Bigl({}-q_0^\dagger \Bigl[}{}
+M_2M_0^\dagger \bigl(I-M_1m_0^\dagger \bigr)M_1M_0^\dagger M_1
-M_3m_0^\dagger M_2
\Bigr] 
\\&\phantom{{}=-G_{0,0}v\Bigl({}-{}}{}
\cdot\bigl(I-q_0^\dagger m_1\bigr)m_0^\dagger \bigl(I-m_1q_0^\dagger\bigr)
\\&\phantom{{}=-G_{0,0}v\Bigl(}{}
-\bigl(I-q_0^\dagger m_1\bigr)m_0^\dagger  \bigl(I-m_1 q_0^\dagger\bigr)
\\&\phantom{{}=-G_{0,0}v\Bigl({}-{}}{}
\cdot\Bigl[
M_4
-M_1M_0^\dagger \bigl(I-M_1m_0^\dagger \bigr)M_3
-M_2\bigl(I-m_0^\dagger M_1\bigr)M_0^\dagger M_2
\\&\phantom{{}=-G_{0,0}v\Bigl({}-{}\cdot\Bigl[{}}{}
+M_1M_0^\dagger M_1\big(I-m_0^\dagger M_1\bigr)M_0^\dagger M_2
-M_2m_0^\dagger M_3
\Bigr]q_0^{\dagger}
\Bigr)
v^*G_{0,0}
\end{align*}
We rewrite it in terms of $G_{0,*}$, using $M_j=v^*G_{0,j}v$ for $j=1,\ldots, 5$:
\begin{align*}
G_1&
=
\bigl[I-G_{0,0}\bigl(I-vq_0^\dagger v^*G_{0,2}\bigr)\bigl(I-vm_0^\dagger  v^*G_{0,1}\bigr)vM_0^\dagger v^*\bigr]
\\&\phantom{{}={}}{}
\cdot\bigl(G_{0,1}-G_{0,1}vm_0^\dagger  v^*G_{0,1}\bigr)
\bigl[I-vM_0^\dagger v^*\bigl(I- G_{0,1}vm_0^\dagger v^*\bigr) \bigl(I-G_{0,2}vq_0^{\dagger}v^*\bigr)G_{0,0}\bigr]
\\&\phantom{{}={}}{}
-G_{0,0}\bigl[\bigl(I-vq_0^\dagger G_{0,2}\bigr)vm_0^\dagger v^*\bigl(I-G_{0,2}vq_0^{\dagger}v^*\bigr)G_{0,2}+vq_0^\dagger v^*G_{0,3}\bigr]
\\&\phantom{{}={}+{}}{}
\cdot
vm_0^\dagger v^*
\bigl[G_{0,2}\bigl(I-vq_0^{\dagger}v^*G_{0,2}\bigr)vm_0^\dagger v^* \bigl(I-G_{0,2}vq_0^\dagger v^*\bigr)+G_{0,3}vq_0^\dagger v^*\bigr]G_{0,0}
\\&\phantom{{}={}}{}
+G_{0,0}\bigl(I-vq_0^\dagger G_{0,2}\bigr)vm_0^\dagger v^*\bigl(I-G_{0,2}vq_0^{\dagger}v^*\bigr)
\\&\phantom{{}={}+{}}{}
\cdot
G_{0,3}\bigl(I-vq_0^{\dagger}v^*G_{0,2}\bigr)vm_0^\dagger v^* \bigl(I-G_{0,2}vq_0^\dagger v^*\bigr)G_{0,0}
\\&\phantom{{}={}}{}
+G_{0,0}vq_0^\dagger v^*G_{0,5}vq_0^\dagger v^*G_{0,0}
\\&\phantom{{}={}}{}
+G_{0,0}vq_0^{\dagger}v^*G_{0,4} \bigl(I-vq_0^\dagger v^*G_{0,2}\bigr)vm_0^\dagger v^*\bigl(I-G_{0,2}vq_0^\dagger v^*\bigr)G_{0,0}
\\&\phantom{{}={}}{}
+G_{0,0}\bigl(I-vq_0^\dagger v^*G_{0,2}\bigr)vm_0^\dagger v^* \bigl(I-G_{0,2}vq_0^\dagger v^*\bigr)G_{0,4}vq_0^{\dagger}v^*G_{0,0}
\\&\phantom{{}={}}{}
-G_{0,0}\bigl(I-vq_0^{\dagger}v^*G_{0,2}\bigr)vm_0^\dagger v^*
\bigl(I- G_{0,2}vq_0^{\dagger}v^*\bigr)G_{0,2}\bigl(I-vm_0^\dagger  v^*G_{0,1}\bigr)
\\&\phantom{{}={}-{}}{}\cdot
\bigl[I-vM_0^\dagger v^*\bigl(I-G_{0,1}vm_0^\dagger v^*\bigr)\bigl(I-G_{0,2}vq_0^{\dagger}v^*\bigr)G_{0,0}\bigr]
\\&\phantom{{}={}}{}
-\bigl[I-G_{0,0}\bigl(I-vq_0^{\dagger}v^*G_{0,2}\bigr)\bigl(I-vm_0^\dagger v^*G_{0,1} \bigr)vM_0^\dagger v^*\bigr]
\\&\phantom{{}={}+{}}{}
\cdot
\bigl(I-G_{0,1}vm_0^\dagger v^* \bigr)G_{0,2}\bigl(I-vq_0^{\dagger}v^*G_{0,2}\bigr)vm_0^\dagger v^*\bigl(I- G_{0,2}vq_0^{\dagger}v^*\bigr)G_{0,0}
\\&\phantom{{}={}}{}
-G_{0,0}vq_0^{\dagger}v^*G_{0,3}\bigl(I-vm_0^\dagger v^*G_{0,1}\bigr)
\\&\phantom{{}={}+{}}{}
\cdot
\bigl[I-vM_0^\dagger v^*\bigl(I-G_{0,1}vm_0^\dagger v^*\bigr)\bigl(I-G_{0,2}vq_0^\dagger v^*\bigr) G_{0,0}\bigr]
\\&\phantom{{}={}}{}
-\bigl[I-G_{0,0}\bigl(I-G_{0,0}vq_0^\dagger v^*G_{0,2}\bigr)\bigl(I-vm_0^\dagger v^*G_{0,1}\bigr)vM_0^\dagger v^*\bigr]
\\&\phantom{{}={}+{}}{}
\cdot
\bigl(I-G_{0,1}vm_0^\dagger v^*\bigr)G_{0,3}vq_0^{\dagger}v^*G_{0,0}
.
\end{align*}
Use 
\begin{align}
\bigl(I-vq_0^\dagger v^*G_{0,2}\bigr)vm_0^\dagger v^* \bigl(I-G_{0,2}vq_0^\dagger v^*\bigr)
=-V\mathcal PV,\quad
vq_0^\dagger v^*
=-V\mathsf PV,
\label{170828}
\end{align}
and then it follows that 
\begin{align*}
G_1
&
=
\bigl[I-G_{0,0}\bigl(I+V\mathsf PVG_{0,2}\bigr)\bigl(I-vm_0^\dagger  v^*G_{0,1}\bigr)vM_0^\dagger v^*\bigr]
\\&\phantom{{}={}}{}
\cdot\bigl(G_{0,1}-G_{0,1}vm_0^\dagger  v^*G_{0,1}\bigr)
\bigl[I-vM_0^\dagger v^*\bigl(I- G_{0,1}vm_0^\dagger v^*\bigr) \bigl(I+G_{0,2}V\mathsf PV\bigr)G_{0,0}\bigr]
\\&\phantom{{}={}}{}
-\bigl(\mathcal PVG_{0,2}+\mathsf PVG_{0,3}\bigr)vm_0^\dagger v^*\bigl(G_{0,2}V\mathcal P+G_{0,3}V\mathsf P\bigr)
\\&\phantom{{}={}}{}
+\mathcal PVG_{0,3}V\mathcal P
+\mathsf PVG_{0,5}V\mathsf P
+\mathsf PVG_{0,4} V\mathcal P
+\mathcal PVG_{0,4}V\mathsf P
\\&\phantom{{}={}}{}
-\bigl(\mathcal PVG_{0,2}+\mathsf PVG_{0,3}\bigr)\bigl(I-vm_0^\dagger  v^*G_{0,1}\bigr)
\\&\phantom{{}={}-{}}{}
\cdot\bigl[I-vM_0^\dagger v^*\bigl(I-G_{0,1}vm_0^\dagger v^*\bigr)\bigl(I-G_{0,2}vq_0^{\dagger}v^*\bigr)G_{0,0}\bigr]
\\&\phantom{{}={}}{}
-\bigl[I-G_{0,0}\bigl(I-vq_0^{\dagger}v^*G_{0,2}\bigr)\bigl(I-vm_0^\dagger v^*G_{0,1} \bigr)vM_0^\dagger v^*\bigr]
\\&\phantom{{}={}-{}}{}
\cdot\bigl(I-G_{0,1}vm_0^\dagger v^* \bigr)\bigl(G_{0,2}V\mathcal P+G_{0,3}V\mathsf P\bigr)
\\&
=
(I-\mathsf P)\bigl[I-G_{0,0}\bigl(I-vm_0^\dagger  v^*G_{0,1}\bigr)vM_0^\dagger v^*\bigr]
\\&\phantom{{}={}}{}
\cdot\bigl(G_{0,1}-G_{0,1}vm_0^\dagger  v^*G_{0,1}\bigr)
\bigl[I-vM_0^\dagger v^*\bigl(I- G_{0,1}vm_0^\dagger v^*\bigr) G_{0,0}\bigr](I-\mathsf P)
\\&\phantom{{}={}}{}
+\mathcal PVG_{0,3}V\mathcal P
-\mathcal PVG_{0,2}vm_0^\dagger v^*G_{0,2}V\mathcal P
\\&\phantom{{}={}}{}
-(I-\mathsf P)\bigl[I-G_{0,0}\bigl(I-vm_0^\dagger v^*G_{0,1} \bigr)vM_0^\dagger v^*\bigr]\bigl(I-G_{0,1}vm_0^\dagger v^* \bigr)G_{0,2}V\mathcal P
\\&\phantom{{}={}}{}
-\mathcal PVG_{0,2}\bigl(I-vm_0^\dagger  v^*G_{0,1}\bigr)
\bigl[I-vM_0^\dagger v^*\bigl(I-G_{0,1}vm_0^\dagger v^*\bigr)G_{0,0}\bigr](I-\mathsf P)
.
\end{align*}

We finally set 
\begin{align*}
T
&=
(I-\mathsf P)\bigl[I-G_{0,0}\bigl(I-vm_0^\dagger  v^*G_{0,1}\bigr)vM_0^\dagger v^*\bigr]
\\&\phantom{{}={}}{}
\cdot\bigl(G_{0,1}-G_{0,1}vm_0^\dagger  v^*G_{0,1}\bigr)
\bigl[I-vM_0^\dagger v^*\bigl(I- G_{0,1}vm_0^\dagger v^*\bigr) G_{0,0}\bigr](I-\mathsf P).
\end{align*}
For that we basically follows the lines of 
the proof of Theorem~\ref{17082015}.
Take an orthonormal resonance basis $\{\Psi_\gamma\}\subset\mathcal E$, 
and let $\mathbf n_\gamma$ and $\mathbf n_{\widetilde\gamma}$ be defined 
as in \eqref{1708202319}.
Then, due to \eqref{170828}, we have
\begin{align*}
G_{0,1}-G_{0,1}vm_0^\dagger  v^*G_{0,1}
=-\sum_{\alpha=1}^n|\mathbf n^{(\alpha)}\rangle \langle\mathbf n^{(\alpha)}|
+\sum_{\gamma}|\mathbf n_\gamma\rangle \langle \mathbf n_\gamma|
=-\sum_{\widetilde\gamma}|\mathbf n_{\widetilde\gamma}\rangle \langle \mathbf n_{\widetilde\gamma}|
,
\end{align*}
so that we can write 
\begin{align*}
T
=-\sum_{\widetilde\gamma}|\Psi_{\widetilde\gamma}\rangle \langle \Psi_{\widetilde\gamma}|
\end{align*}
with
\begin{align*}
\Psi_{\widetilde\gamma}
=
(I-\mathsf P)\bigl[I-G_{0,0}\bigl(I-vm_0^\dagger  v^*G_{0,1}\bigr)vM_0^\dagger v^*\bigr]\mathbf n_{\widetilde\gamma}.
\end{align*}
Since 
\begin{align*}
Qv^*\bigl(G_{0,1}-G_{0,1}vm_0^\dagger  v^*G_{0,1}\bigr)
&=
Q\bigl(I-v^*G_{0,1}vm_0^\dagger \bigr)v^*G_{0,1}
\\&
=\bigl(Q-(Q-S)\bigr)v^*G_{0,1}
\\&
=0,
\end{align*}
we have $Qv^*\mathbf n_{\widetilde\gamma}=0$.
We also note that $HG_{0,0}vm_0^\dagger =0$.
Then it follows that 
\begin{align*}
H\Psi_{\widetilde\gamma}
=vUQv^*\mathbf n_{\widetilde\gamma}
=0,
\end{align*}
and this implies that $\Psi_{\widetilde\gamma}\in \widetilde{\mathcal E}$.
On the other hand, by Lemma~\ref{12.11.24.18.24} we have 
\begin{align*}
\Psi_{\widetilde\gamma}
-
\mathbf n_{\widetilde\gamma}
\in \mathbb C\mathbf 1\oplus \mathcal L^{\beta-2}.
\end{align*}
The basis $\{\Psi_{\widetilde\gamma}\}\subset \widetilde{\mathcal E}$ is obviously orthogonal to $\mathsf E$. 
Hence  $\{\Psi_{\widetilde\gamma}\}\subset \widetilde{\mathcal E}$ 
forms an orthonormal non-resonance basis. We are done.
\end{proof}

\section{Modification of \cite[Lemma~4.16]{IJ1}}\label{17081223}

In this appendix we present a modified version of \cite[Lemma~4.16]{IJ1},
which was needed in the proof of Lemma~\ref{lemma24}. 
Here we discuss only on $\mathbb Z$, not on the graph $G$,
and the notation is independent of the other parts of the paper.
We rather follow the convention of \cite{IJ1}.

We define $\mathbf 1\colon\mathbb Z\to\mathbb C$,
$\mathbf n\colon\mathbb Z\to\mathbb C$ 
and $\mathbf n^2\colon\mathbb Z\to\mathbb C$ by 
\begin{align*}
\mathbf 1[m]=1,\quad
\mathbf n[m]=m
\quad\text{and}\quad  
\mathbf n^2[m]=m^2
\end{align*}
for $m\in\mathbb Z$, respectively.
We also define the operators $G_0^0$ and $G_2^0$ by the operator kernels
\begin{align*}
G_0^0[n,m]=-\tfrac12|n-m|
\quad\text{and}\quad
G_2^0[n,m]=-\tfrac1{12}|n-m|^3+\tfrac1{12}|n-m|,
\end{align*}
respectively.

\begin{lemma}\label{12.12.27.15.3}
Let $u_1,u_2\in \ell^{1,4}(\mathbb Z)$ satisfy
\begin{align}
\langle \mathbf 1, u_1\rangle=\langle \mathbf n,u_1\rangle 
=\langle \mathbf n^2,u_1\rangle =\langle\mathbf 1,u_2\rangle=0.
\label{12.12.27.7.13}
\end{align}
Then $G_0^0u_1\in \ell^{1,2}(\mathbb Z)$, $G_0^0u_2\in\mathbb C\mathbf 1\oplus\ell^{1,2}(\mathbb Z)$, and
\begin{align}
\langle u_2,G_2^0u_1\rangle =-\langle G_0^0u_2,G^0_0u_1\rangle.
\label{13.3.8.14.53}
\end{align}
\end{lemma}
\begin{proof}
We can prove the assertion along the proof of \cite[Lemma~4.16]{IJ1}.
Instead of \cite[(4.37)]{IJ1}, we can make use of the values
\begin{align*}
\hat u_1(0)=\hat u_1'(0)
=\hat u_1''(0)=\hat u_2(0)=0.
\end{align*}
Then modifications of the proof are not difficult.
We omit the detail.
\end{proof}

\bigskip
\noindent
\subsubsection*{Acknowledgements} 
KI was supported by JSPS KAKENHI Grant Number 17K05325.
The authors were partially supported by the Danish Council for Independent Research $|$ Natural Sciences, Grant DFF--4181-00042. 
The authors would like to thank the referee for helpful remarks on the manuscript.

\end{document}